\documentclass[11pt]{amsart} 

\usepackage[english]{babel}
\usepackage[latin1]{inputenc} 
\usepackage{amsmath}
\usepackage{amsthm}
\usepackage{amssymb}
\usepackage{tikz}
\usepackage{bbm}
\usepackage{adjustbox}
\usepackage{euscript}
\usepackage[all]{xy}
\usepackage{tensor}
\usepackage{relsize}
\usepackage[bbgreekl]{mathbbol}
\DeclareSymbolFontAlphabet{\mathbb}{AMSb} 
\DeclareSymbolFontAlphabet{\mathbbl}{bbold}

\usepackage{mathrsfs}
\usepackage{imakeidx}
\usepackage{appendix}
\usepackage{tikz-cd}
\usepackage{verbatim}
\usepackage{enumitem}
\usepackage{multicol}

\usepackage[backref=page]{hyperref}

\usepackage{cleveref}

\usepackage{mathtools}

\hypersetup{colorlinks=true,linkcolor=blue,anchorcolor=blue,citecolor=blue}

\numberwithin{equation}{section} 
\newtheorem{thm}[equation]{Theorem}

\newtheorem{prop}[equation]{Proposition}
\newtheorem{lemma}[equation]{Lemma}
\newtheorem{cor}[equation]{Corollary}

\theoremstyle{definition}

\newtheorem{defn}[equation]{Definition}
\newtheorem{constr}[equation]{Construction}
\newtheorem{example}[equation]{Example}
\theoremstyle{remark}
\newtheorem{rmk}[equation]{Remark}

\newcommand{\F}{\mathbb F}
\newcommand{\Z}{\mathbb Z}
\newcommand{\Spec}{\operatorname{Spec}}
\newcommand{\G}{\mathbb G}
\renewcommand{\P}{\mathbb P}
\newcommand{\C}{\mathbb C}

\newcommand{\mc}[1]{\mathcal{#1}}
\newcommand{\cl}{\overline}
\newcommand{\set}[1]{\{#1\}}
\newcommand{\on}[1]{\operatorname{#1}}
\newcommand{\ang}[1]{\left \langle{#1}\right \rangle}
\DeclareMathAlphabet{\mathpzc}{OT1}{pzc}{m}{it}

\newcommand{\sm}{{\mathrm{sm}}}

\DeclareMathOperator{\Aff}{Aff}
\DeclareMathOperator{\Tot}{\mathrm{Tot}}
\newcommand{\op}{{\mathrm{op}}}
\newcommand{\ra}{{\rightarrow}}
\newcommand{\Mod}{\mathrm{Mod}}
\newcommand{\DMod}[1]{D(\Mod_{#1})}
\newcommand{\Hdg}{{\mathrm{H}}}
\newcommand{\RG}{R\Gamma}
\newcommand{\mbb}{\mathbb}
\newcommand{\mf}{\mathfrak}
\newcommand{\dR}{{\mathrm{dR}}}

\newcommand{\tr}{{\mathrm{tr}}}
\newcommand{\mr}{\mathrm}
\newcommand{\gr}{{\mathrm{gr}}}
\newcommand{\tto}{\xymatrix{\ar[r]&}}

\newcommand{\ism}{\xymatrix{\ar[r]^\sim&}}

\title[Eilenberg-Moore spectral sequence and Hodge cohomology]{Eilenberg-Moore spectral sequence and Hodge cohomology of classifying stacks}

\author{Dmitry Kubrak}

\address{Max Planck Institute\\
	Vivatgasse 7\\
	Bonn, 53111\\ Germany}

\email{dmkubrak@gmail.com} 

\author{Federico Scavia}

\address{Department of Mathematics\\
	University of California\\
	Los Angeles, CA 90095-1555 \\ USA}

\email{scavia@math.ucla.edu}

\subjclass[2020]{14L30, 20G15, 13A18, 18G40, 57T35}   

\begin{document}

	\begin{abstract}
		Let $G$ be a smooth connected reductive group over a field $k$ and $\Gamma$ be a central subgroup of $G$. We construct Eilenberg-Moore-type spectral sequences converging to the Hodge and de Rham cohomology of $B(G/\Gamma)$. As an application, building upon 
  work of Toda 
  and using Totaro's inequality, we show that for all $m\geq 0$ the Hodge and de Rham cohomology algebras of the classifying stacks $B\mr{PGL}_{4m+2}$ and $B\mr{PSO}_{4m+2}$ over $\F_2$ are isomorphic to the singular  $\F_2$-cohomology of the classifying space of the corresponding Lie group. From this we obtain a full description of 
  $H^{>0}(\mr{GL}_{4m+2}, \on{Sym}^j(\mf{pgl}_{4m+2}^\vee))$ and $H^{>0}(\mr{SO}_{4m+2}, \on{Sym}^j(\mf{pso}_{4m+2}^\vee))$ over $\mbb F_2$.
	\end{abstract}

 	\maketitle

{
  \hypersetup{linkcolor=black}
  \tableofcontents
}
	
	\section{Introduction}
	
	Let $p$ be a prime number, and let $G$ be a split reductive group over $\Z$. We denote by $BG$ the  classifying stack of $G$, and by $BG(\C)$ the classifying space of the topological Lie group $G(\C)$. The computation of the mod $p$ singular cohomology ring $H^*_{\on{sing}}(BG(\C);\F_p)$, or equivalently the determination of mod $p$ characteristic classes of principal $G(\C)$-bundles, is one of the most classical problems in algebraic topology, with contributions from a long list of illustrious authors. 
	
	Recently, in \cite{totaro2018hodge}, B. Totaro initiated the study of Hodge cohomology $H^*_{\on{H}}(BG/\F_p)$ and de Rham cohomology $H^*_{\on{dR}}(BG/\F_p)$ of the classifying stack $BG_{\F_p}$. Similarly to the topological situation, one can think of elements of these rings as Hodge and de Rham characteristic classes for $G_{\mbb F_p}$-torsors. However, as Totaro showed, $H^*_{\on{H}}(BG/\F_p)$ also has a purely representation-theoretic interpretation in terms of rational cohomology of the algebraic group $G_{\mbb F_p}$ with coefficients $\on{Sym}^i\mathfrak g^\vee$, where $\mathfrak g$ is the adjoint representation of $G$. In \cite[Theorem 9.2]{totaro2018hodge}, he established a general result, stating that if $p$ is not a torsion prime\footnote{A prime $p$ is called torsion if there is non-trivial $p$-torsion $H^*_{\on{sing}}(G(\C);\Z)$. For any given $G$ there are only finitely many torsion primes, and there also is a simple recipe to find them all (see \cite[Example 6.1.5]{KubrakPrikhodko_HdR}).} for $G$ then Hodge and de Rham cohomology of $BG_{\F_p}$ are in fact isomorphic to the mod $p$ singular cohomology of $BG(\C)$. 	The subtlety of the situation, however, is that there is no natural map between Hodge (or de Rham) and singular cohomology: the above isomorphisms are constructed by explicitly computing and comparing the two sides.
	
	Totaro also investigated what happens at torsion primes in some particular examples. For $p=2$ and $G=\mr{SO}_n$ he constructed isomorphisms of graded rings
	\begin{equation}\label{totaro-so}H^*_{\on{H}}(B\mr{SO}_n/\F_2)\simeq H^*_{\on{dR}}(B\mr{SO}_n/\F_2)\simeq H^*_{\on{sing}}(B\mr{SO}_n(\C);\F_2).\end{equation}
  On the other hand, he computed that 
	\[\dim_{\F_2} H^{32}_{\on{dR}}(B\mr{Spin}_{11}/\F_2)> \dim_{\F_2} H^{32}_{\on{sing}}(B\mr{Spin}_{11}(\C);\F_2),\]
	showing that Hodge and de Rham cohomology of $BG_{\F_p}$ are not isomorphic to mod $p$ singular cohomology of $BG(\C)$ in general, even as graded vector spaces. Some further calculations of $H^*_{\on{H}}(BG/\F_2)$ and $H^*_{\on{dR}}(BG/\F_2)$ have 
 been performed 
 by E. Primozic \cite{primozic2019computations} for $G=G_2$ and $G=\on{Spin}_n$ for $n\leq 11$. 
	
	A general statement which holds even for torsion primes is the inequality of dimensions. First of all, the existence of the Hodge-to-de Rham spectral sequence implies the inequality $$\dim_{\mbb F_p} H^i_\Hdg(BG/\mbb F_p) \ge \dim_{\mbb F_p} H^i_\dR(BG/\mbb F_p).$$ Moreover, as conjectured by Totaro and recently proved by A.~Prikhodko and the first author in \cite{kubrak2021p-adic}, one also has an inequality
	$$\dim_{\mbb F_p} H^i_\dR(BG/\mbb F_p)\ge \dim_{\mbb F_p} H^i_{\mr{sing}}(BG(\C);\mbb F_p).$$ Therefore
	\begin{equation}\label{eq:inequalities}
		\dim_{\mbb F_p} H^i_\Hdg(BG/\mbb F_p) \ge   \dim_{\mbb F_p} H^i_\dR(BG/\mbb F_p)\ge \dim_{\mbb F_p} H^i_{\mr{sing}}(BG(\C);\mbb F_p).
	\end{equation}
	We will refer to (\ref{eq:inequalities}) as Totaro's inequality.
	
	\subsection*{Main results}
	
	The computations of Totaro and Primozic are based on a version of the Hoch\-schild--Serre spectral sequence in Hodge cohomology (see \cite[Proposition 9.3]{totaro2018hodge}). In this paper we attempt to compute the Hodge and de Rham cohomology of the classifying stacks of classical simple adjoint groups  $\mr{PGL}_n$, $\mr{PSp}_n$, $\mr{PSO}_n$ over $\F_2$ by using the \textit{Eilenberg-Moore} spectral sequence instead. The Hodge cohomology in these situations can also be reinterpreted in terms of cohomology of the classical groups $\mr{GL}_n$, $\mr{Sp}_n$, $\mr{SO}_n$, but with coefficients in  modules that are slightly more complicated than $\mr{Sym}^j \mf g^\vee$ (see \Cref{sec-reptheory} for more details).
	
	Let us describe our setup. Let $G$ be a split connected reductive group over a field $k$, let $\Gamma\subset G$ be a central subgroup (so $\Gamma$ is of multiplicative type) and consider the quotient $\overline G\coloneqq G/\Gamma$. For example, we could take $G=\mr{GL}_n$ and $\Gamma=\mbb G_m$, in which case we get $\overline{G}=\mr{PGL}_n$. The multiplication map $\Gamma \times G\to G$ defines an action $B\Gamma\times BG\to BG$ of the group stack $B\Gamma$ and so induces a coaction of the Hopf algebra $H^*_\Hdg(B\Gamma/k)$ on $H^*_\Hdg(BG/k)$. We also get a similar structure for de Rham cohomology. This coaction can be used to give a first approximation to $H^*_\Hdg(B\overline{G}/k)$, as our first general result shows.
	\begin{thm}[Eilenberg-Moore spectral sequence]\label{eilenberg-moore-easy}
		Let $k$ be a field, and consider a short exact sequence of linear algebraic $k$-groups 
		\begin{equation*}\label{central-seq}
			1\to \Gamma\to G\to \cl{G}\to 1.
		\end{equation*}
		where $G$ and $\cl{G}$ are smooth and $\Gamma$ is a central subgroup of multiplicative type. Then we have two (cohomological) first-quadrant convergent spectral sequences
		\begin{align*}
			E_2^{i,j}\coloneqq& \left(\on{Cotor}^i_{H^*_{\on{H}}(B\Gamma/k)}(k, H^{*}_{\on{H}}(BG/k))\right)^j\Rightarrow H^{i+j}_{\on{H}}(B\cl{G}/k), \\
			E_2^{i,j}\coloneqq& \left(\on{Cotor}^i_{H^*_{\on{dR}}(B\Gamma/k)}(k, H^*_{\on{dR}}(BG/k))\right)^j\Rightarrow H^{i+j}_{\on{dR}}(B\cl{G}/k).
		\end{align*}
	\end{thm}
	\noindent Here $k$ is the trivial comodule over $H^*_{\on{H}}(BG/k)$ (and $H^*_{\dR}(BG/k)$) and $\on{Cotor}^i$ are the derived functors of cotensor product (the definition is essentially dual to $\on{Tor}^i$ in the algebra setting, see \Cref{sec:cotensor products} for a reminder). 
	
	\begin{rmk}\label{rem: intro bigrading on the cotor} Recall that Hodge cohomology is in fact a bigraded algebra (see \Cref{ssec:prelim}  for a short reminder). By our construction, in the Hodge setting, the Eilenberg-Moore spectral sequence $E_r^{i,j}$ splits as a direct sum $E_r^{i,j}\simeq \oplus_{h} E_r^{i,j}$ where
	$$
	(E_2^{i,j})^h\coloneqq  \left(\on{Cotor}^i_{H^{*,*}_{\on{H}}(B\Gamma/k)}(k, H^{*,*}_{\on{H}}(BG/k))\right)^{h,j}\Rightarrow H^{h,i+j}_{\on{H}}(B\cl{G}/k).
	$$
	\end{rmk}

	A spectral sequence analogous to those of \Cref{eilenberg-moore-easy} for singular cohomology was used by H. Toda in \cite{toda1987cohomology} to compute the $\mbb F_2$-singular cohomology of $B\mr{PGL}_n(\mbb C)$, $B\mr{PSp}_n(\mbb C)$ and $B\mr{PSO}_n(\mbb C)$ when $n=4m+2$. The main result of our paper is that the answer for Hodge and de Rham cohomology over $\mbb F_2$ stays essentially the same. (When $n$ is odd, $2$ is not a torsion prime for any of these groups, and so Hodge and de Rham cohomology are isomorphic to singular cohomology by the aforementioned result of Totaro.)
	
	\begin{thm}\label{mainthm}
		(1) Let $\overline G$ be either $\mr{PSO}_{4m+2}$ or $\mr{PGL}_{4m+2}$. Then we have isomorphisms of graded rings 
			\[H^*_{\on{H}}(B\overline G/\F_2)\simeq H^*_{\on{dR}}(B\overline G/\F_2)\simeq H^*_{\on{sing}}(B\overline G(\C);\F_2).\]
			
		(2) In the case $\overline G=\mr{PSp}_{4m+2}$ we have an isomorphism of graded vector spaces 
			\[
			H^*_{\on{H}}(B\overline G/\F_2)\simeq H^*_{\on{dR}}(B\overline G/\F_2)\simeq H^*_{\on{sing}}(B\overline G(\C);\F_2).
			\]		
	\end{thm}
	\noindent  
 We don't know if there exists an algebra isomorphism in \Cref{mainthm}(2), the main reason being that the algebra structure on $H^*_{\on{sing}}(B\mr{PSp}_{4m+2}(\C);\F_2)$ is not fully understood (see e.g. \cite[Proposition 4.7]{toda1987cohomology}). In contrast, the algebra structures of $ H^*_{\on{sing}}(B\mr{PSO}_{4m+2}(\C);\F_2)$ and $H^*_{\on{sing}}(B\mr{PGL}_{4m+2}(\C);\F_2)$ can be described explicitly in terms of generators and relations; see \cite[Proposition 4.2, Proposition 4.5]{toda1987cohomology}.

	Let us sketch the main ideas which go into the proof of \Cref{mainthm}. With \Cref{eilenberg-moore-easy} at our disposal, one can try to make direct computations similar to the ones in Toda's work \cite{toda1987cohomology}. After some extra work, this is possible to achieve for $\mr{PGL}_{4m+2}$ and $\mr{PSp}_{4m+2}$. However, Toda's argument doesn't seem to go through directly for $\mr{PSO}_{4m+2}$ (\Cref{rem:Toda's strategy doesn't work}). Our key observation is that if we assume the results of \cite{toda1987cohomology} as given, there is an easier way: some parts of \Cref{mainthm} are implied by Totaro's inequality almost for free, while with a little more computational work one can also get the rest of \Cref{mainthm}, including the most complicated case of $\mr{PSO}_{4m+2}$. This motivated to split the proof of \Cref{mainthm} in two parts. 
	
	\textit{Part 1. Isomorphisms as graded vector spaces for $B\mr{PGL}_{4m+2}$ and $B\mr{PSp}_{4m+2}$.} The main step in Toda's computation of $H^*_{\on{sing}}(B\overline G(\C);\F_2)$ consists in showing that the Eilenberg-Moore spectral sequence degenerates at the second page. Assuming Toda's result, it is enough to identify the second sheets of the Eilenberg-Moore spectral sequences for Hodge and singular cohomology: indeed, by Totaro's inequality (\ref{eq:inequalities}) this would immediately imply the degeneration in the Hodge setting and then also give an equality of dimensions of cohomology. This identification is done by explicitly comparing the comodule structures on $H^*_\Hdg(B\mr{GL}_{4m+2}/\F_2)$ and $H^*_\Hdg(B\mr{Sp}_{4m+2}/\F_2)$ with the ones for singular cohomology (having identified $H^*_\Hdg(B\mbb G_{\mr{m}}/\F_2)$ with $H^*_{\mr{sing}}(B\mbb C^\times,\mbb F_2)$ as Hopf algebras and $H^*_\Hdg(B\mu_2/\mbb F_2)$ with $H^*_{\mr{sing}}(B\mbb  Z/2,\mbb F_2)$ as coalgebras); see \Cref{sec:GL SP graded vector space}.
	
	\textit{Part 2. Isomorphisms as graded rings for $B\mr{PGL}_{4m+2}$ and $B\mr{PSO}_{4m+2}$.}
	For $B\mr{PGL}_{4m+2}$, the spectral sequence argument above already produces a \textit{ring} isomorphism between the Hodge and singular cohomology, but only after passing to the associated graded. To lift this to an isomorphism between the original rings we imitate the computation of Toda in the Hodge setting (see Sections \ref{sec: Cotor} and \ref{sec: PGL}). The case of $B\mr{PSO}_{4m+2}$ is more difficult, as the comodule structure on $H^*_\Hdg(B\mr{SO}_{4m+2}/\F_2)$ is \textit{not} compatible with the one on singular cohomology. Nevertheless, we get around this by directly replacing some results on the structure of the comodule $H^*_{\mr{sing}}(B\mr{SO}_{4m+2}(\mbb C),\F_2)$ with suitable Hodge cohomology analogues (see \Cref{sec: Cotor}). In some ways, the Hodge context actually turns out to be easier for the computation via Toda's method; see \Cref{easier}. After computing $\on{Cotor}$ via an explicit resolution (which is set up in \Cref{compute-cotor}) we deduce the degeneration of the Eilenberg-Moore spectral sequence in the $\mr{PSO}_{4m+2}$-case by comparing to topological side and using Totaro's inequality. To conclude, we identify the resulting descriptions of Hodge and singular cohomology in terms of generators and relations. For more details see \Cref{sec: PSO}.

	Even though the algebras $H^*_\Hdg(B\mr{SO}_{4m+2}(\mbb C)/\F_2)$ and $H^*_{\mr{sing}}(B\mr{SO}_{4m+2}(\mbb C),\F_2)$ are abstractly isomorphic, there is a rather subtle implicit distinction between topological and Hodge settings. Namely, the square
	$$\xymatrix{
		H^*_{\on{H}}(B\mr{PSO}_{4m+2}/\F_2) \ar[rr]_(.465){\text{Section 7}}^(.465)\sim \ar[d] & & H^*_{\on{sing}}(B\mr{PSO}_{4m+2}(\C),\F_2) \ar[d] \\
		H^*_{\on{H}}(B\mr{SO}_{4m+2}/\F_2)  \ar[rr]_(.465){\text{Totaro}}^(.465)\sim & & H^*_{\on{sing}}(B\mr{SO}_{4m+2}(\C),\F_2), 
	}
	$$
	induced by pull-back with respect to the map $B\mr{SO}_{4m+2}\to B\mr{PSO}_{4m+2}$ and isomorphism (\ref{totaro-so}) is \textit{not} commutative. 
	
	\subsection*{Applications to representation theory}
	Recall that Hodge cohomology comes with a natural bigrading: $H^n_\Hdg(X/k)\simeq \oplus_{i+j=n} H^{i,j}(X/k)$, where $H^{i,j}(X/k)\simeq H^j(X,\Omega^i)$. In \cite[Theorem 2.4]{totaro2018hodge}, Totaro showed that if $G$ is a smooth affine $k$-group, one has the following representation-theoretic formula for $H^{i,j}(BG/k)$:
	$$
	H^{i,j}(BG/k)\simeq H^{i-j}(G,\on{Sym}^i \mathfrak g^\vee).
	$$
	The right hand side denotes the cohomology of $G$ as an algebraic group (sometimes also called ``rational cohomology''), and the $G$-action on $\on{Sym}^i \mathfrak g^\vee$ is the natural adjoint action. This gives a geometric interpretation of the cohomology of representations like $\on{Sym}^i \mathfrak g^\vee$, which Totaro used to make new computations. 
If $\Gamma\subset G$ is a central subgroup and $\cl{G}\coloneqq G/\Gamma$, the 
Hodge cohomology of $B\overline{G}$ 
can also be interpreted in terms of rational cohomology of $G$, but with coefficients in more complicated modules, namely $\on{Sym}^i \cl{ \mf g}^\vee$, where $\cl{\mf g}\coloneqq \mr{Lie}(\cl G)$. 
	

 \begin{rmk}
In order to compute the groups $H^{j}(G,\on{Sym}^i \mathfrak g^\vee)$ by the above method, it is necessary to describe Hodge cohomology of $B\cl G$ as a \textit{bigraded} algebra. Given the degeneration of the Eilenberg-Moore spectral sequence, this reduces to understanding the bigraded components of the cotorsion groups $\on{Cotor}^i$ from \Cref{eilenberg-moore-easy}. In order to keep track of the bigrading, we use the explicit resolution of the trivial comodule given by twisted tensor product (see \Cref{constr:twisted tensor product} and \Cref{compute-cotor}).
 \end{rmk}

	For brevity, let us only discuss the result of the computation in the case $\cl{G}=\mr{PGL}_{n}$ here, and refer the reader to \Cref{sec-reptheory} for the remaining cases. For every $n\geq 1$ we have a short exact sequence of $\mr{GL}_n$-modules $0 \to \F_2 \to \mathfrak{gl}_n \to \mathfrak{pgl}_n \to 0$, which is non-split if and only if $n$ is even. From the Hochschild-Serre spectral sequence in rational cohomology one can see that $H^{*}(\mr{GL}_n,\on{Sym}^i \mathfrak{pgl}_n^\vee)\simeq H^{*}(\mr{PGL}_n,\on{Sym}^i \mathfrak{pgl}_n^\vee)$, and, thus 
	$$
	H^{i,j}(B\mr{PGL}_{n}/\F_2)\simeq H^{j-i}(\mr{GL}_n,\on{Sym}^i \mathfrak{pgl}_n^\vee).
	$$

	From \Cref{thm-hodge-pgl}, giving the description of the left hand side in the case $n=4m+2$, we get a full computation of higher cohomology (over $\mbb F_2$) of $\mr{GL}_{4m+2}$ with coefficients in $\on{Sym}^i \mathfrak{pgl}_n^\vee$. To be more precise, there is a certain class $z\in H^{1}(\mr{GL}_{4m+2}, \mathfrak{pgl}_{4m+2}^\vee)$ and a polynomial subalgebra $A\coloneqq \mbb F_2[c_1,b_i]_{1 \le i\le 2m+1}\subset (\oplus_i \on{Sym}^i \mathfrak{pgl}_{4m+2}^\vee)^{\mr{GL}_{4m+2}}$ in the $\mr{GL}_{4m+2}$-invariants such that for any $j>0$ one has 
	$$
	H^{j}(\mr{GL}_{4m+2}, \oplus_i\on{Sym}^i \mathfrak{pgl}_{4m+2}^\vee)\ism A\cdot z^j
	$$ is a free $A$-module of rank 1 generated by $z^j\in H^{j}(\mr{GL}_{4m+2}, \on{Sym}^j\mathfrak{pgl}_{4m+2}^\vee)$. Here $c_1$ has degree 1 and each $b_i$ has degree $4i$. In particular, for all $i,j\geq 0$, we get a formula for the dimension of 
	$H^{j}(\mr{GL}_{4m+2}, \on{Sym}^i \mathfrak{pgl}_{4m+2}^\vee)$ as the number of ways to write $i-j$ as a sum \[\gamma_1+4\beta_2+8\beta_3+\dots+(8m+4)\beta_{2m+1},\] where $\gamma_1$ and the $\beta_h$ are non-negative integers. In particular, $H^{j}(\mr{GL}_{4m+2}, \on{Sym}^i \mathfrak{pgl}_{4m+2}^\vee)\neq 0$ if and only if $i\ge j$.
	
\subsection*{Acknowledgements} The first-named author is grateful to Max Planck Institute for the excellent work conditions during his stay there while this work was being written. He would also like to thank Peter Scholze and Xing Gu for fruitful and helpful conversations.  The second-named author thanks K{\fontencoding{T1}\selectfont\k{e}}stutis \v{C}esnavi\v{c}ius and the Laboratoire de Math\'ema\-tiques d'Orsay (Universit\'e Paris-Saclay) for hospitality during Summer 2021, and the Institut des Hautes \'Etudes Scientifiques for hospitality in the Fall 2021.

	\section{Preliminaries}\label{prelim}
	\subsection{Hodge and de Rham cohomology of stacks}  \label{ssec:prelim} 
	Let $k$ be a field and $X$ be a smooth Artin stack of finite type over $k$. For every $i,j\geq 0$, we denote by $H^i(X,\Omega^j)$ the $i$-th cohomology of the sheaf $\Omega^j$ of $j$-differential forms on the big \'etale site of $X$ (see \cite[Section 2]{totaro2018hodge}).
	
	We denote by $H^{i,j}_\Hdg(X/k)\coloneqq H^j(X,\Omega^i)$ and $H^n_\Hdg(X/k)\coloneqq \oplus_{i+j=n} H^{i,j}_\Hdg(X/k)$ the $(i,j)$-th component and the total $n$-th Hodge cohomology group, respectively. The algebra $H^*_{\on{H}}(X/k)\coloneqq \oplus_{n=0}^\infty H^n_{\Hdg}(X/k)\simeq \oplus_{i,j\geq 0}H^{i,j}_\Hdg(X/k)$ has a natural bigraded $k$-algebra structure.

	We also denote by $H^*_{\on{dR}}(X/k)$ the de Rham cohomology of $X$ (which is a $\mbb Z$-graded $k$-algebra): it can be defined as the hypercohomology of the (de Rham) complex of sheaves $\Omega^*_\dR\coloneqq \Omega^0\to \Omega^1\to \Omega^2\to \ldots $ on big \'etale site of $X$. See \cite[\S 1]{totaro2018hodge} or \Cref{app:de rham cohomology} for more details.  
	
	The abutment filtration $F^i(\Omega^*_\dR)\coloneqq \Omega^{\ge i}_\dR\subset \Omega^*_\dR$ has the associated graded $\oplus_i \Omega^i[-i]$ and induces the Hodge-de Rham spectral sequence:
	$$
	E_1^{i,j}\coloneqq H^{i,j}_\Hdg(X/k)\Rightarrow H^{i+j}_\dR(X/k).
	$$
	
	\begin{rmk}[Another formula for Hodge cohomology]
	By flat descent for the cotangent complex, for any $j\ge 0$ one has a quasi-isomorphism $$\RG(X,\Omega^j)\ism \RG(X,\wedge^j\mbb L_{X/k}),$$ where $\mbb L_{X/k}\in \on{QCoh}(X)$ is the cotangent complex of $X$ (see e.g. \cite[Proposition 1.1.4]{KubrakPrikhodko_HdR}). If $X=BG$ for some smooth $k$-group scheme $G$ then, under the identification $\on{QCoh}(X)^+\simeq D(\on{Rep}(G))^+$, one has an equivalence $\mbb L_{X/k}\simeq \mf g^\vee[-1]$, where $\mf g^\vee$ is the coadjoint representation. This leads to an equivalence $\RG(BG,\Omega^j)\xrightarrow{\sim} \RG(G, \on{Sym}^j\mf g^\vee[-j])$, since $\wedge^j(\mf g^\vee[-1])\simeq (\on{Sym}^j\mf g^\vee)[-j]$ by the decalage isomorphism (see e.g. \cite[Proposition A.2.49]{kubrak2021p-adic}). In particular: \[H^{i,j}(BG)\simeq H^{j-i}(G,\on{Sym}^{i}\mf g^\vee).\]
	\end{rmk}

	\subsection{Cotensor product and Cotor} \label{sec:cotensor products}
	Our main reference on this material is \cite[Appendix A]{ravenel1986complex}. Let $k$ be a field, $L$ be an abelian group (in most examples $L$ will be $0$, $\Z$ or $\Z^2$) and let $\Lambda$ be an $L$-graded
 Hopf $k$-algebra. A left (resp. right) $\Lambda$-comodule is an $L$-graded $k$-vector space $M$ together with a $k$-linear graded map $\phi_M\colon M\to \Lambda\otimes_kM$ (resp. $\phi_M\colon M\to M\otimes_k\Lambda$) which is coassociative and counital. If $M$ is also a $k$-algebra and $\phi_M$ is a $k$-algebra homomorphism then $M$ is called a comodule algebra; see \cite[Definition A1.1.2]{ravenel1986complex}. By \cite[Theorem A1.1.3]{ravenel1986complex}, the category of $L$-graded $\Lambda$-comodules is abelian.

\begin{defn}\label{def:cotensor_product}	\begin{enumerate}
    \item Let $M$ be a right $\Lambda$-comodule and $N$ be a left $\Lambda$-comodule. The \textit{cotensor product} of $M$ and $N$ over $\Lambda$ is defined as the $L$-graded $k$-vector space
	\[M\square_{\Lambda}N\coloneqq \on{Ker}(M\otimes_kN\xrightarrow{\phi_M\otimes 1-1\otimes\phi_N}M\otimes_k\Lambda\otimes_kN);\]
	see \cite[Definition A1.1.4]{ravenel1986complex} for more details. 
	\item Given a left $\Lambda$-comodule $N$, the $\Lambda$-subcomodule $PN\subset N$ of \textit{primitive elements} is defined as 
	\[PN\coloneqq \set{n\in N: \phi_N(n)=1\otimes n}.\]
	\end{enumerate}
\end{defn}	

Note that the canonical isomorphism $k\otimes_kN\simeq N$ induces an isomorphism
	\[k\square_{\Lambda}N\simeq PN.\]
It is also not hard to see that if $N=A$ is a left $\Lambda$-comodule $k$-algebra, then $PA\subset A$ is a $k$-subalgebra.

The functor $M\square_{\Lambda}N$ is left exact in $N$. (Here we use that $k$ is a field.) Thus the following definition makes sense: 
\begin{defn}
    If $i\geq 0$ is an integer, we define the $i$-th \textit{cotorsion group} $\on{Cotor}^i_{\Lambda}(M,N)$ as the $i$-th right derived functor of $M\square_{\Lambda}N$, regarded as an additive functor of $N$ with values in $L$-graded vector spaces; see \cite[Definition A1.2.3]{ravenel1986complex}. We also let \[\on{Cotor}^*_\Lambda(M,N)\coloneqq \oplus_{i\geq 0}\on{Cotor}^i_\Lambda(M,N).\] There is a canonical isomorphism
	\[\on{Cotor}^0_\Lambda(M,N)\simeq M\square_{\Lambda}N.\]Each vector space $\on{Cotor}^i_\Lambda(M,N)$ comes with a natural $L$-grading. 
\end{defn}	

	
	Given two left comodules $M$ and $N$, we write $M\otimes_k N$ for the comodule tensor product of $M$ and $N$, see \cite[Definition A1.1.2]{ravenel1986complex}. The left coaction on $M\otimes_k N$ is given by the composition 
	$$
	\xymatrix{M\otimes_k N \ar[rr]^(.4){\phi_M\otimes \phi_N}&& \Lambda\otimes_k M\otimes_k \Lambda \otimes_k N \ar[r]^\sim & \Lambda\otimes_k \Lambda \otimes_k M\otimes_k N \ar[rr]^(.55){\mu_{\Lambda}\otimes \mr{id}_M \otimes \mr{id}_N} &&\Lambda \otimes_k M\otimes_k N},
	$$
	where $\mu_\Lambda\colon \Lambda\otimes_k\Lambda \to \Lambda$ is the multiplication map. Note that this operation depends crucially on the \textit{algebra} structure on $\Lambda$.
	
	If $M_1$ and $M_2$ are right $\Lambda$-comodules and $N_1$ and $N_2$ are left $\Lambda$-comodules, we have an external cup product map
	\[\on{Cotor}^{*_1}_\Lambda(M_1,N_1)\otimes_k \on{Cotor}^{*_2}_\Lambda(M_2,N_2)\tto \on{Cotor}^{*_1+*_2}_\Lambda(M_1\otimes_kM_2,N_1\otimes_kN_2);\]
	see 	\cite[Definition A1.2.13]{ravenel1986complex}. If $M$ is a left $\Lambda$-comodule algebra and $N$ is a right $\Lambda$-comodule algebra, then letting $M_1=M_2=M$ and $N_1=N_2=N$ and composing the external cup product with the map
	\[\on{Cotor}^*_\Lambda(M\otimes_kM,N\otimes_kN)\tto \on{Cotor}^*_\Lambda(M,N)\]
	induced by the multiplication maps $M\otimes_kM\to M$ and $N\otimes_kN\to N$ gives $\on{Cotor}^*_\Lambda(M,N)$ the structure of a $(\Z\oplus L)$-graded\footnote{Here $\on{Cotor}^i_\Lambda(M,N)$ has $\Z$-grading $i$ and $L$-grading is the one coming from $\Lambda$, $M$ and $N$.} $k$-algebra. If $M$ and $N$ are commutative algebras, so is $\on{Cotor}^*_\Lambda(M,N)$.
	
	\begin{constr}[Cobar construction]
	   	Let $\Delta\colon \Lambda \to \Lambda\otimes_k\Lambda$ and $\epsilon\colon \Lambda\to k$ denote the comultiplication and the counit maps of $\Lambda$, respectively. If $M$ is a graded left $\Lambda$-comodule, we may construct a cosimplicial object $\mc{D}_{\Lambda}(M)$ as follows. For all $s\geq 0$, set $\mc{D}_{\Lambda}(M)^s\coloneqq \Lambda^{\otimes s+1}\otimes_kM$. For every $0\leq i\leq s$, the $i$-th codegeneracy map $\sigma^s_i\colon \mc{D}_{\Lambda}(M)^{s+1}\to \mc{D}_{\Lambda}(M)^s$ is given by
	\[\sigma^s_i(\gamma_0\otimes\cdots\otimes\gamma_{s+1}\otimes m)=\epsilon(\gamma_{i+1})\gamma_0\otimes\cdots\otimes \gamma_{i}\otimes\gamma_{i+2}\otimes\cdots\otimes\gamma_{s+1}\otimes m\]
	for all $\gamma_0,\dots,\gamma_{s+1}\in \Lambda$ and $m\in M$. For every $0\leq i\leq s-1$, the $i$-th coface map $\delta_i^s\colon \mc{D}_{\Lambda}(M)^{s-1}\to \mc{D}_{\Lambda}(M)^{s}$ is given by 
	\[\delta^s_i(\gamma_0\otimes\cdots\otimes\gamma_{s}\otimes m)=\gamma_0\otimes\cdots\otimes\gamma_{i-1}\otimes\Delta(\gamma_i)\otimes\gamma_{i+1}\otimes\dots\otimes\gamma_s\otimes m\]
	for all $\gamma_0,\dots,\gamma_s\in \Lambda$ and $m\in M$, and $\delta_s^s$ is given by 
	\[\delta^s_s(\gamma_0\otimes\cdots\otimes\gamma_{s}\otimes m)=\gamma_0\otimes\dots\otimes\gamma_s\otimes \phi_M(m).\]

	There is a natural augmentation map $M\to \mc{D}_{\Lambda}(M)$. By definition, the {\em non-normalized cobar resolution} of $M$ is the cochain complex $\tilde{D}_{\Lambda}^*(M)$ such that $\tilde{D}_{\Lambda}^s(M)= \mc{D}_{\Lambda}(M)^s$ for all $s\geq 0$ and whose differentials are given by the alternating sums of the codegeneracy maps of $\mc{D}_{\Lambda}(M)$. 
	
	The (normalized) {\em cobar resolution} $D_{\Lambda}^*(M)$ of $M$ is defined as the normalized cochain complex associated to $\mc{D}_{\Lambda}(M)$, that is, by \[D^s_{\Lambda}(M)=\bigcap_{i=1}^{s+1}(\on{Ker}(\delta_i^{s+1}:\mc{D}_{\Lambda}(M)^{s}\to \mc{D}_{\Lambda}(M)^{s+1})=\Lambda\otimes_k \cl{\Lambda}^{\otimes s}\otimes_kM,\] where $\cl{\Lambda}$ is the kernel of the counit of $\Lambda$, and differential $d^s\colon D^s_{\Lambda}(M)\to D^{s+1}_{\Lambda}(M)$ induced by the one on $\tilde{D}_{\Lambda}^*(M)$:
	\begin{align}
		d^s(\gamma_0\otimes \gamma_1\otimes\dots\otimes\gamma_s\otimes m)\coloneqq &\sum_{i=0}^s(-1)^i\gamma_0\otimes\dots\gamma_{i-1}\otimes\Delta(\gamma_i)\otimes\gamma_{i+1}\otimes\dots\otimes\gamma_s\otimes m\\
		&+(-1)^{s+1}\gamma_0\otimes \gamma_1\otimes\dots\otimes\gamma_s\otimes \phi_M(m)  \nonumber
	\end{align}
	for all $\gamma_0\in \Lambda$, $\gamma_1,\dots,\gamma_s\in \cl{\Lambda}$ and $m\in M$. Our definition of the cobar resolution agrees with \cite[Definition A1.2.11]{ravenel1986complex}. It is a part of the cosimplicial Dold-Kan correspondence that the natural inclusion $D^*_{\Lambda}(M)\hookrightarrow \tilde{D}^*_{\Lambda}(M)$ is split and a homotopy equivalence; see e.g. \cite[Lemma 019I, (3)]{stacks-project} (the proof is dual to that of \cite[Lemma 019A]{stacks-project}). The complex $D^*_{\Lambda}(M)$ with the natural map $M \to D^*_{\Lambda}(M)$ gives an injective resolution of $M$; indeed, all of its terms are injective $\Lambda$-comodules by \cite[A1.2.2]{ravenel1986complex}. In particular, given a right $\Lambda$-comodule $N$ one can explicitly compute $\on{Cotor}^*_\Lambda(M,N)$ as the cohomology of $N\square_\Lambda D^*_{\Lambda}(M)$. 
	
	If $\Lambda$, $N$ and $M$ are $L$-graded, so are $D^*_{\Lambda}(M)$ and $N\square_\Lambda D^*_{\Lambda}(M)$; then $H^i(N\square_\Lambda D^*_{\Lambda}(M))$ computes $\on{Cotor}^i_\Lambda(M,N)$ as an $L$-graded vector space. 
	\end{constr}

	\begin{example}\label{main-comodule-algebra}
		If $G$ is a linear algebraic group over a field $k$ and $\Gamma\subset G$ is a central subgroup, the multiplication map $\Gamma\times G\to G$ is a $k$-group homomorphism and so induces a morphism of stacks $B\Gamma\times BG\to BG$. Passing to Hodge cohomology and applying K\"unneth's formula \cite[Proposition 5.1]{totaro2018hodge} we obtain a homomorphism of $\Z^2$-graded $k$-algebras
		\begin{equation}\label{coaction-bg}H^{*,*}_{\on{H}}(BG/k)\tto H^{*,*}_{\on{H}}(B\Gamma/k)\otimes_k H^{*,*}_{\on{H}}(BG/k).\end{equation}
		If $G=\Gamma$ is commutative, (\ref{coaction-bg}) makes $H^{*,*}_{\on{H}}(BG/k)$ into a $\Z^2$-graded Hopf algebra over $k$. If $G$ is an arbitrary linear algebraic group, (\ref{coaction-bg}) makes $H^{*,*}_{\on{H}}(BG/k)$ into a left $H^{*,*}_{\on{H}}(B\Gamma/k)$-comodule algebra. The same holds for de Rham cohomology, using the K\"unneth formula (\Cref{cor: Kunneth in de Rham over a field}). Namely, $H^{*}_{\dR}(B\Gamma/k)$ is a $\Z$-graded Hopf algebra and $H^{*}_{\dR}(BG/k)$ has a natural left $H^{*}_{\dR}(B\Gamma/k)$-comodule algebra structure.
	\end{example}

	\section{The Eilenberg-Moore spectral sequence}
	
	Let $K^{\bullet\bullet}$ be a first-quadrant double cochain complex, and write $\on{Tot}(K^{\bullet\bullet})$ for the associated total complex. By definition, the first spectral sequence associated to $K^{\bullet\bullet}$ is the spectral sequence
	\begin{equation}\label{first-ss}
	E_1^{ij}\coloneqq H^j(K^{i\bullet}) \Rightarrow H^n(\on{Tot}(K^{\bullet\bullet}))    
	\end{equation}
	associated to the filtration $F^*\on{Tot}(K^{\bullet\bullet})$ of $\on{Tot}(K^{\bullet\bullet})$ given by column degree, that is
	\[F^{i}\on{Tot}^n(K^{\bullet\bullet})\coloneqq\bigoplus_{h=i}^nK^{i',n-h};\]
	see \cite[012X]{stacks-project}.
	
	Let $X_{\bullet\bullet}$ be a bisimplicial scheme and $F$ be an abelian sheaf on the small \'etale site of $X_{\bullet\bullet}$. By \cite[Proposition 2.6]{friedlander1982etale}, there exists a first-quadrant spectral sequence
		\begin{equation}\label{friedlander}
		    E_1^{ij}\coloneqq H^j(X_{s\bullet},F|_{X_{s\bullet}})\Rightarrow H^{i+j}(X_{\bullet\bullet},F).
		\end{equation}
	It is defined as the first spectral sequence (\ref{first-ss}) associated to the double complex $K_{\bullet\bullet}=\on{Hom}_{\on{AbSh}(X_{\bullet\bullet})}(\Z_{X_{*\cdot}}, I)$, where $\on{AbSh}(X_{\bullet\bullet})$ is the category of abelian sheaves on the small \'etale site of $X_{\bullet\bullet}$, $\Z_{X_{*\bullet}}\to \Z$ is a certain projective resolution in $\on{AbSh}(X_{\bullet\bullet})$ and $F\to I$ is an injective resolution in $\on{AbSh}(X_{\bullet\bullet})$. (In \cite[Proposition 2.4]{friedlander1982etale}, the simplicial analogue of (\ref{friedlander}) is proved with more details than \cite[Proposition 2.6]{friedlander1982etale}.)
	

	\begin{proof}[Proof of \Cref{eilenberg-moore-easy}]
		We start by constructing the spectral sequence for Hodge cohomology. Since $\Gamma$ is a $k$-group of multiplicative type, there exist a $k$-torus $T$ and a $k$-group embedding $\Gamma \hookrightarrow T$. Define
		$$\tilde{G}\coloneqq (T\times G)/\Gamma.$$
		The projection $\tilde{G}\to T/\Gamma$ is a $G$-torsor, and the projection $\tilde{G}\to \cl{G}$ is a $T$-torsor, where $T$ acts as the subgroup $(T\times \Gamma)/\Gamma\subset \tilde{G}$. We obtain the following commutative diagram
		\begin{equation}\label{bisimp}
			\xymatrix{
				\vdots \ar[d] & \vdots \ar[d] &  \vdots \ar[d] &  \vdots \ar[d] \\
				\cl{G}^2 \ar[d] & \tilde{G}^2 \ar[d] \ar[l] & T\times \tilde{G}^2 \ar[d] \ar[l] & T^3\times \tilde{G}^2 \ar[l] \ar[d] &  \ar[l] \cdots \\
				\cl{G} \ar[d] & \tilde{G} \ar[d] \ar[l] & T\times \tilde{G} \ar[d] \ar[l] & T^3\times \tilde{G} \ar[l]  \ar[d] & \ar[l]\cdots  \\
				\Spec k \ar[d]  & T/\Gamma \ar[d] \ar[l] & (T/\Gamma)^2 \ar[d] \ar[l] & (T/\Gamma)^3 \ar[d] \ar[l] & \ar[l]\cdots   \\
				B\cl{G}  & BG \ar[l] & B\Gamma\times BG \ar[l] & (B\Gamma)^2\times BG \ar[l] & \ar[l]\cdots.
			}
		\end{equation}
		Here, the vertical maps at the bottom are the smooth morphisms induced by the trivial $\cl{G}$-torsor and the $(\Gamma^i\times G)$-torsor $T^i\times \tilde{G}\to (T/\Gamma)^{i+1}$, and the columns are the associated \v Cech covers. The horizontal maps on the left are induced by the natural projection $G\to \cl{G}$. The bottom row is $C(BG\to B\cl{G})$, and the collection of the horizontal arrows above it is the \v Cech nerve of the morphism of simplicial schemes $C(\Spec k \to BG)\to C(\Spec k\to B\cl{G})$.
		
		By \cite[p. 1577]{totaro2018hodge}, the Hodge cohomology of a smooth algebraic stack $X$ is isomorphic to the Hodge cohomology of the simplicial scheme associated to a smooth cover of $U\to X$, where $U$ is a $k$-scheme. Combining this with \cite[Proposition 3.7]{friedlander1982etale}, we obtain for all $h\geq 0$ an isomorphism $H^*(BG, \Omega^h)\simeq H^*(X_{\bullet\bullet}, \Omega^h)$, where $X_{\bullet\bullet}$ is the simplicial scheme consisting of (\ref{bisimp}) with the bottom row removed.
		Now (\ref{friedlander}) for $X_{\bullet\bullet}$ and $F=\Omega^h$ gives a spectral sequence
		$$E^{ij}_1(h)\coloneqq H^j((B\Gamma)^i\times BG,\Omega^h)\Rightarrow H^{i+j}(B\cl{G},\Omega^h).$$
		Letting $h$ vary, the spectral sequences $E^{ij}(h)$ assemble into a spectral sequence of graded $k$-algebras
		$$E^{ij}_1\coloneqq H^j_{\on{H}}((B\Gamma)^i\times BG/k)\Rightarrow H^{i+j}_{\on{H}}(B\cl{G}/k).$$
		By the K\"unneth formula in Hodge cohomology \cite[Proposition 5.1]{totaro2018hodge}, for every $i\geq 0$, \[E^{i*}_1=H^*_{\on{H}}(B\Gamma^i\times BG/k)\simeq H^*_{\on{H}}(B\Gamma/k)^{\otimes i}\otimes H^*_{\on{H}}(BG/k).\]
		This is the non-normalized cobar construction of the $H^*_{\on{H}}(B\Gamma/k)$-comodule algebra $H^*_{\on{H}}(BG/k)$. Indeed, the differentials are alternating sums of projection maps in (\ref{bisimp}), and using this one may check that they agree with those of the cobar construction. Therefore, by \cite[Corollary A1.2.12]{ravenel1986complex}, the $E_2$ page of the spectral sequence computes $\on{Cotor}^*_{H^*(B\Gamma)}(k,H^*_{\on{H}}(BG))$.
		
		The construction for the spectral sequence in de Rham cohomology is entirely analogous. Indeed,  while \cite[Proposition 2.6, Poposition 3.7]{friedlander1982etale} are only phrased for sheaves of abelian groups, they also hold for complexes of sheaves. For the K\"unneth formula in de Rham cohomology, see \Cref{cor: Kunneth in de Rham over a field}. 
	\end{proof}
	
	\begin{rmk}\label{column}
	Let $K^{\bullet\bullet}$ be a first-quadrant double cochain complex, and consider the spectral sequence (\ref{first-ss}) for $K^{\bullet\bullet}$. Let $u\in F^iH^{i+j}(\on{Tot}K^{\bullet\bullet})$ and $v\in F^{i'}H^{i'+j'}(\on{Tot}K^{\bullet\bullet})$. Then $u$ and $v$ are represented by classes $\cl{u}\in E_{\infty}^{i,j}$ and $\cl{v}\in E_{\infty}^{i',j'}$. Suppose that $\cl{u}\cdot\cl{v}=0$ in $E_{\infty}$. This means that $uv\in F^{i+i'+1}H^{i+i'+j+j'}(\on{Tot}K^{\bullet\bullet})$, and so there exists an integer $d\geq 1$ such that the representative  in $\cl{uv}\in E_{\infty}$ of $uv$ has bidegree $(i+j+d,i'+j')$. In other words, if $\cl{u}\cdot \cl{v}=0$ then the column degree of $\cl{uv}$ is strictly greater than the sum of the column degrees of $\cl{u}$ and $\cl{v}$.
		
	Since the Eilenberg-Moore spectral sequence is defined in terms of (\ref{first-ss}), this remark applies to it. We will make use of this observation during the proof of Theorems \ref{thm-hodge-pgl} and \ref{thm-hodge-pso}.
	\end{rmk}
		
	\begin{rmk}It follows from the construction of the Eilenberg-Moore spectral sequence in Hodge cohomology that 
		\[E_2^{0,*}\coloneqq \on{Cotor}^0_{H^*_{\on{H}}(B\Gamma/k)}(k, H^*_{\on{H}}(BG/k)) = PH^{*}_{\on{H}}(BG/k),\]
where we regard $H^*_{\on{H}}(BG/k)$ as a left $H^*_{\on{H}}(B\Gamma/k)$-comodule algebra.
		The corresponding edge homomorphism 
		\[H^*_{\on{H}}(B\cl{G}/k)\to PH^*_{\on{H}}(BG/k)\subset H^*_{\on{H}}(BG/k)\]
		is exactly the pull-back with respect to the map $BG\to B\cl{G}$ induced by the projection $G\to \cl{G}$. A similar description  holds for de Rham cohomology.
	\end{rmk}
	
	\begin{rmk}\label{rem:about Hodge grading}
	    Let us emphasize that, taking into account the bigrading on Hodge cohomology and the construction of the Eilenberg-Moore spectral sequence, $E_2^{i,j}$ decomposes as a direct sum $E_2^{i,j}\simeq \oplus_{h}(E_2^{i,j})^h$ where 
	    $$
	    (E_2^{i,j})^h\coloneqq \left(\on{Cotor}^i_{H^*_{\on{H}}(B\Gamma/k)}(k, H^{*}_{\on{H}}(BG/k))\right)^{h,j}\Rightarrow H^{h,i+j}_{\on{H}}(B\cl{G}/k).
	    $$
	    Here the bigrading on $\on{Cotor}^i_{H^*_{\on{H}}(B\Gamma/k)}(k, H^{*}_{\on{H}}(BG/k))$ is the one coming from $H^{*,*}_{\on{H}}(B\Gamma/k)$ and $H^{*,*}_{\on{H}}(BG/k)$.
	\end{rmk}

	\section{Additive part of Theorem \ref{mainthm} for projective linear and symplectic groups} \label{sec:GL SP graded vector space}
	The goal of this section is to identify the $E_2$ page of the Eilenberg-Moore spectral sequence for the Hodge and de Rham cohomology with their topological counterparts in the case of $\mr{PGL}_n$ and $\mr{PSp}_{n}$. This is done by explicitly comparing the coalgebras and comodules that take part in the Cotor description on both sides. In \Cref{additive} we then record how to deduce the degeneration of the above spectral sequences in the case $n=4m+2$ from Totaro's inequality and the results of Toda \cite{toda1987cohomology} for the singular cohomology. 
	\subsection{Projective linear group}\label{ssect:projective}

	Let $n\geq 1$ be an integer and $p$ be a prime number. We have isomorphisms
	$$H^*_{\on{H}}(B\G_{\on{m}}/\F_p)\simeq \F_p[x_2],\qquad H^*_{\on{H}}(B\mr{GL}_n/\F_p)\simeq \F_p[c_1,\dots,c_n],$$ where $x_2\in H^{1,1}_\Hdg(B\G_{\on{m}}/\F_p)$ and $c_i\in H^{i,i}_\Hdg(B\mr{GL}_n/\F_p)$. If $\iota\colon T_n\hookrightarrow \on{GL}_n$ is the diagonal maximal torus, then by the K\"unneth formula $H^*_{\on{H}}(BT_n/\F_p)$ is a polynomial ring in $n$ generators $t_1,\dots,t_n$ of bidegree $(1,1)$, and the pullback map
	\[B\iota^*\colon H^*_{\on{H}}(B\mr{GL}_n/\F_p)\tto H^*_{\on{H}}(BT_n/\F_p)\simeq \F_p[t_1,\ldots,t_n]\]
	is injective and identifies $c_i$ with the $i$-th symmetric function on variables $t_i$; see \cite[End of proof of Theorem 9.2]{totaro2018hodge}.
	
	If we identify $\G_{\on{m}}$ with the center of $\on{GL}_n$, the multiplication map $\G_{\on{m}}\times \on{GL}_n\to \on{GL}_n$ is a group homomorphism, and so induces a morphism of stacks $B\G_{\on{m}}\times B\mr{GL}_n\to B\mr{GL}_n$. By the K\"unneth formula, we obtain a ring homomorphism 
	\begin{equation}\label{coaction-gl-eq}\phi\colon H^*_{\on{H}}(B\mr{GL}_n/\F_p)\to H^*_{\on{H}}(B\G_{\on{m}}/\F_p)\otimes H^*_{\on{H}}(B\mr{GL}_n/\F_p),\end{equation}
	which endows $H^*_{\on{H}}(B\mr{GL}_n/\F_p)$ with the structure of an $H^*_{\on{H}}(B\G_{\on{m}}/\F_p)$-comodule algebra (as discussed in  \Cref{main-comodule-algebra}).
	
	Let us explicitly describe the coalgebra structure on $H^*_{\on{H}}(B\G_{\on{m}}/\F_p)$. The restriction of the product map $\G_{\on{m}}\times\G_{\on{m}}\to \G_{\on{m}}$ to $\G_{\on{m}}\times\set{1}$ and $\set{1}\times \G_{\on{m}}$ is the identity, hence the comultiplication map 
	$$
	\Delta\colon H^2_{\on{H}}(B\G_{\on{m}}/\F_p) \to H^2_{\on{H}}(B\G_{\on{m}}\times B\G_{\on{m}}/\F_p) \xymatrix{\ar[r]^{\text{K\"unneth}}_\sim&} H^2_{\on{H}}(B\G_{\on{m}}/\mbb F_p)\oplus H^2_{\on{H}}(B\G_{\on{m}}/\mbb F_p)$$
	is the diagonal embedding. It then follows that the Hopf algebra structure on $H^*_{\on{H}}(B\G_{\on{m}}/\F_p)\simeq \F_p[x_2]$ is the unique one given by $\Delta(x_2)=x_2\otimes 1 + 1\otimes x_2$, namely 
	$$
	\Delta(x_2^n)=\sum_{i=0}^n \binom{n}{i}x_2^{n-i}\otimes x_2^i.
	$$

	Since $H^{*,*}_{\on{H}}(B\mr{GL}_n/\F_p)$ is concentrated in bidegrees $(i,i)$ and is a polynomial algebra, the Hodge-de Rham spectral sequence for $B\mr{GL}_n$ degenerates and induces an isomorphism \[H^*_{\on{H}}(B\mr{GL}_n/\F_p)\simeq H^*_{\on{dR}}(B\mr{GL}_n/\F_p).\] Thus the previous discussion also applies to $H^*_{\on{dR}}(B\mr{GL}_n/\F_p)$. Below, in this section we will be giving proofs in the case of Hodge cohomology, but the same arguments then apply to de Rham context.
	
	We now write \[H^*_{\on{sing}}(B\C^{\times};\F_p)\simeq \F_p[x_2^{\on{top}}],\qquad H^*_{\on{sing}}(B\mr{GL}_n(\C);\F_p)\simeq \F_p[c_1^{\on{top}},\dots,c_n^{\on{top}}],\]
	where $|x^{\on{top}}_2|=2$ and $|c_i^{\on{top}}|=2i$. We naturally regard $H^*_{\on{sing}}(B\mr{GL}_n(\mbb C);\F_p)$ as a $H^*_{\on{sing}}(B\C^{\times};\F_p)$-comodule algebra, and let $\phi^{\on{top}}$ be the coaction map. We note that the coalgebra structure on $H^*_{\on{sing}}(B\C^{\times};\F_p)$ is the unique one given by $\Delta(x^{\on{top}}_2)=x^{\on{top}}_2\otimes 1 + 1\otimes x^{\on{top}}_2$.
	
	It is easy to see that there is an isomorphism of Hopf algebras $$H^*_{\on{H}}(B\G_{\on{m}}/\F_p)\simeq H^*_{\on{sing}}(B\C^{\times};\F_p)$$ defined by sending $x_2$ to $x_2^{\on{top}}$. Lemma below shows that it can be extended to an isomorphism between the comodule-algebras $H^*_{\on{H}}(B\mr{GL}_n/\F_p)$ and $H^*_{\on{sing}}(B\mr{GL}_n(\mbb C);\F_p)$, inducing in particular an isomorphism of cotorsion groups of our interest. 
	\begin{lemma}\label{comparison}
		Consider the $H^*_{\on{H}}(B\G_{\on{m}}/\F_p)$-comodule algebra $H^*_{\on{H}}(B\mr{GL}_n/\F_p)=\F_p[c_1,\dots,c_n]$, with the coaction $\phi$ of (\ref{coaction-gl-eq}).
		
		(a) For all $i=1,\dots,n$, we have \[\phi(c_i)=\sum_{i_1+i_2=i}\binom{n-i_2}{i_1}x_2^{i_1}\otimes c_{i_2},\]
		where we use the convention that $c_0\coloneqq 1$.

		(b) The isomorphisms
		\[H^*_{\on{H}}(B\G_{\on{m}}/\F_p)\xrightarrow{\sim}H^*_{\on{sing}}(B\C^{\times};\F_p),\qquad H^*_{\on{H}}(B\mr{GL}_n/\F_p)\xrightarrow{\sim}H^*_{\on{sing}}(B\mr{GL}_n(\C);\F_p)\]
		given by sending $x_2\mapsto x_2^{\on{top}}$ and $c_i\mapsto c_i^{\on{top}}$ induce an isomorphism of bigraded algebras
		\[\on{Cotor}^*_{H^*_{\on{H}}(B\G_{\on{m}}/\F_p)}(\F_p, H^*_{\on{H}}(B\mr{GL}_n/\F_p)) \simeq \on{Cotor}^*_{H^*_{\on{sing}}(B\C^{\times};\F_p)}(\F_p, H^*_{\on{sing}}(B\mr{GL}_n(\C);\F_p)).\]
		
		Entirely analogous statements hold with Hodge cohomology replaced by de Rham cohomology.
	\end{lemma}
	
	\begin{proof}
		(a) The proof is analogous to that of \cite[Proposition 3.2]{toda1987cohomology}. We have a commutative diagram
		\[
		\xymatrix{
			H^*_{\on{H}}(B\mr{GL}_n/\F_p) \ar[r]^(.36){\phi} \ar@{^{(}->}[d]^{B\iota^*}  & H^*_{\Hdg}(B\mr{GL}_n/\F_p)\otimes H^*_\Hdg(B\G_{\on{m}}/\F_p) \ar@{^{(}->}[d]^{1\otimes B\iota^*}  \\
			H^*_{\on{H}}(BT_n/\F_p) \ar[r]^(.36)\phi  & H_{\on{H}}^*(BT_n/\F_p)\otimes H_{\on{H}}^*(B\G_{\on{m}}/\F_p),
			}
		\]
		where $\phi'$ is the coaction map for $T_n$. We have $H^*_{\on{H}}(BT_n/\F_p)=\F_p[t_1,\dots,t_n]$, where $t_i$ has degree $2$ for all $i$. Then \[\phi'(t_i)=1\otimes t_i+x_2\otimes 1,\qquad B\iota^*(\sum c_i)=\prod_{i=1}^n(1+t_i).\]
		Elementary calculations show that
		\[(1\otimes B\iota)^*\phi(\sum c_i)=(1\otimes B\iota)^*\sum_{j=0}^n\sum_{i=0}^j\binom{n-j}{i} x_2^i\otimes c_j.\]
		Now (a) follows from the injectivity of $(1\otimes B\iota)^*$.
		
		(b) By (a) and \cite[Proposition 3.2]{toda1987cohomology}, the stated isomorphisms induce a commutative square 
		\[
		\xymatrix{
			H^*_{\on{H}}(B\mr{GL}_n/\F_p) \ar[r]^(.37){\phi} \ar[d]^\wr  & H^*_{\on{H}}(B\G_{\on{m}}/\F_p)\otimes H^*_{\on{H}}(B\mr{GL}_n/\F_p) \ar[d]^\wr \\
			H^*_{\on{sing}}(B\mr{GL}_n(\C);\F_p) \ar[r]^(.38){\phi^{\on{top}}} & H^*_{\on{sing}}(B\C^{\times};\F_p)\otimes H^*_{\on{sing}}(B\mr{GL}_n(\C);\F_p),
 }
		\]
	where vertical maps are the ones in the statement of the proposition. This shows that the corresponding comodule-algebras are equivalent, and, consequently, one also has an isomorphism of the corresponding $\on{Cotor}$-algebras.
	\end{proof}

	\subsection{Projective symplectic group}\label{sec:PSp}
	
	By \cite[Theorem 9.2]{totaro2018hodge}, we have an isomorphism \[H^*_{\on{H}}(B\mr{Sp}_{2n}/\F_2)=\F_2[q_1,\dots,q_n],\]
	where $q_i$ has bidegree $(2i,2i)$ for $i=1,\dots,n$.
	By \cite[Proposition 10.1]{totaro2018hodge}, we have an isomorphism of Hopf algebras
	\[H^*_{\on{H}}(B\mu_2/\F_2)\simeq \F_2[x_1,x_2]/(x_1^2),\]
	where $x_1\in H^{1,0}_\Hdg(B\mu_2/\F_2)$ and $x_2\in H^{1,1}_\Hdg(B\mu_2/\F_2)$. If we identify $\mu_2$ with the center of $\on{Sp}_{2n}$, the multiplication map $\mu_2\times \on{Sp}_{2n}\to \on{Sp}_{2n}$ induces a ring homomorphism.
	Let 
	\begin{equation}\label{coaction-sp-eq}
	\phi\colon H^*_{\on{H}}(B\mr{Sp}_{2n}/\F_2)\tto H^*_{\on{H}}(B\mu_2/\F_2)\otimes H^*_{\on{H}}(B\mr{Sp}_{2n}/\F_2).
	\end{equation}
	We can view $H^*_{\on{H}}(B\mr{Sp}_{2n}/\F_2)$ as an $H^*_{\on{H}}(B\mu_2/\F_2)$-comodule algebra, with coaction map $\phi$.
	
	\begin{lemma}\label{restrict}
		Let $j\colon\on{Sp}_{2n}\hookrightarrow \on{GL}_{2n}$ be the tautological inclusion. Then
		\[Bj^*(c_{2h})=q_h,\qquad Bj^*(c_{2h+1})=0\]
		for all $h=1,\dots,n$.	In particular, $Bj^*$ is injective.
	\end{lemma}
	
	\begin{proof}
		Let $T_{2n}\subset \on{GL}_{2n}$ be the diagonal maximal torus, and let $T'_{2n}\coloneqq j^{-1}(T_{2n})$. Then $T'_{2n}$ is a maximal torus of $\on{Sp}_{2n}$, and $j$ induces a commutative square
		\[
		\xymatrix{
			H^*_{\on{H}}(B\mr{GL}_{2n}/\F_2) \ar[r]^{Bj^*} \ar@{^{(}->}[d] &  H^*_{\on{H}}(B\mr{Sp}_{2n}/\F_2) \ar@{^{(}->}[d] \\ 
			H^*_{\on{H}}(BT_{2n}/\F_2)\ar[r] &  H^*_{\on{H}}(BT'_{2n}/\F_2),
		}
		\]
		where the vertical are injective by \cite[End of proof of Theorem 9.2]{totaro2018hodge}. For $i=1,\dots,n$, let  $t_i\in H^2(BT_{2n}/\F_2)$ be the generator corresponding to the $i$-th coordinate of $T_{2n}$. Then the bottom horizontal arrow gives an identification \[H^*_{\on{H}}(BT'_{2n}/\F_2)\simeq  H^*_{\on{H}}(BT_{2n}/\F_2)/(t_1+t_{n+1}, t_2+t_{n+2},\dots, t_n+t_{2n}).\] 
		The conclusion follows from the fact that $c_i$ is the $i$-th symmetric function of $t_1,\dots,t_n$, and that under the previous identification $q_i$ is the $i$-th symmetric function of $t_1^2,\dots,t^2_n$; see \cite[Proof of Theorem 6.2.2]{chaput2010adjoint}.
	\end{proof}
	
	Let us also identify the coalgebra structure on $H^*_{\on{H}}(B\mu_2/\F_2)$. Similarly to the $\mbb G_m$-case one can see that the comultiplication map 
	$$
	\Delta\colon H^1_{\on{H}}(B\mu_2/\F_p) \tto H^1_{\on{H}}(B\mu_2\times B\mu_2/\F_p) \xymatrix{\ar[r]^{\text{K\"unneth}}_\sim&} H^1_{\on{H}}(B\mu_2/\mbb F_p)\oplus H^1_{\on{H}}(B\mu_2/\mbb F_p)$$
	is the diagonal embedding, so $\Delta(x_1)=x_1\otimes 1 + 1\otimes x_1$. The class $x_2\in H^*_{\on{H}}(B\mu_2/\F_2)$ is the pullback of the generator of $H^*_{\on{H}}(B\G_{\on{m}}/\F_2)$ (see the proof of \cite[Proposition 10.1]{totaro2018hodge}) and so we also have $\Delta(x_2)=x_2\otimes 1 + 1\otimes x_2$. This defines a Hopf-algebra structure on  $H^*_{\on{H}}(B\mu_2/\F_2)\simeq \mbb F_2[x_1,x_2]/x_1^2$ uniquely. 
	
	Recall that $H^*_{\mr{sing}}(B\mbb Z/2;\mbb F_2)\simeq \mbb F_2[z]$, where $|z|=1$. The natural Hopf algebra structure here is uniquely defined by $\Delta(z)=z\otimes 1+ 1\otimes z$.
	\begin{rmk}\label{rem:Sp case iso as coalgebras}
	Note that $H^*_{\on{H}}(B\mu_2/\F_2)$ is \textit{not} isomorphic to $H^*_{\mr{sing}}(B\mbb Z/2;\mbb F_2)$ as an algebra. Nevertheless, the unique $\mbb F_2$-linear map $\psi\colon H^*_{\on{H}}(B\mu_2/\F_2) \to H^*_{\mr{sing}}(B\mbb Z/2;\mbb F_2)$ which sends a $x_2^i$ to $z^{2i}$ and $x_1x_2^i$ to $z^{2i+1}$ is easily checked to be an isomorphism of \textit{coalgebras}. 
	\end{rmk}
	
	Recall that
		\[H^*_{\on{sing}}(B\mr{Sp}_{2n};\F_2)=\F_2[q^{\on{top}}_1,\dots,q^{\on{top}}_n],\]
		where $q^{\on{top}}_i$ has degree $4i$  for $i=1,\dots,n$.
	
	\begin{lemma}\label{comparison-sp}
		Consider $H^*_{\on{H}}(B\mr{Sp}_{2n}/\F_2)=\F_2[q_1,\dots,q_n]$, with the coaction $\phi$ of (\ref{coaction-sp-eq}). 
		
		(a) For all $i=1,\dots,n$, we have \[\phi(q_i)=\sum_{i_1+i_2=i} \binom{n-i_2}{i_i}x_2^{2i_1}\otimes q_{i_2}.\]
		
		(b) The isomorphism of coalgebras $\psi\colon H^*_{\on{H}}(B\mu_2/\F_2) \to H^*_{\mr{sing}}(B\mbb Z/2;\mbb F_2) $ together with the map of comodules $H^*_{\on{H}}(B\mr{Sp}_{2n}/\F_2) \to  H^*_{\mr{sing}}(B\mr{Sp}_{2n};\mbb F_2)$ sending $q_i$ to $q^{\on{top}}_i$ induces an isomorphism of bigraded $\F_2$-vector spaces
		\[\on{Cotor}^*_{H^*_{\on{H}}(B\mu_2/\F_2)}(\F_2, H^*_{\on{H}}(B\mr{Sp}_{2n}/\F_2))=\on{Cotor}^*_{H^*_{\on{sing}}(B(\Z/2\Z);\F_2)}(\F_2, H^*_{\on{sing}}(B\mr{Sp}_{2n};\F_2)).\]
	\end{lemma}
	
	\begin{proof}
		(a) Since $\phi$ respects the gradings and $H^*_{\on{H}}(B\mr{Sp}_{2n}/\F_2)$ is concentrated in even degrees, the  $x_2^ix_1$-component of $\phi(q_h)$ must be zero for all $h$. To compute the $x_2^i$-components, observe that we have the following commutative square
		\[
		\xymatrix{
			H^*_{\on{H}}(B\mr{Sp}_{2n}/\F_2) \ar@{^{(}->}[r]^{Bj^*} \ar[d] & H^*_{\on{H}}(B\mr{GL}_{2n}/\F_2) \ar[d]  \\
			H^*_{\on{H}}(B\mu_2/\F_2) \ar[r] &  H^*_{\on{H}}(B\G_{\on{m}}/\F_2),
	}
		\]
		where $\G_{\on{m}}$ is viewed as the center of $\on{GL}_{2n}$ and $\mu_2$ as the center of $\on{Sp}_{2n}$. The conclusion follows from \Cref{comparison}(a) and \Cref{restrict}.
		
		(b) We have
		\[H^*_{\on{sing}}(B\mr{Sp}_{2n};\F_2)=\F_2[q^{\on{top}}_1,\dots,q^{\on{top}}_n],\]
		where $q^{\on{top}}_i$ has degree $4i$  for $i=1,\dots,n$. The coaction map
		\[\phi^{\on{top}}\colon H^*_{\on{sing}}(B\mr{Sp}_{2n};\F_2)\to H^*_{\on{sing}}(B(\Z/2\Z);\F_2) \otimes H^*_{\on{sing}}(B\mr{Sp}_{2n}/\F_2)\]
		has been computed in \cite[Proposition 3.4]{toda1987cohomology} and agrees with the formula that we have proved in part (a). Recall (\Cref{rem:Sp case iso as coalgebras}) that we have an isomorphism of coalgebras
		\[\psi\colon H^*_{\on{H}}(B\mu_2/\F_2)\xrightarrow{\sim} H^*_{\on{sing}}(B(\Z/2\Z);\F_2).\]
		
		Comparison of (a) with Toda's formula for $\phi^{\on{top}}$ yields the commutativity of the following diagram of graded linear maps
		\[
		\xymatrix{
			H^*_{\on{H}}(B\mr{Sp}_{2n}/\F_2) \ar[r]^(.38)\phi \ar[d]^\wr  & H^*_{\on{H}}(B\mu_2/\F_2)\otimes H^*_{\on{H}}(B\mr{Sp}_{2n}/\F_2) \ar[d]^\wr \\
			H^*_{\on{sing}}(B\mr{Sp}_{2n};\F_2) \ar[r]^(.37){\phi^{\on{top}}} & H^*_{\on{sing}}(B(\Z/2\Z);\F_2)\otimes H^*_{\on{sing}}(B\mr{Sp}_{2n};\F_2), 
}
		\]
		where the vertical maps are induced by $\psi$ and the graded algebra isomorphism 
		\[H^*_{\on{H}}(B\mr{Sp}_{2n}/\F_2)\ism  H^*_{\on{sing}}(B\mr{Sp}_{2n};\F_2),\qquad q_h\mapsto q_h^{\on{top}}.\]
		This induces the desired isomorphism between $\on{Cotor}^*$ (as graded vector spaces).
	\end{proof}
	
	
	\subsection{Proof of the additive part of Theorem \ref{mainthm} for general linear and symplectic groups}\label{additive}
	
	Let $G$ be one of the split reductive $\Z$-groups $\on{GL}_{4m+2}$ and $\on{Sp}_{4m+2}$, and let $\Gamma$ be the center of $G$.  We now prove that we have isomorphisms of $\Z$-graded $\F_2$-vector spaces
	\[H^*_{\on{H}}(B\cl{G}/\F_2)\simeq H^*_{\on{dR}}(B\cl{G}/\F_2)\simeq H^*_{\on{sing}}(B\cl{G}(\C);\F_2).\]
	In particular, this establishes \Cref{mainthm}(2).

	Consider the Eilenberg-Moore spectral sequence associated by \Cref{eilenberg-moore-easy} to the central short exact sequence
	\begin{equation}\label{ces}1\to \Gamma_{\F_2} \to {G}_{\F_2}\to \cl{G}_{\F_2}\to 1.\end{equation}
	By Lemmas \ref{comparison}(b) 
	and \ref{comparison-sp}(b), its $E_2$ page is isomorphic to the $E_2$ page of the topological Eilenberg-Moore spectral sequence of \cite[(4.1)]{toda1987cohomology} for the fibration $\Gamma(\C)\to G(\C)\to \cl{G}(\C)$. Toda showed that the latter spectral sequence degenerates on the second page; see \cite[p. 99, line 8]{toda1987cohomology} for $G=\on{GL}_{4m+2}$ and \cite[p. 102, line 9]{toda1987cohomology} for $\on{Sp}_{4m+2}$. 
	Therefore,
	\[ \dim_{\F_2}H^i_{\on{H}}(B\cl{G}/\F_2)\leq \dim_{\F_2}H^i_{\on{sing}}(B\cl{G}(\C);\F_2)\]
	for all $i\geq 0$ (where the equality holds for all $i$ if and only if the spectral sequence for Hodge cohomology also degenerates). On the other hand, by \cite[Theorem 5.3.6]{kubrak2021p-adic} and the existence of the Hodge-de Rham spectral sequence, we have 
	\[\dim_{\F_2}H^i_{\on{H}}(B\cl{G}/\F_2)\geq \dim_{\F_2}H^i_{\on{dR}}(B\cl{G}/\F_2)\geq  \dim_{\F_2}H^i_{\on{sing}}(B\cl{G}(\C);\F_2).\]
	We conclude that the algebraic Eilenberg-Moore spectral sequence associated to (\ref{ces}) degenerates and that we have equality of dimensions, as desired.
	\begin{rmk}
	    The same argument would also apply for $n\neq 4m+2$, if we knew that the Eilenberg-Moore spectral degenerates on the topological side in this case.
	\end{rmk}
	
	\section{Computation of Cotor}	
	\label{sec: Cotor}
	
	In this section we develop some tools for the explicit computation of $\on{Cotor}$ groups in subsequent sections. In \Cref{ssect:resolution_of_trivial} we construct an explicit resolution of the trivial comodule over $\Lambda_1\coloneqq H_{\Hdg}^*(B\mbb G_m/\mbb F_2)$ and $\Lambda_2\coloneqq H_{\Hdg}^*(B\mu_2/\mbb F_2)$ via the twisted tensor product construction (\Cref{constr:twisted tensor product}). Then, in \Cref{ssect:subalgebra_like_in_Toda} we extend a key computational lemma of Toda that allows to decompose a graded $\Lambda_1$-comodule algebra $A$ as a tensor product, to graded $\Lambda_2$-comodule algebras. These results will be 
 used in Sections \ref{sec: PGL} and \ref{sec: PSO} to compute the bigrading on the $E_2$ page of the Eilenberg-Moore spectral sequence for $\mr{PGL}_{4m+2}$ and $\mr{PSO}_{4m+2}$.

	\subsection{Totalizations}
	Let $k$ be a field and let $V^{*_1,*_2}$ be a $\Z^2$-graded vector space over $k$. Given a 
 bihomogeneous element $v\in V^{i_1,i_2}$ we denote $|v|\coloneqq (\deg_1(v),\deg_2(v))$ where $\deg_1(v)\coloneqq i_1$ and $\deg_2(v)\coloneqq i_2$. Given a bigraded vector space $V^{*_1,*_2}$, we can associate to it a $\Z$-graded vector space $V^{\mr{tot}}$ by putting $(V^{{\mr{tot}}})^i\coloneqq \oplus_{i_1+i_2=i} V^{i_1,i_2}$. For all $v\in V^{i_1,i_2}$, we denote by $|v|_{\mr{tot}}\coloneqq \deg_1(v)+\deg_2(v)=i_1+i_2$ the total degree of $v$. The \textit{totalization} functor $V^{*_1,*_2}\mapsto V^{{\mr{tot}}}$ from $\mbb Z^2$-graded vector spaces to $\mbb Z$-graded vector spaces is symmetric monoidal. In particular, having a $\Z^2$-graded coalgebra $C^{*_1,*_2}$ or a $\Z^2$-graded algebra $A^{*_1,*_2}$ the corresponding totalizations  $C^{\mr{tot}}$ and $A^{\mr{tot}}$ are naturally a  $\Z$-graded coalgebra and a $\Z$-graded algebra, respectively. Note that the underlying (ungraded) coalgebra $C$ and algebra $A$ stay the same:
	$$
	C\coloneqq \oplus_{i_1,i_2}C^{i_1,i_2}\simeq \oplus_i (C^{\mr{tot}})^i \quad \text{ and } \quad A \coloneqq \oplus_{i_1,i_2}A^{i_1,i_2}\simeq \oplus_i (A^{\mr{tot}})^i.
	$$

	\subsection{An explicit resolution of the trivial comodule as the twisted tensor product}\label{ssect:resolution_of_trivial}
	Let $k=\F_2$. We regard $\F_2$ as a trivial left comodule over any Hopf algebra over $\mbb F_2$ via the counit map. 
	
	\begin{constr}
	   Consider the primitively generated $\Z^2$-graded Hopf algebras
	\begin{equation}\label{lambda}\Lambda_1\coloneqq \F_2[x_2],\qquad \Lambda_2\coloneqq \F_2[x_1,x_2]/(x_1^2),\end{equation} 
	where the bigradings of the generators are $|x_1|\coloneqq (1,0)$ and $|x_2|\coloneqq(1,1)$. Their total gradings are $|x_i|_{\mr{tot}}=i$. As before, by  ``primitively generated,'' we mean that $x_1$ and $x_2$ are primitive elements, that is, the comultiplication map sends $x_i$ to $x_i\otimes 1+1\otimes x_i$ for $i=1,2$. 
	\end{constr} 
	
	\begin{rmk}\label{rem:incarnations of lambdas}
	We note that $\Lambda_1\simeq H^{*,*}_H(B\mbb G_m/\mbb F_2)$ and $\Lambda_1^{\mr{tot}}\simeq H^*_\Hdg(B\mbb G_m/\mbb F_2)\simeq H^*_\dR(B\mbb G_m/\mbb F_2)$, while $\Lambda_2\simeq H^{*,*}_H(B\mu_2/\mbb F_2)$ and  $\Lambda_1^{\mr{tot}}\simeq H^*_\Hdg(B\mu_2/\mbb F_2)\simeq H^*_\dR(B\mu_2/\mbb F_2)$.
	\end{rmk}
	
	
	Below we will often not distinguish between $\Lambda_i$ and $\Lambda_i^{\mr{tot}}$, considering the same coalgebra as $\mbb Z^2$ or $\mbb Z$-graded depending on the context. 
	
	\begin{constr}\label{constr:twisted cochains}\begin{enumerate} \item Consider a bigraded polynomial algebra 
	\[R_1\coloneqq\F_2[z_3,z_5,\dots,z_{2^h+1},\dots]\]
	with the bigradings of the variables given by $|z_{2^h+1}|\coloneqq (2^{h-1}, 2^{h-1})$. The total gradings are $|z_i|_{\mr{tot}}=i-1$. 
	Define an $\mbb F_2$-linear graded map $\theta_1\colon \Lambda_1\to R_1$ by setting:\begin{itemize}
	    \item $\theta_1(x_2^{2^h})=z_{2^{h+1}+1}$;
	    \item $\theta_1(x_2^j)=0$ if $j$ is not a power of $2$.
	\end{itemize} 
\medskip 
	\item Similarly, consider a bigraded polynomial algebra 
	\[R_2\coloneqq\F_2[z_2,z_3,z_5,\dots,z_{2^h+1},\dots]\]
	with bigradings $|z_2|=(1,1)$ and $|z_{2^h+1}|=(2^{h-1}, 2^{h-1})$ for $h\ge 1$. The total gradings are given by $|z_i|_{\mr{tot}}=i-1$. Define a graded map $\theta_2\colon \Lambda_2\to R_2$ by setting:
	\begin{itemize}
	    \item $\theta_2(x_1)=z_2$,
	    \item $\theta_2(x_2^{2^h})=z_{2^{h+1}+1}$,
	    \item $\theta_2(x_2^j)=0$  if $j$ is not a power of $2$,
	    \item $\theta_2(x_1x_2^l)=0 $ for all $l\geq 1$.
	\end{itemize}
	    \end{enumerate}
	\end{constr}
	
		\begin{rmk}
	    The algebras $R_1$, $R_2$ are quite large, but they are not arbitrary, in fact $R_1\simeq H^{*,*}_\Hdg(K(\mbb G_m,2)/\mbb F_2)$ and $R_2\simeq H^{*,*}_\Hdg(K(\mu_2,2)/\mbb F_2)$, but with a slight change of the bigrading: namely ``Hodge bigradings" (see \Cref{rem:Hodge bigrading}) of $z_{2^h+1}$ are $(2^{h-1},2^{h-1}+1)$, while for $z_2$ it should be $(1,1)$. The maps $\theta_i$ then appear as the non-zero differentials that appear in the Leray-Serre spectral sequence associated to the fibrations $B\mbb G_m \to \mr{pt} \to K(\mbb G_m,2)$ and $B\mbb \mu_2 \to \mr{pt} \to K(\mu_2,2)$.
	\end{rmk}

	

	\begin{lemma}\label{lem:thetas are twisted cochains}
	    The maps $\theta_i$ satisfy the following equation:
	    \begin{equation}\label{eq:twisted cochain equation}
	        \mu_{R_i}\circ (\theta_i\otimes \theta_i)\circ \Delta_{\Lambda_i}=0,
	    \end{equation}
	    where $\mu_{R_i}\colon R_i\otimes R_i \to R_i$ and $\Delta_{\Lambda_i}\colon \Lambda_i\to \Lambda_i\otimes \Lambda_i$ are the multiplication and comultiplication maps, respectively.
	\end{lemma}
	\begin{proof}
	    We will only show the case $i=1$, leaving the case $i=2$ to the reader. If $x=\sum_j x_2^{a_j}$, then
		\[\Delta_{\Lambda_1}(x)=\sum_j(x_2\otimes 1+1\otimes x_2)^{a_j}=\sum_j \binom{a_j}{h} x_2^h\otimes x_2^{a_j-h}.\]
		We have $\theta_1(x_2^h)\otimes \theta_1(x_2^{a_j-h})=0$ unless $h=2^d$ and $a_j-h=2^e$ for some $d,e\geq 0$. In that case
		\[
		(\mu_{R_1}\circ (\theta_1\otimes\theta_1))(x_1^{2^d}\otimes x_1^{2^e})=z_{2^{d+1}+1}z_{2^{e+1}+1}=(\mu_{R_1}\circ (\theta_1\otimes\theta_1))(x_1^{2^e}\otimes x_1^{2^d}).
		\]
		Since $\binom{2^d+2^e}{2^d}=\binom{2^d+2^e}{2^e}$, the contributions of $\binom{2^d+2^e}{2^d}x_1^{2^d}\otimes x_1^{2^e}$ and $\binom{2^d+2^e}{2^e}x_1^{2^e}\otimes x_1^{2^d}$ cancel each other out, hence $(\mu_{R_1}\circ (\theta_1\otimes\theta_1)\circ\Delta_{\Lambda})(x)=0$, as desired.
	\end{proof}
	
	\begin{rmk}
	    Maps $\theta$ that satisfy (\ref{eq:twisted cochain equation}) are usually called \textit{twisting cochains}. They allow to define the \textit{twisted tensor product} (see \Cref{constr:twisted tensor product} below).
	\end{rmk}

	To proceed further we need to introduce another auxiliary $\Z$-grading on $R_i$:
	\begin{constr}
	    Define a grading $\deg_z\colon R_i\to \Z$ by putting $\deg_z(z_j)=1$ for all generators $z_j\in R$. In other words, $\deg_z$ is the degree of an element of $R_i$ as a polynomial in $z_j$'s. Given any $\F_2$-vector space $M$ we put a grading $\deg_z$ on $M\otimes R_i$ by setting $\deg_z(m\otimes f)\coloneqq \deg_z(f)$. Note that $\deg_z(\theta_i(x))=1$ for any $x\in \Lambda_i$.
	\end{constr}
	
	\begin{constr}\label{constr:twisted tensor product}
	    Let $M$ be a $\Z^2$-graded left $\Lambda_i$-comodule. Let $\phi_M\colon M \to \Lambda_i\otimes M$ be the coaction map and $m_{R_i}$ be the multiplication. Consider $R_i\otimes M$ and endow it with an operator $d_{\theta_i}\colon R_i\otimes M \to R_i\otimes M$ defined as the composition
	    $$
	    d_{\theta_i}\coloneqq (m_{R_i}\otimes 1)\circ (1\otimes \theta_i\otimes 1)\circ (1\otimes \phi_M).
	    $$
	    One can visualize it as
	    $$
	    \xymatrix{R_i\otimes M \ar[r]^(.42){1\otimes \phi_M} & R_i\otimes \Lambda_i \otimes M \ar[r]^(.47){1\otimes \theta_i \otimes 1} & R_i\otimes R_i\otimes M \ar[r]^(.58){m_{R_i}\otimes 1}& R_i\otimes M.}
	    $$
	    Since $\deg_z(\theta_i(x))=1$ for any $x\in \Lambda_i$ it follows that $d_{\theta_i}$ shifts the grading $\deg_z$ by 1. By \Cref{compute-cotor}(a) below, we have $d_{\theta_i}^2=0$ and so  we can consider a complex 
	    $$
	    (R_i\otimes_{\theta_i}M,d_{\theta_i})\coloneqq 
		\ldots \to (R_i\otimes M)^{\deg_z=0} \xrightarrow{d_{\theta_i}} (R_i\otimes M)^{\deg_z=1}\xrightarrow{d_{\theta_i}} (R_i\otimes M)^{\deg_z=2}\xrightarrow{d_{\theta_i}} \ldots 
		$$
	Moreover, $d_{\theta_i}$ preserves the natural\footnote{Namely, both $R_i$ and $M$ are $\Z^2$-graded, and we consider the natural $\Z^2$-grading on the tensor product.} $\Z^2$-grading on $R_i\otimes M$. This gives a well-defined (extra) $\mbb Z^2$-grading on the complex $(R_i\otimes_{\theta_i}M,d_{\theta_i})$.
	\end{constr}
	
	\begin{rmk}
	    Similarly, if $M$ is a $\Z$-graded comodule over $\Lambda_i^{\mr{tot}}$ the twisted tensor product $(R_i\otimes_{\theta_i} M,d_{\theta_i})$ has a natural $\Z$-grading.
	\end{rmk}
	
	\begin{rmk}\label{rem:terms of the twisted complex}
	    Note that $(R_i\otimes M)^{\deg_z=k}\simeq (R_i)^{\deg_z=k}\otimes M$, so the complex $(R_i\otimes_{\theta_i}M,d_{\theta_i})$ can also be rewritten as 
	    $$
	    \ldots \to (R_i)^{\deg_z=0}\otimes M \xrightarrow{d_{\theta_i}} (R_i)^{\deg_z=1}\otimes M\xrightarrow{d_{\theta_i}} (R_i)^{\deg_z=2}\otimes M\xrightarrow{d_{\theta_i}} \ldots 
	    $$
	\end{rmk}

	The following lemma shows that the twisted tensor product in fact computes $\on{Cotor}$:
	
	\begin{lemma}\label{compute-cotor}
	   Let $i\in\set{1,2}$.
		\begin{enumerate}[label=(\alph*)]
		    \item We have $d_{\theta_i}^2=0$, and so the complex $(R_i\otimes_{\theta_i}M,d_{\theta_i})$ from \Cref{constr:twisted tensor product} is well-defined.
		    \item Let $M=\Lambda_i$. Then \[H^*(R_i\otimes_{\theta_i}\Lambda_i,d_{\theta_i})=\F_2,\]
		    where the $\Z^2$-grading on the right hand side is $(0,0)$.
		    \item Let $M$ be any left $\Z^2$-graded $\Lambda_i$-comodule. Then there is an isomorphism of $\Z^3$-graded vector spaces \[H^*(R_i\otimes_{\theta_i} M,d_{\theta_i})\simeq \on{Cotor}^*_{\Lambda_i}(\F_2, M).\]
		\end{enumerate}	 
	\end{lemma}
	\begin{proof} By a simple diagram chase the operator $d_{\theta_i}^2$ can be identified with the composition
	$$
\xymatrix{R_i\otimes M \ar[r]^(.45){1\otimes \phi_M} & R_i\otimes \Lambda_i \otimes M \ar[r]^(.45){1\otimes \Delta_i \otimes 1} & R_i\otimes \Lambda_i \otimes \Lambda_i \otimes M\ar[d]^{1\otimes \theta_i\otimes \theta_i\otimes 1}\\
	R_i\otimes M &\ar[l]_(.55){\mu_i\otimes 1} R_i\otimes R_i \otimes M &\ar[l]_(.55){1\otimes \mu_i \otimes 1} R_i\otimes R_i \otimes R_i \otimes M.}
	$$
	However, the composition of three middle arrows is 0 by \Cref{lem:thetas are twisted cochains}. Thus $d_{\theta_i}^2=0$ and we get a).

	 For each $n$ consider a Koszul complex $K_n^*$ given as a DG-algebra by $\mbb F_2[t_n, y_{n+1}]/(t_n^2)$ with $|t_n|=0$ and $|y_{n+1}|=1$,  and with differential $d$ determined by $d(t_n)=y_{n+1}$ and $d(y_{n+1})=0$. A direct easy computation shows that $K_n^*$ is quasi-isomorphic to the base field $\mbb F_2$ in degree 0 (considered as a DG-algebra with trivial differential). Consider the infinite tensor products \[K_{\infty,1}^*\coloneqq \otimes_{h=0}^\infty K_{2^{h+1}}^*,\qquad K_{\infty,2}^*\coloneqq \otimes_{h=0}^\infty K_{2^{h}}^*.\] Both of these complexes are quasi-isomorphic to $\mbb F_2$ in degree 0. 
		We now will deform $R_i\otimes_{\theta_i}\Lambda_i$ to $K_{\infty,i}^*$. Namely, let us define a ``weight" function $w$ on $R_i\otimes_{\theta_i}\Lambda_i$ by putting $w(x_2^{2^h})=w(z_{2^h+1})=1$, and then extending to monomials by multiplicativity in the following sense: any power $x_2^n$ can be uniquely represented as a product $x_2^{2^{h_1}}\cdot \ldots \cdot x_2^{2^{h_k}}$ where all $h_i$ are distinct, we then put $w(x_2^n)\coloneqq k$; for a monomial in $z_{2^h+1}$ weight is just given by the degree. Consider a filtration $R_i\otimes_{\theta_i}\Lambda_i$, whose $r$-th filtered piece is given by $x\in R_i\otimes\Lambda_i$ with $w(x)\leq r$. The differential $d_{\theta_i}$ preserves the filtration, and, moreover, the differential on the associated graded now satisfies the Leibniz's rule (which we leave to the reader to check). We can now define natural maps \[\iota_1\colon K_{\infty,1} \to \gr^*(R_1\otimes_{\theta_1}\Lambda_1,d_{\theta_1}),\qquad \iota_2\colon K_{\infty,2} \to \gr^* (R_2\otimes_{\theta_2}\Lambda_2,d_{\theta_2})\] of DG-algebras by sending $y_{2^h+1}$ to $z_{2^h +1}$,  $t_{2^h}$ to $x_2^{2^h}$ for $h\ge 1$ and $t_1$ to $x_2$, which are easily checked to be isomorphisms of complexes. In particular, the associated graded in both cases is quasi-isomorphic to $\mbb F_2$.
		Note that the filtration by weight is exhaustive because $(R_i\otimes_{\theta_i}\Lambda_i,d_{\theta_i})$ is the colimit of subcomplexes $(R_i\otimes_{\theta_i}\Lambda_i)_{\le w}$ of weight $\le w$ where maps $(R_i\otimes_{\theta_i}\Lambda_i)_{\le w}\to (R_i\otimes_{\theta_i}\Lambda_i)_{\le w+1}$ are termwise embeddings of complexes, and thus is equivalent to the corresponding homotopy colimit. Thus, since the associated graded is quasi-isomorphic to $\mbb F_2$, so is $R_i\otimes_{\theta_i}\Lambda_i$. This gives b).

	For c), note that the construction $M\mapsto (R_i\otimes_{\theta_i}M,d_{\theta_i})$ is functorial in $M$. The comultiplication $\Delta_{\Lambda_i}\colon \Lambda_i \to \Lambda_i\otimes \Lambda_i$ is a map of left $\Lambda_i$-modules, where the coaction on $\Lambda_i\otimes \Lambda_i$ is through the left factor. Since the action is through the left factor, one can move the right factor $\Lambda_i$ out of the twisted tensor product: $$(R_i\otimes_{\theta_i}(\Lambda_i\otimes \Lambda_i),d_{\theta_i})\simeq R_i\otimes_{\theta_i}\Lambda_i,d_{\theta_i})\otimes \Lambda_i.$$ Thus we obtain a map of complexes $$(R_i\otimes_{\theta_i}\Lambda_i,d_{\theta_i})\to (R_i\otimes_{\theta_i}(\Lambda_i\otimes \Lambda_i),d_{\theta_i})\otimes \Lambda_i,$$ which in fact defines a right coaction of $\Lambda_i$ on $(R_i\otimes_{\theta_i}\Lambda_i,d_{\theta_i})$. Note that the terms of $(R_i\otimes_{\theta_i}\Lambda_i,d_{\theta_i})$ are of the form $(R_i)^{\deg_z=k}\otimes \Lambda_i$ and thus are free right $\Lambda_i$-comodules; in particular, they are injective. By b) we get that $(R_i\otimes_{\theta_i}\Lambda_i,d_{\theta_i})$ is an injective resolution of the trivial right $\Lambda_i$-comodule $\mbb F_2$.
	
	Finally, from the definiton of the cotensor product and \Cref{constr:twisted tensor product} one sees that there is a natural equivalence $(R_i\otimes_{\theta_i}M,d_{\theta_i}) \simeq (R_i \otimes_{\theta_i}\Lambda_i,d_{\theta_i})\square_{\Lambda_i}M$. Thus (e.g. by \cite[Lemma A1.2.2]{ravenel1986complex}) the ($\Z^2$-graded) complex $(R\otimes_{\theta_i}M,d_{\theta_i})$ computes the right derived functor of $\mbb F_2\square_{\Lambda_i}-$ applied to $M$, whose cohomology is $\on{Cotor}^*_{\Lambda_i}(\F_2, M)$.
	\end{proof}
	
		\begin{rmk}
	    It is not hard to see directly from the definition that $(R_i\otimes_{\theta_i}\F_2,d_{\theta_i})$ is isomorphic to $R_i$ with zero differential.  It follows from \Cref{compute-cotor}(c) that $R_i\simeq \on{Cotor}^*_{\Lambda_i}(\mbb F_2,\mbb F_2)$, where the $*$-grading corresponds to $\deg_z$, while the natural $\Z^2$-grading on $\on{Cotor}$ agrees with the one in \Cref{constr:twisted cochains}. 
	\end{rmk}

	\begin{rmk}\label{not-algebra}
		We would like to emphasize that when $M$ is a comodule algebra, even though the underlying graded module of $R_i\otimes_{\theta_i} M$ has a natural algebra structure, the differential $d_{\theta_i}$ does not satisfy the Leibnitz rule. In particular, a priori there is no clear way to see an algebra structure on $\on{Cotor}^*_{\Lambda_i}(\mbb F_2,M)$ via $R_i\otimes_{\theta_i} M$.
	\end{rmk}

	Here is an observation which partially remedies the problem mentioned in \Cref{not-algebra}. Let $i\in\set{1,2}$ and $A$ be a ($\Z^2$-graded) comodule $\Lambda_i$-algebra, and consider a commutative square of graded vector space homomorphisms
	\begin{equation}\label{compatib}
		\xymatrix{
			R_i\otimes PA \ar[r] \ar[d]& H^*(R_i\otimes_{\theta_i}A,d_{\theta_i}) \ar[d]^\wr \\
			\on{Cotor}^*_{\Lambda_i}(\F_2,\F_2)\otimes \on{Cotor}_{\Lambda_i}^*(\F_2,A) \ar[r] & \on{Cotor}^*_{\Lambda_i}(\F_2,A).
		}
	\end{equation}
	Here, the top horizontal map is the composition 
	\[\eta:R_i\otimes PA \simeq H^*(R_i\otimes_{\theta_i} PA,d_{\theta_i})\tto H^*(R_i\otimes_{\theta_i} A,d_{\theta_i}),\] 
	where the identification comes from the fact that $d_{\theta_i}$ is zero on $R_i\otimes PA$ (as is easily checked from the definition) and the second map is given by functoriality. The bottom horizontal map is induced by the external cup product of (\cite[Definition A1.2.13]{ravenel1986complex}). The vertical map on the right comes from \Cref{compute-cotor}(c), while the left vertical map comes from the isomorphism $PA\coloneqq \on{Cotor}_{\Lambda_i}^0(\F_2,A)$ (which gives an embedding $PA\to \on{Cotor}^*_{\Lambda_i}(\mbb F_2,A)$) and the identifiaction $R_i\simeq \on{Cotor}^*_{\Lambda_i}(\F_2,\F_2)$ (\Cref{compute-cotor}(b)).
	
	\begin{lemma}\label{compatible}
		For every $i\in\set{1,2}$ and every comodule $\Lambda_i$-algebra $A$, the square (\ref{compatib}) is commutative.
	\end{lemma}
	
	\begin{proof}
	Consider the resolution $\mbb F_2\to I_1^*\coloneqq R_i\otimes_{\F_2} \Lambda_i$ provided by \Cref{compute-cotor}(c), and let $\F_2\to I_2$ be any other injective resolution as a right $\Lambda_i$-module. Following \cite[Definition A.1.2.13]{ravenel1986complex}, the multiplication $ \on{Cotor}^*_{\Lambda_i}(\F_2,\F_2)\otimes\on{Cotor}_{\Lambda_i}^*(\F_2,A)\to \on{Cotor}_{\Lambda_i}^*(\F_2,A)$ is induced by the map of complexes 
	$$
	(I_1^*\square_{\Lambda_i}\mbb F_2)\otimes_{\mbb F_2} (I_2^*\square_{\Lambda_i}A)\tto (I_1^*\otimes_{\mbb F_2} I_2^*)\square_{\Lambda_i} A
	$$
	by passing to cohomology. Note that $PA\coloneqq\on{Cotor}_{\Lambda_i}^0(\F_2,A)$ is the 0-th cohomology of $I_2^*\square_{\Lambda_i}A$ and embedds to the latter as $\mbb F_2\otimes PA\subset I_2^0\square A$. We get a commutative diagram 
	$$
	\xymatrix{
			(I_1^*\square_{\Lambda_i}\mbb F_2)\otimes_{\F_2}(\mbb F_2\otimes PA) \ar[r] \ar[d] & (I_1^*\otimes_{\mbb F_2} \mbb F_2)\square_{\Lambda_i}A  \ar[d]\\
			(I_1^*\square_{\Lambda_i}\mbb F_2)\otimes_{\mbb F_2} (I_2^*\square_{\Lambda_i}A) \ar[r]& (I_1^*\otimes_{\mbb F_2} I_2^*)\square_{\Lambda_i} A.
		}
	$$
which gives the desired commutative diagram 	
$$
		\xymatrix{
			R_i\otimes PA \ar[r] \ar[d]& H^*(R_i\otimes_{\theta_i}A,d_{\theta_i}) \ar[d]^\wr \\
			\on{Cotor}^*_{\Lambda_i}(\F_2,\F_2)\otimes \on{Cotor}_{\Lambda_i}^*(\F_2,A) \ar[r] & \on{Cotor}^*_{\Lambda_i}(\F_2,A).
		}
$$
by applying $H^*$.
	\end{proof}
	
	It is important to note that, while by \Cref{not-algebra} the square (\ref{compatib}) is not a diagram of algebras, 
	\begin{equation}\label{algebra-map}
		\text{The map $PA\otimes R_i\tto \on{Cotor}^*_{\Lambda_i}(\F_2,A)$ of (\ref{compatib})  is  an algebra homomorphism.}
	\end{equation}
	
	

	\subsection{Decomposition of a comodule algebra}\label{ssect:subalgebra_like_in_Toda}
	Let $i\in\set{1,2}$. Let $A$ be a $\Z^2$-graded left $\Lambda_i$-comodule algebra and $\phi\colon A\to \Lambda_i\otimes A$ be the corresponding coaction map.
	
		Assume first that $i=1$. Recall that $\Lambda_1\simeq \mbb F_2[x_2]$, where $x_2$ is primitive and $|x_2|=(1,1)$.
		
\begin{constr}\label{const:d in the G_m case}
    For all $j\geq 0$, define linear maps $d_j\colon A\to A$ of bidegree $(-j,-j)$ by
	\[\phi(a)=\sum_{j\ge 0} x_2^j\otimes d_j(a).\]
	
	For all $a,b\in A$, we have
	\begin{equation}\label{eq:relations on d_i in G_m}
	d_0(a)=a,\qquad d_h(ab)=\sum_{i+j=h}d_i(a)d_j(b),\qquad d_id_j(a)=\binom{i+j}{i}d_{i+j}(a)
	\end{equation}
\end{constr}

\begin{example}\label{ex:d_i's for GL_n}
    Let us view $\Lambda_1$ as $H^{*,*}_\Hdg(B\mbb G_m/\mbb F_2)$ and consider the comodule algebra given by $H^{*,*}_\Hdg(B\mr{GL}_n/\mbb F_2)\simeq \mbb F_2[c_1,\ldots, c_n]$, as in \Cref{ssect:projective}. Then from \Cref{comparison}(a) one sees that $d_i(c_j)= \binom{n-j+i}{i} c_{j-i}$, where we put $c_0\coloneqq 1$.
\end{example}

One has the following lemma by Toda, which allows to decompose $A$ into two simpler parts, provided that there is an element $a_\sharp\in A$ satisfying a special condition.
	\begin{lemma}\label{toda36}
		Let $q$ be a power of $2$, and let $a_{\sharp}\in A$ be a bihomogeneous element such that \[d_q(a_{\sharp})=1,\qquad d_h(a_{\sharp})=0 \text{ for all $h> q$.}\] Consider the linear subspace
		\[P_qA\coloneqq\set{a\in A: d_i(a)\text{ for all $i\geq q$}}\subset A.\]
		Then the map $\F_2[x]\otimes P_qA\to A$ with $|x|\coloneqq |a_{\sharp}|=2q$ given by $f(x)\otimes b\mapsto f(a_{\sharp})b$ is a graded linear isomorphism. Consequently, the projection $A\twoheadrightarrow A/(a_\sharp)$ induces an additive isomorphism $P_qA\simeq A/(a_\sharp)$.
	\end{lemma}
	
	\begin{proof}
		This is \cite[Lemma 3.7]{toda1987cohomology}. Even though $A$ is $\Z$-graded in \textit{loc.cit.}, the same proof works in the bigraded setting. Our $P_qA$ is denoted there by $B$.
	\end{proof}
\begin{rmk}
    The main point of \Cref{toda36} is that it provides a unique lift of an element in $A/(a_\sharp)$ to an element in $P_qA\subset A$ (in other words, to an element of $A$ on which $d_i$'s act by 0 for $i\ge q$). 
\end{rmk}
	
	\begin{constr}\label{con:star-product}
	    The additive isomorphism $P_qA\simeq A/(a_\sharp)$ induces a new multiplication $*$ on $P_qA$ by setting
	\begin{equation}\label{asterisk}
		\text{$a*a'\coloneqq a''$ if $aa'\equiv a''\pmod{a_{\sharp}}$.}
	\end{equation}
	\end{constr}
	
\begin{rmk}\label{formula for star-product}
    In the case $q=2$, following \cite[p. 92]{toda1987cohomology} there is also an explicit formula for the $*$-product:
    \begin{equation}\label{eq:formula for *product}
    a*a'= a\cdot a' + d_1(a)\cdot d_1(a')\cdot a_\sharp.
    \end{equation}
    Indeed, $a*a'\equiv aa'\pmod{a_{\sharp}}$, so it is just enough to check that the right hand side lies in $P_2A$. From the relations in (\ref{eq:relations on d_i in G_m}) it is clear that $d_i(a*a')=0$ for $i>2$, while
    $$
    d_2(a*a')=d_2( a\cdot a')+ d_2(d_1(a)\cdot d_1(a')\cdot a_\sharp)= d_1(a)\cdot d_1(a') + d_1(a)\cdot d_1(a')\cdot 1 =0.
    $$
   It also follows from the formula above that if one of the $a$ or $a'$ belongs to $PA\subset P_2A$, then $a*a'=a\cdot a'$.

\end{rmk}	
	
\begin{example}\label{ex:c_q ia asharp}
    Continuing \Cref{ex:d_i's for GL_n}, let $n=qk$ where $q$ is a power of $2$ and $k$ is odd. Then 
    $$d_q(c_q)=\binom{qk}{q} \cdot c_0 = 1.$$ 
    In particular \Cref{toda36} applies to $A=H^{*,*}_\Hdg(B\mr{GL}_n/\mbb F_2)\simeq \mbb F_2[c_1,\ldots, c_n]$ where we take $a_\sharp$ to be $c_q$.
\end{example}

\begin{rmk}\label{rem:d_1 as a differential on P_2}
 It is immediate from the formula for $d_id_j$ in \Cref{const:d in the G_m case} that $d_1\colon A\to A$ preserves $P_qA$ and that $d_1^2=0$. Thus $d_1$ can be considered as a differential on $A$ as well as any of the subalgebras $P_qA\subset A$. 
 
 Note, however, that $d_1$ doesn't satisfy the Leibnitz rule with respect to the $*$-multiplication on $P_qA$. Instead, one has the formula
 $$
 d_1(a*b)=d_1(a)*b+a*d_1(b) + d_1(a)\cdot d_1(b)\cdot a_1.
 $$
It follows that $d_1$ is at least $PA$-linear.
\end{rmk}

\begin{rmk}
    \Cref{toda36} also works in $\Z$-graded setting, namely when we consider $\Lambda_1^{\mr{tot}}$ and a $\Z$-graded $\Lambda_1^{\mr{tot}}$-comodule algebra $A$.
\end{rmk}

\medskip 

Assume now that $i=2$. Recall that $\Lambda_2\simeq \mbb F_2[x_1,x_2]/(x_1^2)$, where both $x_i$ are primitive and $|x_1|=(1,0)$, $|x_2|=(1,1)$.
		
\begin{constr}\label{condtr:d_i for so_n}
    For all $j\geq 0$, define maps $d_{2j},d_{2j+1}\colon A\to A$ of bidegrees $(-j,-j)$ and $(-j-1,-j)$ by
		\[\phi(a)=\sum_{j\ge 0} x_2^{j}\otimes d_{2j}(a)+\sum_{j\ge 0} x_1x_2^j\otimes d_{2j+1}(a).\]
Again, we have 
\begin{equation}\label{eq:composition of d_i for mu_2}
d_0(a)=a\qquad \text{and} \qquad d_id_j(a)=\binom{i+j}{i}d_{i+j}(a).
\end{equation}
The formula for $d_i(ab)$ is, however, different from \Cref{const:d in the G_m case}, due to the fact that $x_1^2=0\in\Lambda_2$:
namely, for all $a,b\in A$, we still have
	\[d_{2h+1}(ab)=\sum_{i+j=2h+1}d_i(a)d_j(b),\]
	but 
	\begin{equation}\label{eq:multiplicativity for mu_2}
	  d_{2h}(ab)=\sum_{i+j=h}d_{2i}(a)d_{2j}(b),  
	\end{equation}
	with the terms in the sum only having even indices. 
\end{constr}

Here is a version of Toda's lemma for $\Lambda_2$-comodule algebras.
	\begin{lemma}\label{toda36-new}
		Let $q>1$ be a power of $2$, and let $a_{\sharp}\in A$ be a bihomogeneous element such that \[d_q(a_{\sharp})=1,\qquad d_h(a_{\sharp})=0 \text{ for all $h> q$.}\] 
		(a) Consider the linear subspace
		\[P_{q}A\coloneqq\set{a\in A: d_i(a)=0\text{ for all $i\geq q$}}.\]
		Consider the polynomial ring $\F_2[x]$, where $|x|\coloneqq|a_{\sharp}|=(q/2,q/2)$. Then the map $\F_2[x]\otimes P_{q}A\to A$ given by $f(x)\otimes b\mapsto f(a_{\sharp})b$ is an $\F_2$-linear isomorphism.
		
		(b) If $q=2$, then $P_2A$ is a subring of $A$ and the map $\F_2[x]\otimes P_{2}A\to A$ in (a) is a ring isomorphism.
		
	\end{lemma}
	\begin{rmk}
	Let us point out that there is no analogue of \Cref{toda36-new}(b) in the setting of \Cref{toda36}.
	\end{rmk}
	\begin{rmk}
    As with \Cref{toda36}, \Cref{toda36-new} also works in $\Z$-graded setting, when we consider $\Lambda_2^{\mr{tot}}$ and a $\Z$-graded $\Lambda_2^{\mr{tot}}$-comodule algebra $A$.
\end{rmk}

	\begin{proof}
		(a) The proof is analogous to that of \cite[Lemma 3.6]{toda1987cohomology}. To prove injectivity, it suffices to show the following: if $b_1,\dots,b_r\in P_{q}A$ are bihomogeneous and $\sum_{i=0}^ra_{\sharp}^ib_i=0$, then $b_1=\dots=b_r=0$. To see this, suppose by contradiction that $b_r\neq 0$, and let $h\geq 0$ be such that $d_h(b_r)\neq 0$ and $d_{h'}(b_r)=0$ for all $h'> h$. (Such $h$ exists, because $d_0(b_r)=b_r$.) Note that by our assumptions on $a_\sharp$ and multiplicativity (\ref{eq:multiplicativity for mu_2}) we have $d_{rq}(a_{\sharp}^r)=d_{q}(a_{\sharp})^r=1$. We then also have
		\[0=d_{rq+h}(\sum_{i=0}^ra_{\sharp}^ib_i)=d_{rq}(a_{\sharp}^r)d_h(b_r)=d_h(b_r)\neq 0,\]
		a contradiction. 
		
		We now prove surjectivity. It suffices to show that every bihomogeneous element $a\neq 0$ of $A$ is in the image. Let $h$ be the maximal integer such that $d_h(a)\neq 0$, and write $h=qj+h'$, where $0\leq h'<q$. For all $s> h'$, $d_sd_{qj}(a)$ is a multiple of $d_{s+qj}(a)$, hence zero. Note that this means that $d_{qj}(a)\in P_qA$. Since $q$ is a power of 2 and $h'<q$, e.g. by Lucas's formula we have $\binom{qj+h'}{ h'}=1$, and so by (\ref{eq:composition of d_i for mu_2}) we have $d_{h'}d_{qj}(a)=d_h(a)$. Setting $b\coloneqq d_{qj}(a)\in P_qA$ and $a'\coloneqq a-a_{\sharp}^jb$, then for all $i\geq 0$ we have
		\[d_i(a')=d_i(a)-\sum_{s=0}^{h'}d_{i-h'}(a_{\sharp}^j)d_s(b).\]
		We now show that $d_i(a')=0$ for all $i\geq h$. Indeed, if $i>h$, then $i-s>qj$ for all $0\leq s\leq h'$ and so, due to the multiplicativity of $\phi$, we have $d_{i-s}(a_{\sharp}^j)=0$. It follows that $d_i(a')=d_i(a)=0$ for all $i>h$. On the other hand, when $i=h$, then \[d_h(a')=d_h(a)-d_{qj}(a^j_{\sharp})d_{h'}(b)=d_h(a)-(d_q(a_{\sharp}))^jd_h(a)=d_h(a)-d_h(a)=0.\]
		Thus $a=a_{\sharp}^jb+a'$, where $d_i(a')=0$ for all $i\geq h$. Since $a$, $a_{\sharp}$ and $b$ are bihomogeneous, $a_{\sharp}^jb$ and $a'$ are homogeneous of degree $|a|$. Therefore, applying the same reasoning to $a'$ and iterating, we eventually write $a$ as $\sum_{i=0}^ra_{\sharp}^ib_i$ for some $b_i\in B$.
		
		(b) Let $a,b\in P_2A$. Then $\phi(a)=a+d_1(a)x_1$ and $\phi(b)=b+d_1(b)x_1$. Since $x_1^2=0$, we obtain $\phi(ab)=\phi(a)\phi(b)=ab+(d_1(a)+d_1(b))x_1$, that is, $ab\in P_2A$. It follows that $P_2A$ is a subring of $A$. It remains to note that the explicit formula for isomorphism in (a) is obviously multiplicative.
 	\end{proof}

	\begin{rmk}\label{rem:d_1 is a differential in O_n case}
	   It follows from \Cref{eq:composition of d_i for mu_2} that $d_1^2=0$ and that $d_1$ preserves $P_qA$. Thus, as in \Cref{rem:d_1 as a differential on P_2}, $d_1$ defines a differential on both $A$ and $P_qA\subset A$.
	\end{rmk}

	\begin{rmk}
	    	   We also get that the differential $d_1$ satisfies the following cyclic equation: 
	   \begin{equation}\label{eq:cyclic property for d_1}
	       d_1(ab)d_1(c)+d_1(bc)d_1(a)+d_1(ac)d_1(b)=0.
	   \end{equation} 
	   Indeed, the latter expression is nothing but 
	   $
	   d_1(d_1(abc))
	   $, which is equal to 0.
	\end{rmk}
	
	\section{Hodge cohomology of \texorpdfstring{$B\mr{PGL}_{4m+2}$}{BPGL}} \label{sec: PGL}

	The goal of this section is to prove \Cref{thm-hodge-pgl}, which gives an explicit description of $H^{*,*}_{\mr H}(B\mr{PGL}_{4m+2}/\F_2)$ and $H^*_{\dR}(B\mr{PGL}_{4m+2}/\F_2)$ as bigraded and graded algebras, respectively. It also implies the full version of \Cref{mainthm}(1) for $\cl{G}=\mr{PGL}_{4m+2}$. For the proof we use the Eilenberg-Moore spectral sequence (\Cref{eilenberg-moore-easy}) and the strategy devised by Toda \cite{toda1987cohomology} in the topological setting.
	
	Consider $H^{*,*}_{\on{H}}(B\mr{GL}_{4m+2}/\F_2)$ as a $H^{*,*}_{\on{H}}(B\mbb G_m/\F_2)$-comodule algebra. By \Cref{ex:c_q ia asharp}, $c_2\in H^{1,1}_{\on{H}}(B\mr{GL}_{4m+2}/\F_2)$ can be taken as $a_\sharp$, and \Cref{toda36} applies.

	We will start by identifying the subalgebras 
	$$
	PH^{*,*}_{\on{H}}(B\mr{GL}_{4m+2}/\F_2)\subset  P_2H^{*,*}_{\on{H}}(B\mr{GL}_{4m+2}/\F_2)\subset H^{*,*}_{\on{H}}(B\mr{GL}_{4m+2}/\F_2).$$
	Recall that the operation $d_1$ (defined in \Cref{const:d in the G_m case}) induces a $PH^{*,*}_{\on{H}}(B\mr{GL}_{4m+2}/\F_2)$-linear differential on $P_2H^{*,*}_{\on{H}}(B\mr{GL}_{4m+2}/\F_2)$ (\Cref{rem:d_1 as a differential on P_2}).
	\begin{lemma}\label{toda36-311}
		Assume that $n=4m+2$.
		
		\smallskip \begin{enumerate}[label=(\alph*)]
		    \item 
	
		There exists a unique sequence \[\cl{c}_1,\dots,\cl{c}_{4m+2}\in H^{*,*}_{\on{H}}(B\mr{GL}_{4m+2}/\F_2)\] such that $\cl{c}_i$ has bidegree $(i,i)$ and
		\begin{enumerate}[label=(\arabic*)]
		\item $\cl{c}_i$ are polynomial generators: $H^{*,*}_{\on{H}}(B\mr{GL}_{4m+2}/\F_2)=\F_2[\cl{c}_1,\cl{c}_2\dots,\cl{c}_{4m+2}]$;
			\item $\cl{c}_1\coloneqq c_1$, $\cl{c}_2\coloneqq c_2$;
			\item for all $j>1$, we have 
			$$\cl{c}_{2j}\equiv c_{2j}\pmod{c_2} \quad\text{and}\quad  \cl{c}_{2j} \in P_2H^{j,j}_{\on{H}}(B\mr{GL}_{4m+2}/\F_2)$$
			(so $d_i(\cl{c}_{2j})=0$ for $i\geq 2$).
			\item $\cl{c}_{2j-1}= d_1(\cl{c}_{2j})$; 
			
			\noindent one automatically has that $\cl{c}_{2j-1}$ is primitive: $\cl{c}_{2j-1}\in PH^{*,*}_{\on{H}}(B\mr{GL}_{4m+2}/\F_2)$.

		\end{enumerate}  
\smallskip 
		\item  \begin{enumerate}[label=(\arabic*)]
		    \item The subalgebra $P_2H^{*,*}_{\on{H}}(B\mr{GL}_{4m+2}/\F_2)\subset H^{*,*}_{\on{H}}(B\mr{GL}_{4m+2}/\F_2)$ is  freely generated by $\cl{c}_1,\cl{c}_3,\cl{c}_4,\cl{c}_5,\dots,\cl{c}_{4m+2}$ under the $*$-multiplication (see \Cref{con:star-product});
		    \item For all $1< h\leq 2m+1$ define elements
		\[b_h\coloneqq\cl{c}_{2h}*\cl{c}_{2h}+c_1  \cl c_{2h} \cl c_{2h-1}\in H^{4h,4h}_{\on{H}}(B\mr{GL}_n/\F_2).\]
		Then $b_h$ are primitive (in particular, $b_h \in P_2H^{*,*}_{\on{H}}$ and $d_1(b_h)=0$) and the natural map 
		$$
		\F_2[c_1, b_2,\ldots, b_{2m+1}] \tto H^*(P_2H^{*,*}_{\on{H}}(B\mr{GL}_{4m+2}/\F_2),d_1)
		$$ is an isomorphism.
		\end{enumerate}
		\smallskip 
		
	    \item For any (unordered) tuple of integers $I=\{i_1,\ldots, i_r\}$ write $l(I)\coloneqq r$ and $d(I)\coloneqq \sum_{k=1}^r i_k$. Define
		\[y_I\coloneqq d_1(\cl{c}_{2i_1}\!*\ldots*\cl{c}_{2i_r})\in PH^{2d(I)-1,2d(I)-1}_{\on{H}}(B\mr{GL}_n/\F_2).\]
		In particular, $y_{\{i\}}=\cl c_{2i-1}$. 
		
		\begin{enumerate}[label=(\arabic*)]
		    \item The subalgebra $PH^{*,*}_{\on{H}}(B\mr{GL}_n/\F_2)\subset H^{*,*}_{\on{H}}(B\mr{GL}_n/\F_2)$ of primitive elements is generated by $c_1$, $b_h$ and $y_I$ for $I=\{1<i_1<\ldots<i_r \le 2m+1\}$, and the relations given by
		\begin{equation*}
			y_Iy_J=\sum_{\emptyset\neq K\subset I}y_{(I-K)\cup J}y_{\{k_1\}}\ldots y_{\{k_s\}}c_1^{l(K)-1},
		\end{equation*}
		for all subsets $I,J\subset \{2,\ldots, 2m+1\}$ and where we put
		\begin{equation*}
			y_{\{h,h,j_1,\ldots,j_s\}}\coloneqq y_{\{j_1,\ldots,j_s\}}b_h+y_{\{h,j_1,\ldots,j_s\}}y_{\{h\}}c_1.
		\end{equation*}
		\item $PH^{*,*}_{\on{H}}(B\mr{GL}_n/\F_2)$ is a finitely generated module over the polynomial subalgebra 
		$$\mbb F_2[c_1,\cl{c}_3,\cl{c}_5,\ldots, \cl{c}_{4m-1}, b_2,b_3,\ldots,b_{2m+1}]\subset PH^{*,*}_{\on{H}}(B\mr{GL}_n/\F_2).$$
		\end{enumerate}
		\end{enumerate}
		
		Analogous statements hold with Hodge cohomology replaced by de Rham cohomology. (In this case $\cl{c}_i, b_h, y_I\in H^*_{\on{dR}}$ have degrees $2i$, $8h$ and $4d(I)-2$,  respectively.)	
	\end{lemma}
	
	\begin{proof}It is enough to show the corresponding statement for the corresponding $\Z$-graded algebra $H^*_\Hdg(B\mr{GL}_n/\F_2)$.
		By \Cref{comparison}(b), it then suffices to prove the analogous assertions in the topological setting. Thus (a) and (b1) follow from \cite[Proposition 3.7]{toda1987cohomology}, (b2) is given by \cite[Lemma 3.10(ii)]{toda1987cohomology}, and (c) follows from \cite[Proposition 3.11]{toda1987cohomology}. 
		
		It remains to explain the bidegrees of the elements. The $\cl{c}_i$ are constructed from the $c_i$ by applying \Cref{toda36}. In particular, since the isomorphism of \Cref{toda36} respects the grading, $\cl{c}_i$ and $c_i$ have the same bidegree $(i,i)$. The $*$-product preserves bigrading and so we get $b_h$ is homogeneous of bidegree $(4h,4h)$. Similarly $d_1$ reduces the bigrading by $(1,1)$ and so we get $|y_I|=(2d(I)-1,2d(I)-1)$. Everything works similarly in the de Rham setting where the degrees of $\cl c_i$, $b_h$ and $y_I$ are given by $2i$, $8h$ and $4d(I)-2$, respectively.
	\end{proof}
	
	\begin{rmk}
	    Let us comment upon the logic behind the statements of \Cref{toda36-311} (and how the proof could go without appealing to \cite{toda1987cohomology}). Having \Cref{toda36}, it is more or less immediate that $P_2H^{*,*}_{\on{H}}(B\mr{GL}_n/\F_2)$ is isomorphic to the polynomial ring over $\cl c_1, \cl c_3, \cl c_4,\ldots, \cl c_{4m+2}$ via the $*$-multiplication. The subalgebra $PH^{*,*}_{\on{H}}(B\mr{GL}_n/\F_2)\subset P_2H^{*,*}_{\on{H}}(B\mr{GL}_n/\F_2)$ then is identified with the kernel of the differential $d_1$ on $P_2H^{*,*}_{\on{H}}(B\mr{GL}_n/\F_2)$. One can understand this kernel in two steps: first, by finding a subalgebra in $\on{Ker}(d_1)$ that maps isomorphically to the cohomology $H^*(P_2H^{*,*}_{\on{H}}(B\mr{GL}_n/\F_2),d_1)$ of $d_1$ --- this is given by $\F_2[c_1,b_2,\ldots,b_{2m+1}]$; second, by describing the image of $d_1$ --- the latter contains elements $\cl c_{2i-1}$, and the whole image is spanned by the remaining $y_I$'s (with $l(I)\ge 2$) over the polynomial algebra $\F_2[c_1,\cl c_3, \cl c_5,\ldots, \cl c_{4k+1}, b_2,\ldots,b_{2m+1}]$. 
	    
	    The reader can also look at the proof of the analogous statement for $B\mr{SO}_{4m+2}$ (\Cref{toda36-311-so}), which is slightly more natural due to the fact that the $*$-multiplication on $P_2H^{*,*}_\Hdg(B\mr{SO}_{4m+2}/\F_2)$ coincides with the usual one.
	\end{rmk}
	\begin{example}\label{ex:stuff on new c_i and b_h}
	Formulas for the elements $\cl c_i\in P_2H^{*,*}_{\on{H}}(B\mr{GL}_{4m+2}/\F_2) \subset H^*_{\on{H}}(B\mr{GL}_{4m+2}/\F_2)\simeq \mbb F_2[c_1,\ldots, c_{4m+2}]$ get complicated pretty fast as $i$ grows. Here are the first few of them: 
	\begin{align*}
	    \cl c_1= c_1 \qquad \cl c_2= c_2 \qquad \cl c_3=  c_3 + mc_1^3 & \qquad \cl c_4=  c_4 + m(c_2+c_1^2)c_2 \\
	    \cl c_5=  c_5 + c_4c_1 +c_3(c_2 +c_1^2)   \qquad \cl c_6 &=  c_6 + (c_4 +c_3c_1)c_2.
	\end{align*}
	Let us also record that 
	$$
	b_h=\cl{c}_{2h}*\cl{c}_{2h}+c_1 \cl c_{2h} \cl c_{2h-1}=\cl{c}_{2h}^2+\cl{c}_{2h-1}^2c_2 +c_1 \cl c_{2h}\cl c_{2h-1}
	$$
	(here we just used the explicit formula for $*$-product from \Cref{formula for star-product}). 
	\end{example}
	\smallskip 
	
	We can now compute the $E_2$ page of the Eilenberg-Moore spectral sequence. 
	\begin{lemma}\label{coaction-gl}
		Assume that $n=4m+2$, for some integer $m\geq 0$. Then \[\on{Cotor}^*_{H^{*,*}_{\on{H}}(B\G_{\on{m}}/\F_2)}(\F_2, H^{*,*}_{\on{H}}(B\mr{GL}_{4m+2}/\F_2))\] is isomorphic, as a $\Z^3$-graded algebra, to
		\[(1\otimes PH^{*,*}_{\on{H}}(B\mr{GL}_{4m+2}/\F_2))\oplus (z_3\F_2[z_3]\otimes \F_2[c_1,b_h]_{h=2}^{2m+1}),\]
		where the $b_h$ have been defined in \Cref{toda36-311}(b), and $|z_3|=(1,1)$.	
		Analogous assertion holds with Hodge cohomology replaced by de Rham cohomology (here $|z_3|=2$).
		
	\end{lemma}
	\begin{rmk}\label{rem:multiplication in cotor}
	    Here the algebra structure is induced by the surjective homomorphism 
	    $$
	    PH^{*,*}_{\on{H}}(B\mr{GL}_{4m+2}/\F_2)) \xymatrix{\ar@{->>}[r]&} H^*(P_2 H^{*,*}_{\on{H}}(B\mr{GL}_{4m+2}/\F_2), d_1)\simeq \F_2[c_1,b_h]_{h=2}^{2m+1},
	    $$
	    where the last isomorphism is given by \Cref{toda36-311}(b). Note that the elements $y_I$ from \Cref{toda36-311}(c) are defined as images under $d_1$ of $\cl c_{i_1}*\ldots *\cl c_{i_r}\in P_2H^{*,*}_\Hdg(B\mr{GL}_n/\mbb F_2)$, and so they map to 0 under the above projection. Consequently, in  $\on{Cotor}^*_{H^{*,*}_{\on{H}}(B\G_{\on{m}}/\F_2)}(\F_2, H^{*,*}_{\on{H}}(B\mr{GL}_{4m+2}/\F_2))$ we have  $y_Iz_3=0$ for any $I$.
	\end{rmk}
	\begin{proof}
It is enough to check the statement for the $\Z$-graded algebras $H^{*}_{\on{H}}$ underlying $H^{*,*}_{\on{H}}$ (via totalization). Recall \Cref{comparison}(b). 
		Recall also that we can compute $\on{Cotor}$ via the twisted tensor product (see \Cref{compute-cotor}(c)). This gives a commutative diagram of graded vector spaces
		\[
		\begin{tikzcd}
			PH^{*}_{\on{H}}(B\mr{GL}_n/\F_2)\otimes R_1 \arrow[r, "\sim"] \arrow[d] & PH^{*,*}_{\on{sing}}(B\mr{GL}_n(\C);\F_2)\otimes R_1 \arrow[d]  \\
			H^*(H^*_{\on{H}}(B\mr{GL}_n/\F_2)\otimes_{\theta_1} R_1) \arrow[r, "\sim"] \arrow[d,"\wr"]  & H^*(H^*_{\on{sing}}(B\mr{GL}_n(\C);\F_2)\otimes_{\theta_1} R_1) \arrow[d,"\wr"]   \\	
			\on{Cotor}^*_{H^*_{\on{H}}(B\G_{\on{m}}/\F_2)}(\F_2, H^*_{\on{H}}(B\mr{GL}_n/\F_2))  \arrow[r, "\sim"] & \on{Cotor}^*_{H^*_{\on{sing}}(B\C^{\times};\F_2)}(\F_2, H^*_{\on{sing}}(B\mr{GL}_n(\C);\F_2)),  
		\end{tikzcd}
		\]
		where the top square comes from (\ref{compatib}) and \Cref{comparison}(b), while the bottom square comes from \Cref{compute-cotor}(c).
		Here, the elements $z_i\in R_1$, $\cl c_i$, $b_h$ and $y_I$ from \Cref{toda36-311} are getting mapped to the analogous elements ($z_i$, $\cl a_i$, $b_h$ and $y_I$) in the notation of \cite[Section 3 and (4.7)]{toda1987cohomology}. The statement of the lemma then follows from the analogous description in \cite[Section 3 and (4.7)]{toda1987cohomology}. To see that the algebra structure above is the correct one, note the maps in the outer rectangle in the above diagram are homomorphisms of algebras, and, by the description in \cite[(4.7)]{toda1987cohomology}, the composition of the vertical maps on the right is surjective, 
		hence the same is true for the composition of the two vertical maps on the left. 
	\end{proof}
	
		

\begin{rmk}
For the reader's convenience, let us also sketch the idea behind Toda's computation on the topological side, especially since the original paper is somewhat brief. First of all, \cite[Theorem 4.1]{toda1987cohomology} identifies $H^*(H^*_{\on{sing}}(B\mr{GL}_n(\C);\F_2)\otimes_{\theta_1} R_1)$ with the cohomology of a subcomplex 
$$
C\coloneqq (P_2H^*_{\on{sing}}(B\mr{GL}_n(\C);\F_2)\otimes \mbb F_2[z_3], d_{\theta_1}|_{P_2H^*_{\on{sing}}\otimes \mbb F_2[z_3]})\subset (H^*_{\on{sing}}(B\mr{GL}_n(\C);\F_2)\otimes_{\theta_1} R_1, d_{\theta_1}).
$$ By the definition of $d_{\theta_1}$, its restriction to $C$ can be identified with $d_1\otimes z_3$. The cohomology of $C$ then can be identified with $H^*_{\on{sing}}(B\mr{GL}_n(\C);\F_2)$ plus direct sum of infinitely many copies of $H^*(P_2H^*_{\on{sing}}(B\mr{GL}_n(\C);\F_2), d_1)$ multiplied by powers of $z_3$. Using the topological analogue of \Cref{toda36-311}(b) to describe the latter one arrives at the description of $H^*(H^*_{\on{sing}}(B\mr{GL}_n(\C);\F_2)\otimes_{\theta_1} R_1)$ as in \Cref{coaction-gl}. Moreover, since $H^*(P_2H^*_{\on{sing}}(B\mr{GL}_n(\C);\F_2), d_1)$ is generated by classes in $PH^{*}_{\on{sing}}(B\mr{GL}_n(\C);\F_2)$ one gets that the map 
$$
PH^*_{\on{sing}}(B\mr{GL}_n(\C);\F_2)\otimes R_1 \tto \on{Cotor}^*_{H^*_{\on{sing}}(B\C^{\times};\F_2)}(\F_2, H^*_{\on{sing}}(B\mr{GL}_n(\C);\F_2))
$$
is a surjective homomorphism of algebras and one recovers the algebra structure on $\on{Cotor}$ as well. 
\end{rmk}

To start let us compute Hodge and de Rham cohomology of $B\mr{PGL}_n$ for all $n$ in degrees up to 3.
\begin{lemma}\label{pgl-small}
We have isomorphisms \begin{enumerate}
    \item $H^0_\Hdg(B\mr{PGL}_{n}/\F_2)\simeq \mbb \F_2$,
    \item $H^1_\Hdg(B\mr{PGL}_{n}/\F_2)\simeq 0$,
    \item $H^2_\Hdg(B\mr{PGL}_{n}/\F_2)\simeq H^{1,1}_\Hdg(B\mr{PGL}_{n}/\mbb F_2)\simeq \mbb \F_2$ if $n$ is even and $H^2_\Hdg(B\mr{PGL}_{n}/\mbb \F_2)\simeq 0$ if $n$ is odd,
    \item $H^3_\Hdg(B\mr{PGL}_{n}/\mbb F_2)\simeq H^{1,2}_\Hdg(B\mr{PGL}_{n}/\F_2)\simeq \F_2$ if $n$ is even and $H^3_\Hdg(B\mr{PGL}_{n}/\mbb \F_2)\simeq 0$ if $n$ is odd.
\end{enumerate}  Similar assertions also hold for de Rham cohomology. We will denote by $x_2$ and $x_3$ the unique non-zero elements in $H^2_\Hdg$ and $H^3_\Hdg$ in the case $n$ is even.
\end{lemma}
\begin{proof}
    	By \cite[Theorem 2.4]{totaro2018hodge} and \cite[Theorem 1.1]{chaput2010adjoint} (cf. \cite[Theorem 8.1]{totaro2018hodge}) for every split simple $k$-group $G$ not of type $C$ we have
		\begin{align*}H^2_{\on{H}}(BG/k)=&H^2(BG,\mc{O})\oplus H^1(BG,\Omega^1)\oplus H^0(BG,\Omega^2) \\
			\simeq &H^2(G,k)\oplus (\mathfrak{g}^*)^G\oplus H^{-2}(G,S^2(\mathfrak{g}^*)) \\
			\simeq &H^2(G,k)\oplus (\mathfrak{t}^*)^W,		
		\end{align*}
		where $\mathfrak{g}$ and $\mathfrak{t}$ are the Lie algebras of $G$ and a maximal torus $T\subset G$, respectively, and where $W$ is the Weyl group of $G$.
		By \cite[II, Corollary 4.11]{jantzen2003representations},  we have $H^0(G,k)=k$ and $H^i(G,k)=0$ for $i>0$. Setting $k=\F_2$ and $G=\on{PGL}_{n}$ yields
		\begin{equation}\label{pgl-h2} 
			H^2_{\on{H}}(B\mr{PGL}_{n}/\F_2)=(\mathfrak{t}^*)^W.
		\end{equation}
		If $G=\on{PGL}_{n}$, then $W=S_{n}$ and the  $S_{n}$-representation $\mathfrak{t}^*$ fits into a short exact sequence
		\[0\to \mathfrak{t}^*\to \F_2^{\oplus n}\xrightarrow{\Sigma} \F_2\to 0,\]
		where $S_{n}$ acts on $\F_2^{\oplus n}$ by permutation and the map $\Sigma$ is given by the summation of  coordinates. The invariants $(\F_2^{\oplus n})^{S_n}$ are spanned by the vector $(1,1,\dots,1)$ which lies in $\mathfrak{t}^*$ if and only if $n$ is even. This gives (3); in particular, if $n$ is even there is a unique $x_2\in (\mathfrak{t}^*)^{S_n}\simeq  H^1(B\mr{PGL}_{n},\Omega^1)$ such that \[H^2_{\on{H}}(B\mr{PGL}_{n}/\F_2)=\F_2\cdot x_2.\] 
		
		To compute the other cohomology groups we will use a Leray-Serre-type spectral sequence. Namely, let $\on{PGL}_{n}$ act on $\P^{n-1}$ as its automorphism group, and let $P\subset \on{PGL}_{n}$ be the stabilizer of $(1:0:\dots:0)\in \P^{n-1}(\mbb F_2)$. We have $H^*_{\on{H}}(\P^{n-1}/\mbb F_2)= \F_2[h]/(h^{n})$, where $h$ has Hodge bidegree $(1,1)$ (meaning  $h\in H^1(\mbb P^{n-1},\Omega^1)\simeq H^2_\Hdg(\mbb P^{n-1}/\mbb F_2)$). The Levi subgroup corresponding to the parabolic $P$ is isomorphic to $\on{GL}_{n-1}$. Therefore, by \cite[Proposition 9.5]{totaro2018hodge} we have a spectral sequence
		\begin{equation}\label{parabolic-ss}E_2^{i,j}\coloneqq H^i_{\on{H}}(B\mr{PGL}_{n}/\F_2)\otimes H^j_{\on{H}}(\P^{n-1}/\F_2)\Rightarrow H^*_{\on{H}}(B\mr{GL}_{n-1}/\F_2).\end{equation}
	  We know $H^2_{\on{H}}(B\mr{GL}_{n-1}/\F_2)=\F_2$ and $H^3_{\on{H}}(B\mr{GL}_{n-1}/\F_2)=0$. From this it is easy to see that the $E_2$ page of the spectral sequence (\ref{parabolic-ss}) is given in low degrees by
\[
\text{$n$ even} \begin{tikzpicture}
			\matrix (m) [matrix of math nodes,
			nodes in empty cells,nodes={minimum width=5ex,
				minimum height=5ex,outer sep=-5pt},
			column sep=1ex,row sep=1ex]{
				&     &   &  \\
				2	  &  h & \cdot	& \cdot	&\\
				1     &  0 &  0  &  \cdot & \\
				0     &  1  & 0 & x_2 & x_3 \\
				\quad\strut &   0  &  1  &  2  & 3 & \strut \\};
			\draw[thick] (m-1-1.east) -- (m-5-1.east) ;
			\draw[thick] (m-5-1.north) -- (m-5-5.north) ;
		\end{tikzpicture} \quad
			\text{$n$ odd}	\begin{tikzpicture}
			\matrix (m) [matrix of math nodes,
			nodes in empty cells,nodes={minimum width=5ex,
				minimum height=5ex,outer sep=-5pt},
			column sep=1ex,row sep=1ex]{
				&     &   &  \\
				2	  &  h & \cdot	& \cdot	&\\
				1     &  0 &  0  &  \cdot & \\
				0     &  1  & 0 & 0 & 0 \\
				\quad\strut &   0  &  1  &  2  & 3 & \strut \\};
			\draw[thick] (m-1-1.east) -- (m-5-1.east) ;
			\draw[thick] (m-5-1.north) -- (m-5-5.north) ;
		\end{tikzpicture} 			    
\]
		where  $x_3\in H^2(B\mr{PGL}_{4m+2},\Omega^1)$ is defined by $d_3(h)=x_3$. Thus for all $n$ we get 
		\begin{equation}\label{pgl-h1}H^1_{\on{H}}(B\mr{PGL}_{n}/\F_2)\simeq 0. \end{equation} 
		We also get that
		\[
		H^3_{\on{H}}(B\mr{PGL}_{n}/\F_2)\simeq H^{1,2}_{\on{H}}(B\mr{PGL}_{n}/\F_2)\simeq \F_2 \cdot {x_3}
		\]
		when $n$ is even and 
		\[
		H^3_{\on{H}}(B\mr{PGL}_{n}/\F_2)\simeq 0
		\]
		if $n$ is odd.
		This gives us the groups $H^i_{\on{H}}(B\mr{PGL}_{n}/\F_2)$ for $i\leq 3$. By looking at the bigrading of generators it is also easy to see that the Hodge-de-Rham spectral sequence necessarily degenerates in degrees $\leq 3$, and so
		\begin{equation}\label{hdr-pgl-leq3} H^i_{\on{H}}(\on{PGL}_{n}/\F_2)\simeq H^i_{\on{dR}}(\on{PGL}_{n}/\F_2) \text{ for all $i\leq 3$.}\qedhere\end{equation}
\end{proof}

	Now let us compute the rest of the Hodge cohomology ring in the case of $B\mr{PGL}_{4m+2}$.
	\begin{rmk}
	If $n$ is odd, then $2$ is not a torsion prime for $\mr{PGL}_n$, then by the general result of Totaro \cite[Theorem 9.2]{totaro2018hodge} one has $H^*_\Hdg(B\mr{PGL}_n/\mbb F_2)\simeq H^*_\dR(B\mr{PGL}_n/\mbb F_2)\simeq \mbb F_2[c_2,\ldots,c_n]$ where $|c_i|=2i$. Thus understanding the case $n=4k+2$ is the next natural step.  
	\end{rmk}
	\begin{thm}\label{thm-hodge-pgl}
		The bigraded ring $H^{*,*}_{\on{H}}(B\mr{PGL}_{4m+2}/\F_2)$ is generated by elements 
		\smallskip
		
		\[x_2\in H^{1,1}_\Hdg(B\mr{PGL}_{4m+2}/\mbb F_2),\quad x_3\in H^{1,2}_\Hdg(B\mr{PGL}_{4m+2}/\mbb F_2),\]

		\[b_h\in H^{4h,4h}_\Hdg(B\mr{PGL}_{4m+2}/\mbb F_2),\quad y_I\in H^{2d(I)-1,2d(I)-1}_\Hdg(B\mr{PGL}_{4m+2}/\mbb F_2).\]
	\smallskip 
		
		Here $1< h\leq 2m+1$, $I=(i_1,\ldots,i_r)\in \{1<i_1<\dots<i_r\leq 2m+1\}$, and $d(I)\coloneqq i_1+\dots+i_r$. The relations are generated by
		\begin{equation}\label{rel1}
			y_Iy_J=\sum_{\emptyset\neq K\subset I}y_{(I-K)\cup J}y_{\{k_1\}}\dots y_{\{k_s\}}x_2^{l(K)-1},
		\end{equation}
		\begin{equation}\label{rel2}
			y_{\{h,h,j_1,\dots,j_s\}}=y_{\{j_1,\dots,j_s\}}b_h+y_{\{h,j_1,\dots,j_s\}}y_{\{h\}}x_2=0,
		\end{equation}
		\begin{equation}\label{rel3}
			x_3y_I=0 \text{ for all $I$.}
		\end{equation}
		Similarly, the graded ring $H^*_{\on{dR}}(B\mr{PGL}_{4m+2}/\mbb F_2)$ has generators $x_2\in H^2_\dR$, $x_3\in H^3_\dR$, $b_h\in H^{8h}_\dR$, $y_I\in H^{4d(I)-2}_\dR$, where $h$, $I$ and $d(I)$ are as above, and relations generated by (\ref{rel1}), (\ref{rel2}) and (\ref{rel3}). In particular, we have an isomorphism of graded rings $H^*_{\on{H}}(B\mr{PGL}_{4m+2}/\F_2)\simeq H^*_{\on{dR}}(B\mr{PGL}_{4m+2}/\F_2)$.
	\end{thm}	
	\begin{rmk}
	    From \Cref{toda36-311}(c) and the conclusion of \Cref{thm-hodge-pgl} one can see that there is a slightly more economical expression for $H^*_{\on{H}}(B\mr{PGL}_{4m+2}/\F_2)$ as the middle term in the short exact sequence
	    $$0\to \mbb F_2[x_2,x_3,b_2,\ldots, b_{2m+1}]\xrightarrow{\cdot x_3}H^*_{\on{H}}(B\mr{PGL}_{4m+2}/\F_2) \xrightarrow{Bp^*} PH^{*,*}_{\on{H}}(B\mr{GL}_{4m+2}/\F_2)\to 0 .$$
	\end{rmk}
	 
	\begin{proof}[Proof of \Cref{thm-hodge-pgl}]

 Consider the Eilenberg-Moore spectral sequence of \Cref{eilenberg-moore-easy} associated to the short exact sequence
		\[1\to \G_{\on{m}}\to \on{GL}_{4m+2}\xrightarrow{p} \on{PGL}_{4m+2}\to 1.\]
		In \Cref{additive} we proved that it degenerates on the $E_2$ page, which (by the computation in \Cref{coaction-gl}) is given by
		\[(1\otimes PH^{*,*}_\Hdg(B\mr{GL}_{4m+2}/\F_2))\oplus (z_3\F_2[z_3]\otimes \F_2[c_1,b'_h]_{h=2}^{2m+1}),\]
		with $E_2^{0,*}\simeq E_\infty^{0,*}\simeq (1\otimes PH^{*,*}_\Hdg(B\mr{GL}_{4m+2}/\F_2))$ and $E_2^{>0,*}\simeq E_\infty^{>0,*}\simeq (z_3\F_2[z_3]\otimes \F_2[c_1,b'_h]_{h=2}^{2m+1})$.
		Below let us denote by $b_h'\in PH^{4h,4h}_\Hdg(B\mr{GL}_{4m+2}/\F_2)$ the elements that we called $b_h$ in \Cref{coaction-gl}. Let  $b_h\in H^{4h,4h}_\Hdg(B\mr{PGL_{4m+2}}/\mbb F_2)$ be a choice liftings of $b_h'$, meaning that $Bp^*(b_h)=1\otimes b'_h$ (such a lifting exists since the spectral sequence degenerates). Note that since $Bp^*$ preserves the Hodge bigrading we can pick $b_h$ to be bihomogeneous (explicitly, lying in $H^{4h,4h}_\Hdg(B\mr{PGL_{4m+2}}/\mbb F_2)$).
		
		Recall that (by \Cref{rem:about Hodge grading}) a homogeneous element $$x\in (\on{Cotor}^i_{H^{*,*}_\Hdg(B\mbb G_m)}(\F_2,H^{*,*}_\Hdg(B\mr{GL}_{4m+2}/\mbb F_2)))^{h,j}\simeq (E_\infty^{i,j})^h$$ gives a class in $\gr_i(H_\Hdg^{h,i+j}(B\mr{PGL}_{4m+2}/\F_2))$. We let $i+j+h$ be the ``total degree". For elements in $(1\otimes PH^{s,t}_\Hdg(B\mr{GL}_{4m+2}/\F_2))$ the $\Z^3$-grading $(i,j,k)$ is $(0,s,t)$, while for $z_3$ it is $(1,1,1)$.
		
		Note that by \Cref{coaction-gl} we have an embedding of the subalgebra $\F_2[z_3]\otimes \F_2[c_1,b'_2,\dots,b'_{2m+1}]\subset E_\infty^{*,*}$. The elements $c_1$ and $z_3$ are the only non-zero elements of total degree 2 and 3, thus they have to be\footnote{Since $H^2_\Hdg(B\mr{PGL}_{4m+2}/\mbb F_2)\simeq\mbb F_2\cdot x_2$ and $H^3_\Hdg(B\mr{PGL}_{4m+2}/\mbb F_2)\simeq \mbb F_2\cdot x_3$.} the images of elements $x_2$ and $x_3$ from \Cref{pgl-small} in the infinite page $E_\infty^{*,*}$ (which is identified with the associated graded for the spectral sequence filtration, and by the ``image'' we mean the image in this associated graded). More generally, by definition, the image of each $b_h$ is  $b_{h'}$, and so the subalgebra $\F_2[z_3]\otimes \F_2[c_1,b'_2,\dots,b'_{2m+1}]$ is the image in the associated graded of the subalgebra in $H^*_\Hdg(B\mr{PGL}_{4m+2}/\mbb F_2)$ generated by $x_2,x_3$ and $b_h$'s. In particular, we see that there are no non-trivial relations between $x_2,x_3$ and $b_h$'s (otherwise there would be some in the associated graded as well) and we get an embedding $\F_2[x_2,x_3,b_2,\dots,b_{2m+1}]\subset H^*_\Hdg(B\mr{PGL}_{4m+2}/\mbb F_2)$. Moreover, looking at the description in  \Cref{coaction-gl} once again, we see that $x_3\F_2[x_2,x_3,b_2,\dots,b_{2m+1}]$ maps isomorphically to $\on{Ker}Bp^*$ (indeed, this follows from the isomorphism $z_3\F_2[c_1,z_3,b_2,\dots,b_{2m+1}]\simeq E_\infty^{>0,*}$ for the associated graded). We thus have isomorphisms		
		\begin{equation}\label{ker-im}
			\on{Im}Bp^*= PH^{*,*}_{\on{H}}(B\mr{GL}_{4m+2}/\F_2),\qquad \on{Ker}Bp^*=x_3\F_2[x_2,x_3,b_2,\dots,b_{2m+1}].
		\end{equation}
		To complete the proof of \Cref{thm-hodge-pgl}, it remains to construct elements $y_I\in H^*_{\on{H}}(B\mr{PGL}_{4m+2}/\F_2)$ and show that all relations are generated by (\ref{rel1}), (\ref{rel2}), (\ref{rel3}).

		By \Cref{toda36-311}, the subalgebra  $PH^{*,*}_{\on{H}}(B\mr{GL}_{4m+2}/\F_2)$ is generated by
		\[y_I'=y'(i_1,\ldots,i_r)\coloneqq d_1(\cl{c}_{2i_1}\!*\ldots *\cl{c}_{2i_r}),\qquad I=\set{i_1,\ldots,i_r},\quad 1<i_j\leq 2m+1\text{ for all $j$}.\]
		as an $\F_2[c_1,b_2',\dots,b_{2m+1}']$-module (here again we call by $y_I'$ the elements that were called $y_I$ in \Cref{toda36-311}).
		
		By the degeneration of the spectral sequence, we may pick $y_I\in H^{4d(I)-2,4d(I)-2}_{\on{H}}(B\mr{PGL}_{4m+2})$  such that $Bp^*(y_I)=y'_I$. Recall from \Cref{column} that the spectral sequence filtration on $H^*_{\on{H}}(B\mr{PGL}_{4m+2}/\F_2)$ is given by column degree. 
		
		Note that $E_\infty^{>k,*}\simeq z_3^k\F_2[z_3]\otimes \F_2[c_1,b'_2,\dots,b'_h]$ and $x_3$ maps to $z_3$ in the associated graded. 
		 By \Cref{rem:multiplication in cotor}
		we have $(z_3\otimes 1)\cdot (1\otimes y'_I)=0$ in $E_2$, hence $x_3y_I=0$ in the associated graded in the $E_\infty$-page. By \Cref{column}, this means that there exists $f\in H^*_{\on{H}}(B\mr{PGL}_{4m+2}/\F_2)$ such that $x_3y_I=x_3^2f$ in $H^*_{\on{H}}(B\mr{PGL}_{4m+2}/\F_2)$. Replacing $y_I$ by $y_I-x_3f$, we now have $Bp^*(y_I)=y_I$ and $x_3y_I=0$, that is, (\ref{rel3}) holds for such $y_I$.
		
			To check the relations (\ref{rel1}) and (\ref{rel2}) we proceed as follows. 
		By \Cref{toda36-311}(c), the relations (\ref{rel1}) and (\ref{rel2}) hold after we apply $Bp^*$, or, in other words, the difference of the left and right hand sides lies in $\on{Ker}Bp^*$. 
		Since $x_3y_I=0$ and every term in relations (\ref{rel1}) and (\ref{rel2})  contains at least one $y_I$, they are killed by multiplication by $x_3$. However, by (\ref{ker-im}), no element in $\on{Ker}Bp^*$ is killed by $x_3$, hence relations (\ref{rel1}) and (\ref{rel2}) hold on the nose.
		
	It remains to show that there are no further relations. Since we proved that the relations in the proposition hold, we have the map from the ring in the statement of the theorem (call it $A^*$) to $H^*_{\on{H}}(B\mr{PGL}_{4m+2}/\F_2)$.  The associated graded of $A^*$ by the powers of $x_3$ is isomorphic to $A^*$ again and also coincides with the description of $\on{Cotor}$ in \Cref{compute-cotor} (using the description of $PH^{*}_{\mr{sing}}(B\mr{GL}_{4k+2},\F_2)$ from \Cref{toda36-311}) by an easy direct inspection.  It follows that the map to $E_\infty$-page is an isomorphism. Since both filtrations are complete, this completes the proof for Hodge cohomology. The proof in de Rham cohomology context is entirely analogous, using (\ref{hdr-pgl-leq3}) as a starting input in degrees $\leq 3$.
	\end{proof}
	
	\begin{proof}[Proof of \Cref{mainthm}(1) for $\mr{PGL}_{4m+2}$]
		The conclusion follows by comparing the explicit descriptions in \Cref{thm-hodge-pgl} and \cite[Proposition 4.2]{toda1987cohomology}.
	\end{proof}

	\section{Hodge cohomology of \texorpdfstring{$B\mr{PSO}_{4m+2}$}{BPSO}}\label{sec: PSO}
		In this section we compute the Hodge cohomology ring of $B\mr{PSO}_{4m+2}$ (\Cref{thm-hodge-pso}). This then implies \Cref{mainthm}(1) for $B\mr{PSO}_{4m+2}$ by explicitly comparing the answers in the singular, Hodge and de Rham settings. For the computation we again follow Toda's strategy, but, contrary to the $\mr{PGL}_{4m+2}$ case, the details will be quite different.
	
	\subsection{Cohomology of \texorpdfstring{$B\mr{O}_n$}{BO}  and \texorpdfstring{$B\mr{SO}_n$}{BSO}}\label{ssec:cohomology of BO_n and BSO_n} By the orthogonal group $\mr O_n$ over $\mbb F_2$ we mean the corresponding Chevalley model, namely the group scheme of linear transformations of $(\mbb F_2^{\oplus n},q)$ that preserve the non-singular quadratic form $q=x_1x_2+x_3x_4+\ldots +x_{n-2}x_{n-1}+ x_n^2$ in the case $n$ is odd and $q=x_1x_2+\ldots +x_{n-1}x_n$ in the case $n$ is even. The correct definition of special orthogonal group $\mr{SO}_n$ in characteristic 2 is somewhat tricky: namely, if $n$ is odd, $\mr{SO}_n$ is defined in the usual way as the kernel of the determinant map $\det\colon \mr{O}_n\to \mu_2$, but if $n$ is even one considers the ``Dickson determinant" $D\colon \mr{O_n}\to \mbb Z/2$ (see \cite[Section 4.1.2]{chaput2010adjoint}) instead, and defines $\mr{SO}_n\coloneqq \ker(D)\subset \mr{O}_n$. 
	
	We briefly recall the structure of Hodge cohomology rings of the corresponding classifying stacks $B\mr{O}_n$ and $B\mr{SO}_n$ established in \cite[Section 11]{totaro2018hodge}. If $n=2r$ is even then 
	$$
	H^{*,*}_{\on{H}}(B\mr{O}_{2r}/\F_2)\simeq \F_2[u_1,\dots,u_{2r}] \quad \text{ and } \quad  H^{*,*}_{\on{H}}(B\mr{SO}_{2r}/\F_2)\simeq \F_2[u_2,\dots,u_{2r}] 
	$$
	with $|u_{2a}|=(a,a)$ and $|u_{2a+1}|=({a,a+1})$. The natural restriction map $H^{*,*}_{\on{H}}(B\mr{O}_{2r}/\F_2) \to H^{*,*}_{\on{H}}(B\mr{SO}_{2r}/\F_2)$ induced by the embedding $\mr{SO}_{2r} \to \mr{O}_{2r}$ simply sends $u_1$ to 0 (and $u_i$ to $u_i$ for $i\ge 2$).

	It is often convenient to pull-back cohomology of $B\mr{O}_{2r}$ (and $B\mr{SO}_{2r}$) to the classifying stack of a product of several copies of $B\mr{O}_2$ inside. Namely, we have an embedding $\mr{O}_2^r\to \mr{O}_{2r}$ which induces a cover $(B\mr{O}_{2})^r\to B\mr{O}_{2r}$. Let $s_i,t_i\in H^{*,*}_\Hdg(B\mr{O}_{2}^r/\mbb F_2)$ to be the pull-back of $u_1,u_2\in H^{*,*}_\Hdg(B\mr{O}_{2}/\mbb F_2)$ under the $i$-th projection. 
	By \cite[Lemma 11.3]{totaro2018hodge}, the pull-back map $$H^{*,*}_\Hdg(B\mr{O}_{2r}/\F_2)\tto H^{*,*}_\Hdg (B\mr{O}_2^r/\F_2)\simeq \mbb F_2[s_1,t_1,s_2,t_2,\ldots, s_r,t_r] $$ is an embedding and sends 
		\begin{equation}\label{eq:pull-back for SO_2r}
		    u_{2a}\mapsto \sum_{1\leq i_1<\dots <i_a\leq r}t_{i_1}\cdots t_{i_a}, \quad u_{2a+1}\mapsto \sum_{j=1}^r\left(s_j\cdot \sum_{\stackrel{1\leq i_1<\dots <i_a\leq r}{i_h\text{ \underline{not} equal to } j}}t_{i_1}\cdots t_{i_a}\right).
		\end{equation}
		One can also get similar formulas in the case $n=2r+1$ is odd, but we won't need them further, so let us just refer the reader to \cite[Section 11]{totaro2018hodge}. 
	

	Finally, let us note that by \cite[Theorem 10.1 and Theorem 11.1]{totaro2018hodge}, the Hodge-de Rham spectral sequences for $B\mr{O}_n$, $B\mr{SO}_n$ and $B\mu_2$ degenerate and induce natural isomorphisms of graded rings $H^*_{\on{H}}(B\mr{O}_n/\F_2)\simeq H^*_{\on{dR}}(B\mr{O}_n/\F_2)$ and $H^*_{\on{H}}(B\mr{SO}_n/\F_2)\simeq H^*_{\on{dR}}(B\mr{SO}_n/\F_2)$. This way the above discussion also applies to the de Rham cohomology rings $H^*_{\on{dR}}(B\mr{O}_n/\F_2)$ and $H^*_{\on{dR}}(B\mr{SO}_n/\F_2)$. 

\subsection{The coaction of \texorpdfstring{$H^*_{\on{H}}(B\mu_2/\F_2)$}{H*(B\textmu}}	Let $n=2r$ be an even integer. We have 
	\[H^{*,*}_{\on{H}}(B\mu_2/\F_2)\simeq \Lambda_2\coloneqq  \F_2[x_1,x_2]/(x_1^2),\qquad H^{*,*}_{\on{H}}(B\mr{SO}_{2r}/\F_2)\simeq \F_2[u_2,\dots,u_{2r}],\] where \[x_1\in H^{1,0}_\Hdg(B\mu_2/\F_2),\ x_2\in H^{1,1}_\Hdg(B\mu_2/\F_2),\ u_{2a}\in H^{a,a}_\Hdg(B\mr{SO}_n/\F_2),\ u_{2a+1}\in H^{a,a+1}_\Hdg(B\mr{SO}_n/\F_2).\] 
	
	
	For even $n$, the center of $\on{SO}_{n}$ is non-trivial and isomorphic to $\mu_2$. As in \Cref{main-comodule-algebra}, we can consider the multiplication map $\mu_2\times \on{SO}_{2r}\to \on{SO}_{2r}$ that induces a ring map
	\begin{equation}\label{coaction-so-eq}\phi\colon H^{*,*}_{\on{H}}(B\mr{SO}_{2r}/\F_2)\tto H^{*,*}_{\on{H}}(B\mu_2/\F_2)\otimes H^{*,*}_{\on{H}}(B\mr{SO}_{2r}/\F_2).\end{equation}
	We view $H^{*,*}_{\on{H}}(B\mr{SO}_{2r}/\F_2)$ as a $H^{*,*}_{\on{H}}(B\mu_2/\F_2)$-comodule algebra, with the coaction map $\phi$.	
Similarly, we can consider de Rham cohomology instead of Hodge.

	We will first describe the $H^{*,*}_{\on{H}}(B\mu_2/\F_2)$-comodule structure on $H^{*,*}_{\on{H}}(B\mr{SO}_{2r}/\F_2)$.
	
	\begin{lemma}\label{comparison-so}
		Let $n=2r$ be an even integer, and consider the $H^{*,*}_{\on{H}}(B\mu_2/\F_2)$-comodule algebra $H^{*,*}_{\on{H}}(B\mr{SO}_{2r}/\F_2)$, with coaction $\phi$ as in (\ref{coaction-so-eq}). We have
		\[\phi(u_{2a})=\sum_{i+j=a}\binom{r-j}{i}x_2^i\otimes u_{2j}+\sum_{i+j=a}\binom{r-j}{i}x_2^ix_1\otimes u_{2j-1}\]
		and
		\[\phi(u_{2a+1})=\sum_{i+j=a}\binom{r-j}{i}x_2^i\otimes u_{2j+1},\]
		where we put $u_1\coloneqq 0$ and $u_0\coloneqq 1$.
		
	\end{lemma}
	
	\begin{proof}
		Let $V\simeq \mbb F_2^{\oplus 2}$ be the tautological $2$-dimensional $k$-linear representation of $\on{O}_2$, and let $H\subset \on{O}_2$ be the subgroup isomorphic to $\Z/2\Z\times \mu_2$, where $\Z/2\Z$ permutes the coordinates of $V$ and $\mu_2$ acts by scaling. We 
		have
		\[H^{*,*}_{\on{H}}(B(\Z/2\Z)/\F_2)\simeq \F_2[z],\ \ H^{*,*}_{\on{H}}(B\mu_2/\F_2)\simeq \F_2[x_1,x_2]/(x_1^2),\ \  H^{*,*}_{\on{H}}(BH/\F_2)\simeq \F_2[z,x_1,x_2]/(x_1^2),\]
		where $z$, $x_2$ and $x_1$ have bidegrees $(0,1)$, $(1,1)$ and $(1,0)$, respectively.
		
		The coaction $H^{*,*}_{\on{H}}(BH/\F_2)\to H^{*,*}_{\on{H}}(B\mu_2/\F_2)\otimes H^{*,*}_{\on{H}}(BH/\F_2)$ induced by $\mu_2\times H \to H$ sends \[z\mapsto 1\otimes z,\quad t\mapsto 1\otimes x_2+x_2\otimes 1,\quad x_1\mapsto 1\otimes x_1+x_1\otimes 1.\]
		By \cite[Discussion above Lemma 11.3]{totaro2018hodge},\footnote{In \textit{loc.cit.} $z,x_1,x_2$ are denoted by $s,v,t$, respectively.} the pullback map $H^{*,*}_{\on{H}}(B\mr{O}_2/\F_2)\to H^{*,*}_{\on{H}}(BH/\F_2)$ sends $u_1\mapsto z$, $u_2\mapsto x_2+x_1z$. This allows to compute the coaction $\phi\colon H^{*,*}_{\on{H}}(B\mr{O}_2/\F_2)\to H^{*,*}_{\on{H}}(B\mu_2/\F_2)\otimes H^{*,*}_{\on{H}}(B\mr{O}_2/\F_2)$: it sends \[u_1\mapsto 1\otimes u_1,\quad u_2\mapsto 1\otimes u_2+x_2\otimes 1+x_1\otimes u_1.\]
		Write \[H^*_{\on{H}}(B\mr{O}_2^m/\F_2)=\F_2[s_1,\dots,s_m,t_1,\dots,t_m],\] where $s_i,t_i$ are pullbacks of $u_1,u_2\in H^{*,*}_{\on{H}}(B\mr{O}_2/\F_2)$ along the $i$-th projection. By the above computation the coaction $H^{*,*}_{\on{H}}(B\mr{O}_2^m/\F_2)\to H^{*,*}_{\on{H}}(B\mu_2/\F_2)\otimes H^{*,*}_{\on{H}}(B\mr{O}_2^m/\F_2)$ sends
		\[s_i\mapsto 1\otimes s_i,\quad t_i\mapsto 1\otimes t_i+t\otimes 1+v\otimes s_i.\]
		 Finally, following \Cref{eq:pull-back for SO_2r}, the pullback $H^{*,*}_\Hdg(B\mr{O}_{2m}/\F_2)\to H^{*,*}_\Hdg(B\mr{O}_2^m/\F_2)$ sends 
		\[u_{2a}\mapsto \sum_{1\leq i_1<\dots <i_a\leq m}t_{i_1}\cdots t_{i_a}, \quad u_{2a+1}\mapsto \sum_{j=1}^rs_j\cdot \left(\sum_{1\leq i_1<\dots <i_a\leq m, i_h\neq j}t_{i_1}\cdots t_{i_a}\right).\]
		One then checks formulas in statement of the lemma by plugging $\phi(t_i)$ and $\phi(s_i)$, and opening the brackets (this is a direct computation that we leave to the reader).

		Finally, following the discussion in \Cref{ssec:cohomology of BO_n and BSO_n}, to pass from $\mr O_{2m}$ to $\mr{SO}_{2m}\subset \mr{O}_{2m}$ we just need to set $u_1=0$ in all the formulas.
	\end{proof}

	\begin{rmk}
	On the topological side we have isomorphisms  \[H^*_{\on{sing}}(B(\Z/2\Z);\F_2)\simeq \F_2[z_1],\qquad H^*_{\on{sing}}(B\mr{SO}_n(\C);\F_2)=\F_2[w_2,\dots,w_n],\] where $|z_1|=1$, and $w_i$, with $|w_i|=i$ are the Stiefel-Whitney classes. We immediately see that $H^*_{\on{H}}(B\mu_2/\F_2)$ is \emph{not} isomorphic to $H^*_{\on{sing}}(B(\Z/2\Z);\F_2)$ as an algebra. In fact the situation is even worse: if we take the (unique) isomorphism of $H^*_{\on{H}}(B\mu_2/\F_2)$ and $H^*_{\on{sing}}(B(\Z/2\Z);\F_2)$ as coalgebras (\Cref{rem:Sp case iso as coalgebras}), it can't be extended to an isomorphism of the comodule algebras $H^*_{\on{H}}(B\mr{SO}_n/\F_2)$ and $H^*_{\on{sing}}(B\mr{SO}_n(\C);\F_2)$. This is the reason why the argument from \Cref{sec:GL SP graded vector space} doesn't apply to $\mr{PSO}_n$ as well. In particular we don't yet know that the Eilenberg-Moore spectral sequence collapses on second page (but we will show this in \Cref{cor:EM-for SO_4k+2 degenerates} below).
	\end{rmk}
	\subsection{Computation of \texorpdfstring{$\on{Cotor}$}{Cotor}}
	Note that $A=H^{*,*}_\Hdg(B\mr{SO}_{4m+2}/\mbb F_2)$ with $a_\sharp\coloneqq u_2$ satisfies the assumptions of \Cref{toda36-new} with $q=2$. Indeed, $d_2(\cl u_2)=1$, and $d_i(\cl u_2)=0$ for $i>2$ by degree reasons. We now establish an analogue of \Cref{toda36-311} in the setting of the special orthogonal group. Namely, we will describe more or less explicitly the subalgebras	
	$$
	PH^{*,*}_{\on{H}}(B\mr{SO}_{4m+2}/\F_2)\subset  P_2H^{*,*}_{\on{H}}(B\mr{SO}_{4m+2}/\F_2)\subset H^{*,*}_{\on{H}}(B\mr{SO}_{4m+2}/\F_2),$$
	as well as the cohomology of $P_2H^{*,*}_{\on{H}}(B\mr{SO}_{4m+2}/\F_2)$ with respect to the differential induced by $d_1$. 
	
	\begin{lemma}\label{toda36-311-so}
		Assume that $n=4m+2$.
		\begin{enumerate}[label=(\alph*)]
		    \item  There exists a unique sequence \[\cl{u}_2,\dots,\cl{u}_{4m+2}\in H^{*,*}_{\on{H}}(B\mr{SO}_{4m+2}/\F_2)\] with bidegrees $|\cl u_{2a}|=(a,a)$ and $|\cl u_{2a+1}|=(a,a+1)$, such that
		\begin{enumerate}[label=(\arabic*)]
			\item $\cl{u}_2=u_2$,
			\item For all $j\ge 2$, $\cl{u}_{j}\equiv u_{j}\pmod{u_2}$ and $\cl{u}_{j}\in P_2H^{*,*}_{\on{H}}(B\mr{SO}_{4m+2}/\F_2)$,
			\item For all $j\geq 2$, $\cl{u}_{2j-1}=d_1(\cl{u}_{2j})$,
			\item $H^{*,*}_{\on{H}}(B\mr{SO}_{4m+2}/\F_2)=\F_2[\cl{u}_2,\cl{u}_3,\dots,\cl{u}_{4m+2}]$,
			\item $P_2H^{*,*}_{\on{H}}(B\mr{SO}_{4m+2}/\F_2)=\F_2[\cl{u}_3,\cl{u}_4,\dots,\cl{u}_{4m+2}]$.
		\end{enumerate}  
		\smallskip
		
		\item 
		\begin{enumerate}[label=(\arabic*)]
		    \item 
	The elements $\cl u_3, \cl u_5, \ldots, \cl u_{4m+1}$ and $b_h\coloneqq {\cl u}_{2h}^2$ for $h>1$ are primitive (equivalently, they lie in $P_2H^{*,*}_{\on{H}}$ and are killed by $d_1$). 
	
	\item 
The natural map 
$$
\mbb F_2[b_2,\ldots, b_{2m+1}]\tto H^*(P_2H^{*,*}_{\on{H}}(B\mr{SO}_{4m+2}/\F_2), d_1)
$$
is an isomorphism.

		
		\end{enumerate}
		\smallskip 
		\item For any (unordered) tuple of integers $I=\{i_1,\ldots, i_r\}$, set $d(I)\coloneqq i_1+\ldots +i_r$, 
	 \[\cl{u}_I\coloneqq\cl{u}_{2i_1}\cdot\ldots\cdot \cl{u}_{2i_r}\in P_2H^{d(I),d(I)}_{\on{H}}(B\mr{SO}_{4m+2}/\F_2),\] and
		\[y_I=y(i_1,\ldots,i_r)\coloneqq d_1(u_{\cl{I}})\in PH^{d(I)-1,d(I)}_{\on{H}}(B\mr{SO}_{4m+2}/\F_2).\]
		In particular, $y_{\{i\}}=\cl u_{2i-1}$.
		
		\begin{enumerate}[label=(\arabic*)] \item The subalgebra $PH^{*,*}_{\on{H}}(B\mr{SO}_{4m+2}/\F_2)\subset H^*_{\on{H}}(B\mr{SO}_{4m+2}/\F_2)$ of primitive elements is generated by the $b_h$ and the $y_I$ for $I=\set{1<i_1< \dots< i_r\leq 2m+1}$, with the relations  generated by	
		\begin{equation}\label{eq:rel on y_I for SO_N}
			y_{\{i_1,\ldots,i_r\}}\cdot y_{\{j_1,\ldots j_s\}}=\sum_{h=1}^r y_{\{i_1,\ldots,i_{h-1},i_{h+1},\ldots,i_r, j_1,\ldots j_s\}}\cdot y_{\{i_h\}},
		\end{equation}
		    where we have $y_{\{h,h,i_1,\ldots, i_r\}}= b_h\cdot y_{I}$.

\smallskip
		\item $PH^{*,*}_{\on{H}}(B\mr{SO}_{4m+2}/\F_2)$ is a finitely generated module over the polynomial subalgebra
		$$
		\mbb F_2[\cl u_3,\cl u_5, \ldots,\cl u_{4m+1},b_2,b_3,\ldots,b_{2m+1}]\subset PH^{*,*}_{\on{H}}(B\mr{SO}_{4m+2}/\F_2).
		$$
		\end{enumerate}
		
		\end{enumerate}
		
		Entirely analogous statements hold with Hodge cohomology replaced by de Rham cohomology.

	\end{lemma}

	\begin{proof}
		We write $A$ for  $H^*_{\on{H}}(B\mr{SO}_{4m+2}/\F_2)$.
		
		(a)  By \Cref{toda36-new}(b) applied to $A$ and $a_{\sharp}=u_2$, the composite map \[P_2A\hookrightarrow A\to A/(u_2)\]
		is a ring isomorphism. Thus we can let $\cl{u}_2\coloneqq u_2$ and $\cl{u}_j\in P_2A$ be the inverse image of $u_j\pmod{u_2}$ for all $j\geq 2$ under this isomorphism. Elements $\cl{u}_j$ satisfy all the properties except, possibly, 3). For that one, note that by \Cref{comparison-so} we have $d_1(u_2)=0$ and $d_1(u_{2j})=u_{2j-1}$ for all $j> 1$. We have $d_1(\cl u_{2j})\in P_2A$ and it follows from the above that $d_1(\cl u_{2j}) \equiv d_1( u_{2j}) \equiv u_{2j-1} \pmod{u_2}$ and so $d_1(\cl u_{2j})=\cl u_{2j-1}$.\smallskip
		
		(b) Recall that we put $b_h\coloneqq \cl u_{2h}^2$. We have $d_1(b_h)=2\cdot \cl u_{2h}\cdot d_1(\cl u_{2h})=0$, while $d_1(\cl u_{2i-1})=d_1(d_1(\cl u_{2i}))=0$, so (1) follows. For (2) let us describe the differential $d_1$ on $P_2H^{*,*}_{\on{H}}(B\mr{SO}_{4m+2}/\F_2)\simeq \F_2[\cl u_3,\cl u_4, \ldots, \cl u_{4m+2}]$: $$ d_1(\cl u_{2i-1})=0\quad \text{and} \quad d_1(\cl u_{2i})=\cl u_{2i-1}.$$ Let $C\coloneqq \F_2[\cl{u}_3,\cl{u}_5,\ldots, \cl{u}_{4m+1},b_2,b_3,\ldots,b_{2m+1}]$. Note that $C\subset \on{Ker} d_1$, and so $d_1$ is $C$-linear. Moreover, elements $\cl u_I\coloneqq u_{2i_1}\cdot\ldots \cdot u_{2i_r}$ for all $I=\{1<i_1<\ldots<i_r\le 2m+1\}$ form a basis of $P_2H^{*,*}_{\on{H}}(B\mr{SO}_{4m+2}/\F_2)$ over $C$. One can then identify $(P_2H^{*,*}_{\on{H}}(B\mr{SO}_{4m+2}/\F_2),d_1)$ with the Koszul complex over $C$ for the regular sequence $\cl{u}_3,\cl{u}_5,\ldots, \cl{u}_{4m+1}$. This way we get that the cohomology of $(P_2H^{*,*}_{\on{H}}(B\mr{SO}_{4m+2}/\F_2),d_1)$ is given by $C/(\cl{u}_3,\cl{u}_5,\ldots, \cl{u}_{4m+1})\simeq \F_2[b_2,b_3,\ldots,b_{2m+1}]$.
		\smallskip 
		
		(c) Recall that $PH^{*,*}_{\on{H}}(B\mr{SO}_{4m+2}/\F_2)$ is identified with $\on{Ker}(d_1)\subset P_2H^{*,*}_{\on{H}}(B\mr{SO}_{4m+2}/\F_2)$. In (b) we showed that $\F_2[b_2,b_3,\ldots,b_{2m+1}]\subset \on{Ker}(d_1)$ maps isomorphically to the cohomology. From this we get that the whole $\on{Ker}(d_1)$ is a direct sum of $\F_2[b_2,b_3,\ldots,b_{2m+1}]$ and the image $\on{Im}(d_1)$. By the discussion in b) we have that $\on{Im}(d_1)$ is spanned over $C$ by $y_I\coloneqq d_1(\cl{u}_I)$ for all $I=\{1<i_1<\ldots<i_r\le 2m+1\}$. Moreover $y_{\{i\}}=d_1(\cl{u_{2i}})=\cl u_{2i-1}\in C$, so we get that $PH^{*,*}_{\on{H}}(B\mr{SO}_{4m+2}/\F_2)$ is spanned over $C$ by $y_I$'s with $l(I)\ge 2$. It remains to understand the relations between $y_I$. 
		
		First of all, indeed 
		$$
		y_{\{h,h,i_1,\ldots, i_r\}}=d_1(\cl u_h^2\cdot \cl u_I)=\cl u_h^2\cdot d_1(\cl u_I)=b_h\cdot y_{I}.
		$$
		Then, applying (\ref{eq:cyclic property for d_1}) to $a=\cl u_{I\setminus \{i_r\}}$, $b=\cl u_{i_r}$, $c=\cl u_J$, for any $I=\{i_1,\ldots,i_r\}$ and $J=\{j_1,\ldots j_s\}$ we get
		$$
		y_{\{i_1,\ldots, i_{r-1},i_r\}}\cdot y_{\{j_1,\ldots j_s\}}= y_{\{i_1,\ldots, i_{r-1}\}}\cdot y_{\{i_r,j_1,\ldots j_s\}}+ y_{\{i_1,\ldots, i_{r-1},j_1,\ldots j_s\}}\cdot y_{\{i_r\}}.
		$$\smallskip 
		Applying the same equation again to the first term on the right and continuing, we end up with (\ref{eq:rel on y_I for SO_N}). Taking $J=\{j\}$ be a 1-element set, we get
		\begin{equation}\label{eq:some equation on y_I}
		 y_{\{i_1,\ldots,i_r\}}y_{\{j\}}=\sum_{s=1}^r y_{\{i_1,\ldots,i_{s-1},i_{s+1},\ldots, i_r,j\}}\cdot y_{\{i_s\}}.   
		\end{equation}
		Note that this is exactly the relation which is obtained from $d_1^2(\cl u_{I\cup \{j\}})=0$. Indeed, for any $J=\{j_1,\ldots,j_r\}$
		\begin{equation}\label{eq:d_1(u_I)}
		 d_1(\cl u_I)=\sum_{s=1}^r \cl u_{J\ \! \setminus \{j_s\}}\cdot d_1(\cl u_{2j_s})= \sum_{s=1}^r \cl u_{J\ \! \setminus \{j_s\}}\cdot y_{j_s}.  
		\end{equation}
				Putting $J=I\cup \{j\}$ and applying again we exactly get the sum of the left and right hand sides in \Cref{eq:some equation on y_I}.
		
		Now, to show that Equations (\ref{eq:rel on y_I for SO_N}) generate all the relations
		assume $\sum f_Iy_I=0$ for some $f_I\in C$. Since $d_1(f_I)=0$ for all $I$, we get that $\sum f_I\cl{u}_I\in \on{Ker}d_1$. From the proof of b) we know that we can write \[\sum f_I\cl{u}_I=d_1(\sum g_J\cl{u}_J)+z,\] for some $g_J\in C$ and $z\in \F_2[b_2,b_3,\dots,b_{2m+1}]$. Note that $d_1(g_J)=0$ and so we get that 
		$$
		\sum f_I\cl{u}_I=\sum g_J d_1(\cl{u}_J)+z.
		$$ By the discussion above (\ref{eq:d_1(u_I)}) the expression for $d_1(\cl u_I)$ as a sum of $u_J$'s with coefficients in $C$ gives (\ref{eq:some equation on y_I}) after plugging $y_J$'s in place of $u_J$'s. Since $\cl u_I$ (with $I=\set{1<i_1< \dots< i_r\leq 2m+1}$) form a basis of $P_2H^{*,*}$ over the algebra $C$ we get that $\sum f_Iy_I=0$ in fact is a linear combination of relations in (\ref{eq:d_1(u_I)}) which are a particular case of (\ref{eq:rel on y_I for SO_N}).
	\end{proof}
	
	\begin{rmk}\label{easier}
	Note the difference in the formulas for elements $b_h$ as defined in \Cref{toda36-311-so}(b) and those defined in the topological setting \cite[p. 92]{toda1987cohomology}. Again, this is due to the fact that $H^{*}_{\on{H}}(B\mu_2/\F_2)$ and $H^*_{\on{sing}}(B(\Z/2\Z);\F_2)$ are not isomorphic as algebras: \Cref{toda36-new}(b), which is available only for $H^{*}_{\on{H}}(B\mu_2/\F_2)$, makes the formula for $b_h$ very simple.
	\end{rmk}
%
	
	\smallskip
	
	We are ready to compute the $\on{Cotor}^*_{H^{*,*}_{\on{H}}(B\mu_2/\mbb F_2)}(\F_2, H^{*,*}_{\on{H}}(B\mr{SO}_{4m+2}/\mbb F_2))$.
	
	\begin{lemma}\label{coaction-so}
		For every $m\geq 0$, we have an isomorphism of $\Z^3$-graded algebras
		\smallskip
		$$\on{Cotor}^*_{H^{*,*}_{\on{H}}(B\mu_2)}(\F_2, H^{*,*}_{\on{H}}(B\mr{SO}_{4m+2}))\simeq (1\otimes PH^{*,*}_{\on{H}}(B\mr{SO}_{4m+2}))\oplus (z_2\F_2[z_2]\otimes \F_2[b_2,\dots,b_{2m+1}]),$$
		where the natural bigrading of $z_2$ is $(1,0)$, while the $\on{Cotor}$-grading is 1. Also, the elements $b_h\in PH^{2h,2h}_{\on{H}}(B\mr{SO}_{4m+2})$ are the ones defined in \Cref{toda36-311-so}.
		
	\end{lemma}
	\smallskip
	
	\begin{rmk}
	    	    	Here $1\otimes PH^{*,*}_{\on{H}}(B\mr{SO}_{4m+2})$ acts on $z_2\F_2[z_2]\otimes \F_2[b_2,\dots,b_{2m+1}]$ via the surjective homomorphism $$PH^{*,*}_{\on{H}}(B\mr{SO}_{4m+2})\twoheadrightarrow H^*(P_2H^{*,*}_{\on{H}}(B\mr{SO}_{4m+2})),d_1)\simeq \F_2[b_2,\dots,b_{2m+1}].$$ Also, the ``$\on{Cotor}$"-grading of $(1\otimes PH^{*,*}_{\on{H}}(B\mr{SO}_{4m+2}))$ (including $b_i$'s) is 0, while  ``$\on{Cotor}$"-grading of $z_2$ is 1.
	\end{rmk}

	\begin{proof}
	 Let $A\coloneqq H^*_{\on{H}}(B\mr{SO}_n/\F_2)$. By \Cref{compute-cotor}(c), $\on{Cotor}$ groups can be computed as the cohomology $H^*(R_2\otimes_{\theta_2} A,d_{\theta_2})$ of the twisted tensor product. 
	 
	 Let us compute the differential $d_{\theta_2}\colon R_2\otimes A \to R_2\otimes A$ explicitly.  The differential $d_{\theta_2}$ is $R_2\otimes 1$-linear and so it is enough to understand $d_{\theta_2}$ on $1\otimes A$. Recall that $A\simeq P_2A[u_2]$ by \Cref{toda36-new}(b). We have $\phi(u_2)=1\otimes u_2 + x_2\otimes 1$ and $\phi(u_2^{2^i})=\phi(u_2)^{2^i}=1\otimes u_2^{2^i} + x_2^{2^i}\otimes 1$. By definition of $d_{\theta}$ (see \Cref{constr:twisted tensor product}) and our choice of the twisted cochain (\Cref{constr:twisted cochains}) this shows that $d_{\theta_2}(1\otimes u_2^{2^i})= z_{2^{i+1}+1}\otimes 1$ for $i\ge 0$. More generally, if we take $u_2^n$ and take the 2-adic expansion $n=2^{i_1}+2^{i_2}+ \ldots + 2^{i_r}$ with all $i_1< i_2<\ldots < i_r$, then we have
	 $$
	 \phi(u_2^n)=\phi(u_2)^{2^{i_1}}\cdot \ldots \cdot  \phi(u_2)^{2^{i_r}}=(1\otimes u_2^{2^{i_1}} + x_2^{2^{i_1}}\otimes 1)\cdot \ldots \cdot (1\otimes u_2^{2^{i_r}} + x_2^{2^{i_r}}\otimes 1).
	 $$
	Opening the brackets, one gets a formula for $d_{\theta_2}$:
	$$
	d_{\theta_2}(1\otimes u_2^n)=\sum_{j=1}^r z_{2^{i_j+1}+1}\otimes u_2^{n-2^{i_j}}.
	$$ 
	Moreover, if $a\in P_2A$, then $\phi(a)=1\otimes a+x_1 \otimes d_1(a)$ and so $d_{\theta_2}(a)=z_2 \otimes d_1(a)$. Finally, 
	$\phi(au_2^n)= \phi(a)\phi(u_2^n)$ and by opening the brackets in this product one sees that 
	$$d(1\otimes au_2^n)=d(1\otimes a)\cdot 1 \otimes u_2^n  + 1\otimes a \cdot d(u_2^n).$$
	Note that by the above formulas $\F_2[z_2]\otimes P_2A\subset  R_2\otimes A$ is closed under the differential, and so is $ \F_2[z_3,z_5,z_9,\ldots]\otimes\F_2[u_2]\subset R_2\otimes A$. Moreover, the partial Leibnitz rule above shows that there is a decomposition as a tensor product 
	$$
	(R_2\otimes_{\theta_2}A,d_{\theta_2})\simeq (\F_2[z_2]\otimes P_2A, d_{\theta_2}) \bigotimes \ \! (\F_2[z_3,z_5,z_9,\ldots]\otimes \F_2[u_2], d_{\theta_2}).
	$$
	By an analogous argument to \Cref{compute-cotor}(b) one can show that $(\F_2[z_3,z_5,z_9,\ldots]\otimes \F_2[u_2], d_{\theta_2})$ is quasi-isomorphic to the infinite Koszul complex $\otimes_{i=0}^{\infty} (\F_2[x_i]/(x_i^2)\otimes \F_2[y_i], d)$ with $d(x_i)=y_i$ and where $x_i$ and $y_i$ map to $1\otimes u_2^{2^i}$ and $z_{2^i+1}\otimes 1$,  respectively. Consequently, $(\F_2[z_3,z_5,z_9,\ldots]\otimes \F_2[u_2], d_{\theta_2})$ is acyclic\footnote{Meaning that it is quasi-isomorphic to $\mbb F_2$ in degree 0.} and we get a quasi-isomorphism
	$$
	( R_2\otimes_{\theta_2} A,d_{\theta_2})\simeq (\F_2[z_2]\otimes P_2A , d_{\theta_2}|_{\F_2[z_2]\otimes P_2A }).
	$$
	Recall that $d_{\theta_2}|_{\F_2[z_2]\otimes P_2A }$ is sending $ z_2^k\otimes a\mapsto z_2^{k+1}\otimes d_1(a)$. Thus $(\F_2[z_2]\otimes P_2A , d_{\theta_2})$ looks like
	$$
	1 \otimes P_2A \xrightarrow{z_2\otimes d_1(-)}z_2\otimes P_2A  \xrightarrow{z_2\otimes d_1(-)} z_2^2\otimes P_2A \xrightarrow{z_2\otimes d_1(-)} \ldots ,
	$$	
	from which we get that $H^0(\F_2[z_2]\otimes P_2A, d_{\theta_2})\simeq  1\otimes PA \simeq \on{ker}(1\otimes d_1)\subset 1 \otimes P_2A$, while $H^i(\F_2[z_2]\otimes P_2A, d_{\theta_2})\simeq z_2^i\otimes H^*(P_2A,d_1)$, which by \Cref{toda36-311-so}(b2) is isomorphic to $z_2^i\otimes \mbb F_2[b_1,\ldots, b_h]$. The algebra structure then also can be understood via (\ref{algebra-map}): namely, the map $\F_2[z_2]\otimes PA \to H^*(\F_2[z_2]\otimes P_2A, d_{\theta_2})$ is an algebra homomorphism.
	\end{proof}
	
	\begin{cor}\label{cor:EM-for SO_4k+2 degenerates}
	    The Eilenberg-Moore spectral sequence (from \Cref{eilenberg-moore-easy}) for 
	    $$
	    1\to \mu_2\to \mr{SO}_{4k+2} \to \mr{PSO}_{4k+2}\to 1
	    $$
	    over $\F_2$ degenerates on the $E_2$ page. 
	\end{cor}
	\begin{proof}
	 Similarly to \Cref{additive} we need to show that the dimensions of terms in the second page for the Eilenberg-Moore spectral sequence in Hodge and singular cohomology are the same. Indeed, by Toda's result \cite[Section 4.4]{toda1987cohomology} the spectral sequence degenerates on the singular cohomology side, so by Totaro's inequality we would get that it should also degenerate for Hodge and de Rham cohomology. The comparison is established by direct inspection. First, comparing \Cref{toda36-311-so}(c1) and \cite[Proposition 3.11]{toda1987cohomology} one sees that $PH^{*}_\Hdg(B\mr{SO}_{4k+2}/\F_2)$ and $PH^*_{\mr{sing}}(B\mr{SO}_{4k+2}(\mbb C),\mbb F_2)$ are given by the same generators and relations. This gives an isomorphism between $E_2^{0,*}$ in Hodge and singular cohomology. Finally, one observes $E_2^{>0,*}$ in both settings is given by a free module over a polynomial ring with generators in same degrees, see \Cref{coaction-so} and \cite[Section 4.4]{toda1987cohomology}, so the dimensions are also the same.
	\end{proof}
	
	\begin{rmk}\label{rem:Toda's strategy doesn't work}
	    Let us point out that Toda's strategy of proving degeneration (using pull-back with respect to the ``tensor product" map $\mr{O}_2\times \mr{SO}_{2m+1}\to \mr{SO}_{4m+2}$) doesn't work in the Hodge setting: the reason is that the corresponding pull-back map 
	    $$
	    H^*_\Hdg(B\mr{SO}_{4m+2}/\F_2)\to H^*_\Hdg(B\mr{O}_{2}/\F_2)\otimes H^*_\Hdg(B\mr{SO}_{2m+1}/\F_2)
	    $$
	    is no longer an embedding.
	\end{rmk}
	
	\subsection{Computation of Hodge cohomology of \texorpdfstring{$B\mr{PSO}_{4m+2}$}{BPSO}}
	We start by understanding Hodge cohomology of $B\mr{PSO}_{2r}$ in low degrees. 
	\begin{lemma}\label{lem:SO cohomology in low degrees} We have isomorphisms 
	\begin{enumerate}
	    \item $H^0_{\on{H}}(B\mr{PSO}_{2r}/\mbb F_2)\simeq \mbb F_2$ \
	    \item $H^1_{\on{H}}(B\mr{PSO}_{2r}/\mbb F_2)\simeq 0$,\
	    \item 
	    $
		H^2_{\on{H}}(B\mr{PSO}_{2r}/\mbb F_2)\simeq H^{1,1}_{\on{H}}(B\mr{PSO}_{2r}/\mbb F_2)\simeq 
		\begin{cases}
			\F_2, & \text{if }r=2k+1\\
		\F_2\oplus \F_2, &\text{if } r=2k.
		\end{cases}
		$
	\end{enumerate}
		In the case $r={2k+1}$ we let $x_2$ be the (unique) generator of $H^2_{\on{H}}(B\mr{PSO}_{4k+2}/\mbb F_2)$.	
	\end{lemma}
	\begin{proof}As in the proof of \Cref{pgl-small}, using that $\mr{SO}_n$ (and so $\mr{PSO}_{2r}$) is smooth and connected, we have isomorphisms 
	\[
		H^i_{\on{H}}(B\mr{PSO}_{2r}/\mbb F_2)\simeq 
		\begin{cases}
			\F_2 & i=0\\
			0 & i=1\\
		(\mf t^\vee)^W & i=2,
		\end{cases}
		\]
		where $\mf t$ is the Lie algebra of the maximal torus of $\mr{PSO}_{2r}$ and $W$ is the Weyl group. Thus it remains to show that $(\mf t^\vee)^W\simeq \F_2$. The maximal torus $T'\simeq \mathbb G_m^r\subset \mr{SO}_{2r}$ is given by $\{(t_1,t_1^{-1},\ldots, t_r,t_r^{-1})\}\subset \mr{SO}_{2r}$ (in the basis where the quadratic form $q_{2r}$ is given by $x_1x_2+\ldots+x_{2r-1}x_{2r}$). The maximal torus $T\subset \mr{PSO}_{2r}$ is obtained as the quotient of $T'$ by diagonal copy of $\mu_2$. Note that both tori are split and so there is a $W$-equivariant identification of the Lie algebras $\mf t$ and ${\mf t}'$ and the mod $ 2$ reductions $X_*(T)_{\mbb F_2}$ and $X_*(T')_{\mbb F_2}$ of the cocharacter lattices. Let $\chi_i\colon \mbb G_m \to T'$ be cocharacter corresponding to $t_i$; it is not hard to see that $X_*(T')$ is embedded into $X_*(T)$ as the lattice $\mbb Z\cdot \chi_1\oplus \ldots \oplus \mbb Z\cdot \chi_r$ inside the lattice generated by $\frac{1}{2}(\chi_1+\ldots +\chi_r)$ and $\chi_i$. The group $W$ is isomorphic to $S_r\ltimes (\Z/2\Z)^{r-1}$ where $(\Z/2\Z)^{r-1}$ acts trivially on $X_*(T')$ (and $\mf t'$) and $S_r$ acts by transpositions on $\chi_i$'s. For $X_*(T')$ we then have a short exact sequence 
		$$
		0\to \mbb Z\cdot (\chi_1+\ldots+\chi_r) \to X_*(T') \to L \to 0
		$$
		for a $W$-module $L$. This gives a short exact sequence 
		\begin{equation}\label{eq:ses for tori}
		    0\to \mbb F_2 \to \mf t' \to L/2 \to 0.
		\end{equation}
		Note however that the image of $\mf t' \to \mf t$ is exactly given by $L/2$ and so gives a $W$-equivariant splitting $\mf t\simeq \mbb F_2\oplus L/2$. In particular, $(\mf t^\vee)^W\simeq ((L/2)^\vee)^W\oplus \mbb F_2$.
		
		From (\ref{eq:ses for tori}) we have a short exact sequence 
		$$
		0\to (L/2)^\vee \to \mf t'^\vee \to \mbb F_2 \to 0,
		$$
		inducing a left-exact sequence
		$$
		0\to ((L/2)^\vee)^W \to (\mf t'^\vee)^W \to \mbb F_2 \to \ldots
		$$
		By direct inspection $\mf t'^\vee\simeq \mbb \F_2^{\oplus r}$ is the ``permutation module"\footnote{Meaning $(\Z/2\Z)^{r-1}$ acts trivially, while there is a choice of $r$ vectors that form a basis, such that $S_r$ acts on them by permutations.} for $W$ and maps to $\mbb F_2$ by $(x_1,\ldots, x_n)\mapsto \sum x_i$, while $(\mf t'^\vee)^W$ is spanned by the vector $(1,1,\ldots, 1)$. So the map $(\mf t'^\vee)^W\simeq \mbb F_2$ is given by multiplication by $r$ and we get that $((L/2)^\vee)^W$ is $\mbb F_2$ or $0$ depending on whether $r$ is even or odd. 
 	\end{proof}

	\begin{thm}\label{thm-hodge-pso}
		The bigraded ring $H^*_{\mr{H}}(B\mr{PSO}_{4m+2}/\F_2)$ is generated by
		\[x_2\in H^{1,1}_\Hdg(B\mr{PSO}_{4m+2}/\F_2),\qquad b_h\in H^{2h,2h}_\Hdg(B\mr{PSO}_{4m+2}/F_2),\] \[y_I=y_{\{i_1,\dots,i_r\}}\in H^{d(I),d(I)-1}_\Hdg(B\mr{PSO}_{4m+2}/\mbb F_2).\]
		Here $1< h\leq 2m+1$, $I=(i_1,\dots,i_r)$, where $1<i_1<\dots<i_r\leq 2m+1$, and $d(I)\coloneqq i_1+\dots+i_r$. The relations are generated by
		\begin{equation}\label{pso-rel1}
			x_2\cdot y_I=0,
		\end{equation}
		\begin{equation}\label{pso-rel3}
			y_{\{i_1,\dots,i_r\}}y_{\{i_{r+1},\dots,i_s\}}=\sum_{j=1}^ry_{\{i_1,\dots,i_{j-1},i_{j+1},\dots,i_s\}}y_{\{i_j\}} \text{ for $s>r\geq 2$},
		\end{equation}
		with the convention that
				\begin{equation}\label{pso-rel2}	
			y_{\{h,h,j_1,\dots,j_s\}}=y_{\{j_1,\dots,j_s\}}b_h.
		\end{equation}
		Similarly, the graded ring $H^*_{\on{dR}}(B\mr{PSO}_{4m+2}/\F_2)$ has generators $x_2\in H^2$, $b_h\in H^{4h}$, $y_I\in H^{2d(I)-1}$, where $h$, $I$ and $d(I)$ are as above, and relations generated by (\ref{pso-rel1}), (\ref{pso-rel2}) and (\ref{pso-rel3}). In particular, the Hodge-to-de Rham spectral sequence for $B\mr{PSO}_{4m+2}$ degenerates and we have an isomorphism of graded rings $H^*_{\on{H}}(B\mr{PSO}_{4m+2}/\F_2)\simeq H^*_{\on{dR}}(B\mr{PSO}_{4m+2}/\F_2)$.
	\end{thm}
	
	\begin{proof}
		The proof is similar to that of \Cref{thm-hodge-pgl}. As there, we only prove the theorem for Hodge cohomology, but a similar argument applies to de Rham cohomology. By \Cref{cor:EM-for SO_4k+2 degenerates}, Eilenberg-Moore spectral sequence for 
		$$1\to \mu_2\to \mr{SO}_{4k+2}\xrightarrow{p}  \mr{PSO}_{4k+2} \to 1$$ 
		degenerates on the $E_2$ page, which is isomorphic to  
		\[1\otimes PH^{*}_{\on{H}}(B\mr{SO}_{4m+2})\oplus z_2\F_2[z_2]\otimes \F_2[b'_2,\dots,b'_{2m+1}]\]
		by an explicit computation which we made in \Cref{coaction-so} (and where we now call $b_h'\in H^{2h,2h}_{\on{H}}(B\mr{SO}_{4m+2}/\mbb F_2)$ what we called $b_h$ in \Cref{coaction-so}(b)). From degeneration, we see that the pull-back map $Bp^*\colon H^*_{\on{H}}(B\mr{PSO}_{4m+2})\twoheadrightarrow 1\otimes PH^{*,*}_{\on{H}}(B\mr{SO}_{4m+2})$ is a surjection. We let $b_h\!\in\! H^{2h,2h}_{\on{H}}(B\mr{PSO}_{4m+2}/\mbb F_2)$ be any fixed lift of $1\otimes b_h'$. 
	
	The only element of total degree 2 in $E_\infty^{*,*}\simeq E_2^{*,*}$ is $z_2$. Thus it has to be the image of the class $x_2\in H^{1,1}_\Hdg(B\mr{PSO}_{4m+2}/\mbb F_2)$ from \Cref{lem:SO cohomology in low degrees} in $E_\infty^{*,*}$. Since $\mbb F_2[z_2]\otimes \mbb F_2[b_2',\ldots, b_{2m+1}']$ embeds into $E_2^{*,*}$ we get that the natural map $\mbb F_2[x_2]\otimes \mbb F_2[b_2,\ldots,b_{2m+1}] \to H^*_{\on{H}}(B\mr{PSO}_{4m+2}/\mbb F_2)$ is an embedding. We also get isomorphisms 
			\begin{equation}\label{ker-im-so}
			\on{Im}Bp^*= PH^{*,*}_{\on{H}}(B\mr{SO}_{4m+2}/\F_2),\qquad \on{Ker}Bp^*=x_2\F_2[x_2,b_2,\dots,b_{2m+1}].
		\end{equation}
		
				By \Cref{toda36-311-so}, the subalgebra  $PH^{*,*}_{\on{H}}(B\mr{SO}_{4m+2}/\F_2)$ is generated by
		\[y_I'=y'(i_1,\ldots,i_r)\coloneqq d_1(\cl{u}_{2i_1}\!*\ldots *\cl{u}_{2i_r}),\qquad I=\set{i_1,\ldots,i_r},\quad 1<i_j\leq 2m+1\text{ for all $j$}.\]
		as an $\F_2[b_2',\dots,b_{2m+1}']$-module (here again we call by $y_I'$ the elements that were called $y_I$ in \Cref{toda36-311}).
	Then, as in \Cref{thm-hodge-pgl}, taking $y_I\in H^{4d(I),4d(I)-1}_\Hdg(B\mr{PSO}_{4m+2}/\mbb F_2)$ such that $Bp^*(y_I)=y_I'$ and possibly replacing them by $y_I-x_2f$ one makes them satisify the relations \Cref{pso-rel1} and \Cref{pso-rel3}. The same argument as in \Cref{thm-hodge-pgl} also shows that these are the only relations.
	\end{proof}
	
	\begin{proof}[Proof of \Cref{mainthm} for $\mr{PSO}_{4m+2}$]
		The conclusion follows by comparing the descriptions given by \Cref{thm-hodge-pso} and \cite[Proposition 4.5]{toda1987cohomology}.
	\end{proof}

	\section{Applications to representation theory}\label{sec-reptheory}
	
	In this section, we reinterpret our computations in terms of representation theory. Let $k$ be a field, $G$ be a connected reductive $k$-group, $\Gamma$ be a central subgroup of $G$, and $\cl{G}\coloneqq G/\Gamma$ be the adjoint quotient. Recall that by Totaro's work \cite[Corollary 2.2]{totaro2018hodge} one has the following interpretation of the Hodge cohomology of $\cl G$: for all $i,j\geq 0$ we have
	\[H^{i,j}(B\cl{G}/k)\ism H^{j-i}(\cl{G},\on{Sym}^i \cl{\mathfrak{g}}^\vee),\]
	where $\cl{\mf g}$ is the Lie algebra of $\cl{G}$. 	
	Note that $\cl{\mf g}$ is also a $G$-module via the projection $G\twoheadrightarrow \cl{G}$. Since $\Gamma$ is of multiplicative type, the functor of $\Gamma$-invariants is exact, hence the Hochschild-Serre spectral sequence for $1\to \Gamma \to G \to \cl G \to 1$ provides isomorphisms
	$$
	H^{*}({G},\on{Sym}^i \cl{\mathfrak{g}}^\vee)\ism H^{*}(\cl{G},\on{Sym}^i \cl{\mathfrak{g}}^\vee).$$ Altogether, this shows the following: $H^{i,j}(B\cl{G}/k)\simeq 0$ if $i>j$, the ``pure" part $\oplus_{i}H^{i,i}(B\cl{G}/k)$ of Hodge cohomology is isomorphic to the algebra $$H^{0}({G},\on{Sym}^* \cl{\mathfrak{g}}^\vee)\simeq {(\on{Sym}^* \cl{\mathfrak{g}}^\vee)}^G,$$ and the ``non-pure" part $\oplus_{i\neq j}H^{i,j}(B\cl{G}/k)$ is given by the higher cohomology $H^{>0}({G},\on{Sym}^* \cl{\mathfrak{g}}^\vee)$.

	Using the computations of Hodge cohomology that we made in previous sections we will analyze this picture in the case when $k=\mbb F_2$, $G$ is $\mr{GL}_{4m+2}$, $\mr{SO}_{4m+2}$ or $\mr{Sp}_{4m+2}$, and $\Gamma$ is the center of $G$. Recall that for each of the $G$ under consideration, the Eilenberg-Moore spectral sequence in Hodge cohomology for $1\to \Gamma\to G\to\cl{G}\to 1$ degenerates at the $E_2$ page. Thus, in order to calculate the dimensions of $H^{i,j}(B\cl{G}/k)$, it will be enough to compute the dimension of the corresponding (see \Cref{rem:Hodge bigrading}) bigraded component on the $E_2$ page. Following the proofs of Theorems \ref{thm-hodge-pgl} and \ref{thm-hodge-pso} for $\mr{PGL}_{4m+2}$ and $\mr{PSO}_{4m+2}$, the $E_2$ page is in fact isomorphic to Hodge cohomology as a bigraded algebra (so we can also understand the multiplicative structure on $H^{*}({G},\on{Sym}^* \cl{\mathfrak{g}}^\vee)$ this way).
	
	\begin{rmk}\label{rem:Hodge bigrading}
	Recall that (assuming the degeneration of the Eilenber-Moore spectral sequence) a homogeneous class 
	$$x\in (\on{Cotor}^i_{H^{*,*}_\Hdg(B\Gamma)}(\F_2,H^{*,*}_\Hdg(BG/\mbb F_2)))^{h,j}\simeq (E_\infty^{i,j})^h$$ gives a class in $\gr_i(H_\Hdg^{h,i+j}(B\cl G/\F_2))$. Thus the bigrading we are interested in is given by $(h,i+j)$. We will call it \textit{Hodge bigrading} from now on and will denote it by $|x|_\Hdg\in \mbb Z^2$. 
	\end{rmk}
	
	\subsection{Projective linear group}
	
	When $\cl{G}=\on{PGL}_{n}$, the representation in question is $\mf{pgl}_{n}^\vee$. We have a short exact sequence of $\mr{GL}_{n}$-modules $$0\to \mbb F_2 \to\mf{gl}_{n}\to \mf{pgl}_{n} \to 0,$$ giving a short exact sequence 
	$$
	0\to \mf{pgl}_{n}^\vee \to\mf{gl}_{n}^\vee\to \mbb F_2 \to 0.
	$$
	\begin{rmk}
	    When $n$ is even this short exact sequence is non-split. Indeed, any such splitting would induce a Lie algebra direct sum decomposition of $\mf{gl}_{n}$ as $\mbb F_2\oplus [\mf{gl}_{n},\mf{gl}_{n}]$. However, when $n$ is even one has $\mbb F_2\subset [\mf{gl}_{n},\mf{gl}_{n}]$. Indeed, if $n=2$ then 
	    $$\begin{pmatrix}1 & 0 \\
	    0& 1\end{pmatrix}= \left[\begin{pmatrix}0 & 1 \\
	    0& 0\end{pmatrix}, \begin{pmatrix}0 & 0 \\
	    1& 0\end{pmatrix}\right], $$
	    and the general $n=2r$ case reduces to this one by considering the analogous block-diagonal matrices (with $r$ blocks of size $2$).
	\end{rmk}
	\begin{rmk}
	    We also note that if $n$ is even the representations $\mf{pgl}_{n}$ and $\mf{pgl}_{n}^\vee$ are not irreducible. Indeed, the trace function $\tr\colon \mf{ gl}_n \to \mbb F_2$ is $\mr{GL}_n$-invariant and is 0 on scalars $\mbb F_2\subset \mf{gl}_n$ and so defines a (non-zero) map $\mf{pgl}_{n}\to \mbb F_2$. Its kernel, however, is an irreducible $\mr{GL}_n$-module.
	\end{rmk}
	
	By \cite[Theorem 9.1]{totaro2018hodge}, the higher cohomology of $\mr{GL}_{n}$ with coefficients in $\on{Sym}^j\mf{gl}_{n}^\vee$ is $0$. In contrast, for even $n$, due to the non-splitness of the above short exact sequence, the higher cohomology of $\on{Sym}^j\mf{pgl}_{n}^\vee$ become quite complicated. Our computation (\Cref{thm-hodge-pso}) of Hodge cohomology of $B\mr{PGL}_{4m+2}$ allows to describe it fully in the case $n=4m+2$.

	Recall that the $E_2$ page for $\mr{PGL}_{4m+2}$ has been computed in \Cref{coaction-gl} as
	\[PH^{*}_\Hdg(B\mr{GL}_{4m+2}/\F_2)\ \oplus \  z_3\F_2[z_3]\otimes \F_2[c_1,b_2, \ldots,b_{2m+1}],\]
	where $PH^{*}_\Hdg(B\mr{GL}_{4m+2}/\F_2)$ is the $0$-th column $E_2^{0,*}$ and $z_3\F_2[z_3]\otimes \F_2[c_1,b_2, \ldots,b_{2m+1}]$ gives the rest.
	The Hodge-bidegrees here are given as follows: $PH^{*,*}_\Hdg(B\mr{GL}_{4m+2}/\F_2)$ is pure, $|z_3|_\Hdg=(1,2)$, $|c_1|_\Hdg=(1,1)$ and $|b_h|_\Hdg=(4h,4h)$. Therefore, we have an isomorphism 
	$$
	H^{>0}(\mr{GL}_{4m+2}, \on{Sym}^*\mf{pgl}_{4m+2}^\vee) \simeq z_3\F_2[z_3]\otimes \F_2[c_1,b_2, \ldots,b_{2m+1}].
	$$
	where $c_1,b_2, \ldots,b_{2m+1}\in H^0(\mr{GL}_{4m+2}, \mf{pgl}_{4m+2}^\vee)\simeq (\on{Sym}^*\mf{pgl}_{4m+2}^\vee)^{\mr{GL}_{4m+2}}$ are certain invariant polynomials of degrees $1,4,8,\ldots,8m+4$ (and which can be explicitly understood via \Cref{toda36-311}) and\footnote{This class is exactly the one that classifies the non-split extension $
	0\to \mf{pgl}_{n}^\vee \to\mf{gl}_{n}^\vee\to \mbb F_2 \to 0.
	$.} $z_3\in H^1(\mr{GL}_{4m+2}, \mf{pgl}_{4m+2}^\vee)$.
	This way we get that $H^{i}(\mr{GL}_{4m+2}, \on{Sym}^j\mf{pgl}_{4m+2}^\vee)$ has a basis consisting of monomials of the form $z_3^{i-j}f$, where $f$ is a monomial in $c_1,b_2,\dots,b_{2m+1}$ of total degree $j-i$. Therefore $\dim_{\mbb F_2} H^{i}(\mr{GL}_{4m+2}, \on{Sym}^j\mf{pgl}_{4m+2}^\vee)$ equals to the number of ways to write $j-i$ as a sum \[\gamma_1+4\beta_2+8\beta_3+\dots+(8m+4)\beta_{2m+1},\] where $\gamma_1$ and the $\beta_h$ are non-negative integers.
	
	\subsection{Projective orthogonal group}
	
	When $\cl{G}=\on{PSO}_{4m+2}$, the representation in question is $\mf{pso}_{4m+2}^\vee$. Since $\mu_2$ is not smooth, the ``Lie algebra" of $\mu_2$ is in fact a complex, namely the dual $\mbb L_{\mu_2/\mbb F_2}^\vee$ to the cotangent complex $\mbb L_{\mu_2/\mbb F_2}$. We have $H^0(\mbb L_{\mu_2/\mbb F_2}^\vee)=H^1(\mbb L_{\mu_2/\mbb F_2}^\vee)=\mbb F_2$. We have a fiber sequence $\mbb L_{\mu_2/\mbb F_2}^\vee \to \mf g \to \cl {\mf g}$ in the derived category of $G$-modules (where $G$ acts trivially on $\mbb L_{\mu_2/\mbb F_2}^\vee$) which gives an exact sequence of $\mr{SO}_{4m+2}$-modules
	\begin{equation}\label{eq:les for pso}
	0\to \mbb F_2 \to \mf{so}_{2r} \to \mf{pso}_{2r} \to \mbb F_2 \to 0
	\end{equation}
	as the long exact sequence of cohomology. However, for odd $r$ the first map $\mbb F_2 \to \mf{so}_{4m+2}$ is in fact split, as the lemma below shows.
	\begin{lemma}\label{lem:so_4m+2 in char 2}
	    The Lie algebra $\mf{so}_{4m+2}$ over $\F_2$ splits as $\F_2\oplus \mf{l}$ where $\mf{l}\simeq [\mf{so}_{4m+2}, \mf{so}_{4m+2}]$. This splitting is preserved by the $\mr{SO}_{4m+2}$-action and gives the decomposition of $\mf{so}_{4m+2}$ into a sum of simple representations.
	\end{lemma}
	\begin{proof}
	    By \cite[Table 1]{Hogeweij}\footnote{Note that we are in the $D_\ell$-type $\ell$ odd, intermediate case, in the notations of \textit{loc.cit.}} we have that the center $\mf{z}(\mf{so}_{4m+2})$ and $\mf{l}\coloneqq [\mf{so}_{4m+2}, \mf{so}_{4m+2}]$ are the only non-trivial Lie ideals in $\mf{so}_{4m+2}$. In particular, $\mf{z}(\mf{so}_{4m+2})$ is 1-dimensional and is exactly given by the image of $\F_2$ under the above map. Thus we only need to check that $\mf{z}(\mf{so}_{4m+2})$ doesn't belong to $\mf l$. Let's identify a Cartan subalgebra $\mf{h}\subset \mf{so}_{4m+2}$ with $X_*(T)\otimes \mbb F_2$ (where $X_*(T)$ are cocharacters of the maximal torus corresponding to $\mf h$). E.g. by \cite[Section 1]{Hogeweij} the intersection of $[\mf{so}_{4m+2}, \mf{so}_{4m+2}]$ with the Cartan subalgebra $\mf{h}$ is given by the image of coroot lattice $R^\vee\otimes \mbb F_2 \to X_*(T)\otimes \F_2$. In the standard basis for $X_*(T)$ (dual to what is usually denoted denoted $\varepsilon_1,\ldots, \varepsilon_{2m+1}\in X^*(T)$) the center $\mf{z}(\mf{so}_{4m+2})$ is spanned by the vector $(1,1,\ldots, 1,1)$, while the image of $R^\vee\otimes \mbb F_2$ is described as the kernel of the sum-of-coordinates map $\F_2^{\oplus 2m+1}\xrightarrow{\Sigma} \F_2$. Since $2m+1$ is odd we see that $(1,1,\ldots, 1,1)$ doesn't belong to $\ker(\Sigma)$, and so $\mf{so}_{4m+2}\simeq \F_2\oplus \mf{l}$. The adjoint action of $\mr{SO}_{4m+2}$ preserves both the center and the commutator and so respects this decomposition. It remains to show that $\mf{l}$ is irreducible. But, any $\mr{SO}_{4m+2}$-invariant subspace would in particular give a Lie ideal in $\mf{l}$, of which there aren't any by \cite[Table 1]{Hogeweij} again.
	\end{proof}
	\begin{rmk}
	    Essentially by the definition of roots, the highest weight for $\mf{l}$ is the longest root $\theta\in X_*(T)$ of $\mr{SO}_{4m+2}$. Since $\mf l$ is irreducible we have $\mf l=L(\theta)$. Recall that the highest weight of the dual $L(\lambda)^\vee$ is $-w_0\lambda$ where $w_0$ is the longest element of the Weyl group, and so $L(\lambda)^\vee\simeq L(-w_0\lambda)$. Since $w_0\theta=-\theta$ we get that $L(\theta)^\vee\simeq L(\theta)$ and so $\mf l$ is self-dual.
	\end{rmk}
	
	Note that by Totaro's computation $H^{1,2}_\Hdg(B\mr{SO}_{4m+2}/\mbb F_2)\simeq \mbb F_2\cdot u_3$ and that 
	$$
	H^{1,2}_\Hdg(B\mr{SO}_{4m+2}/\mbb F_2)\simeq H^1(\mr{SO}_{4m+2},\mf{so}_{4m+2}^\vee)\simeq H^1(\mr{SO}_{4m+2},\mf l^\vee)\oplus H^1(\mr{SO}_{4m+2},\mbb F_2).
	$$
	Since $H^1(\mr{SO}_{4m+2},\mbb F_2)\simeq 0$ (either by Kempf vanishing or comparing with $H^{0,1}_\Hdg(B\mr{SO}_{4m+2}/\mbb F_2)$) we get 
	$$H^1(\mr{SO}_{4m+2},\mf l)\simeq H^1(\mr{SO}_{4m+2},\mf l^\vee) \simeq \mbb F_2.$$
	
	Now, from \Cref{lem:so_4m+2 in char 2} and (\ref{eq:les for pso}) we get a short exact sequence of $\mr{SO}_{4m+2}$-modules:
	$$
	0\to \mf l \to \mf{pso}_{4m+2} \to \F_2 \to 0.
	$$
	\begin{lemma}
	    This extension is non-split. 
	\end{lemma}
	\begin{proof}
	    Due to $\mr{SO}_{4m+2}$-equivariance, the map $\mf{pso}_{4m+2} \to \F_2$ is in fact a map of Lie algebras and a splitting would also give a decomposition of the Lie algebra $\mf{pso}_{4m+2}$ as $\F_2\oplus \mf{l}$. However, the center $\mf{z}(\mf{pso}_{4m+2})$ is trivial (again, see \cite[Table 1]{Hogeweij}), and so this is impossible.
	\end{proof}
	As a consequence, we get that the $\mr{SO}_{4m+2}$-representation $\mf{pso}_{4m+2}$ is  the unique non-zero class in $\on{Ext}^1_{\mr{SO}_{4m+2}}(\mbb F_2,\mf l)\simeq H^1(\mr{SO}_{4m+2},\mf l)$. This class necessarily corresponds to the (unique non-zero) class $u_3\in H^{1,2}_\Hdg(B\mr{SO}_{4m+2}/\mbb \F_2)$.
	
	Passing to linear duals, we get a short (non-split) exact sequence
	$$
	0\to \mbb F_2 \to \mf{pso}_{4m+2}^\vee \to \mf l \to 0.
	$$
	Since $\mf l$ is self-dual, by the above discussion we also know that such a non-split extension is unique.
	
	\begin{rmk}
	    Using \cite[Table 1]{Hogeweij} one can see that the Lie algebra $\mf{spin}_{4m+2}$ is also a non-split extension of $\mf{l}$ by $\F_2$. By uniqueness, we get that the $\mr{SO}_{4m+2}$-representations $\mf{pso}_{4m+2}^\vee$ and $\mf{spin}_{4m+2}$ are isomorphic.
	\end{rmk}
	
	Let us now compute the cohomology of $\on{Sym}^*\mf{pso}^\vee_{4m+2}$. By \Cref{coaction-so} the $E_2$ page of the Eilenberg-Moore spectral sequence is given by \[PH^{*}_{\on{H}}(B\mr{SO}_{4m+2})\oplus z_2\F_2[z_2]\otimes \F_2[b_2,\dots,b_{2m+1}],\]
	with $PH^{*}_{\on{H}}(B\mr{SO}_{4m+2})$ being the 0-th column $E_2^{0,*}$ and $z_2\F_2[z_2]\otimes \F_2[b_2,\dots,b_{2m+1}]\simeq E_2^{>0,*}$ being the rest of $E_2$. The Hodge bigradings 
	here are given by $|z_2|_\Hdg=(1,1)$ and $|b_h|_\Hdg=(2h,2h)$. It follows that the non-pure part of the Hodge cohomology in fact lies in $E_2^{0,*}$ and embeds in the non-pure part of Hodge cohomology of $B\mr{SO}_{4m+2}$.

	On the level of representations we get the following: let $q^\vee \colon \on{Sym}^*\mf{pso}^\vee_{4m+2} \to \on{Sym}^*\mf{so}^\vee_{4m+2}$ be the natural map induced by $\mf{pso}^\vee_{4m+2}\to \mf{so}^\vee_{4m+2}$. Then $q^\vee$ induces an embedding
	$$
	H^{>0}(\mr{SO}_{4m+2}, \on{Sym}^*\mf{pso}^\vee_{4m+2}) \hookrightarrow H^{>0}(\mr{SO}_{4m+2}, \on{Sym}^*\mf{so}^\vee_{4m+2}).
	$$
	Moreover, the image can be described fairly explicitly. Namely, in the notations of \Cref{toda36-311-so}, $H^{>0}(\mr{SO}_{4m+2}, \on{Sym}^*\mf{pso}^\vee_{4m+2})$ can be identified with the ideal generated by non-pure elements $\cl{u}_{2k+1}$ and $y_{I}$ inside $PH^{*,*}_\Hdg(B\mr{SO}_{4m+2})$. It can be also seen as the intersection of the ideal $(u_3,u_5,\ldots, u_{4m+1})\subset H^{*,*}_\Hdg(\mr{SO}_{4m+2})$ (which is isomorphic to $H^{>0}(B\mr{SO}_{4m+2}, \on{Sym}^*(\mf{so}^\vee_{4m+2}))$ by Totaro's computation \cite[Theorem 11.1]{totaro2018hodge}) and the subalgebra $PH^{*,*}_\Hdg(B\mr{SO}_{4m+2})$ of primitive elements. This reduces the problem of computing $H^{i}(B\mr{SO}_{4m+2}, \on{Sym}^j(\mf{pso}^\vee_{4m+2}))$ to a much simpler linear algebra computation: namely one just needs to take the corresponding bigraded component in $(u_3,u_5,\ldots, u_{4m+1})\subset H^{*,*}_\Hdg(B\mr{SO}_{4m+2})$ and compute the intersection of kernels of $d_i$'s (from \Cref{condtr:d_i for so_n}) for all $i>0$.

	\subsection{Projective symplectic group}
	
	Assume now that $\cl{G}=\mr{PSp}_{4m+2}$. Similarly to the case of $\mr{PSO}_n$ we have an exact sequence of $\mr{PSp}_{4m+2}$-modules
	\begin{equation*}
	0\to \mbb F_2 \to \mf{sp}_{4m+2} \to \mf{psp}_{4m+2} \to \mbb F_2 \to 0.
	\end{equation*}
	which then gives an exact sequence 
	\begin{equation}\label{eq:les for psp}
	0\to \mbb F_2 \to  \mf{psp}_{4m+2}^\vee \to  \mf{sp}_{4m+2}^\vee\to \mbb F_2 \to 0
	\end{equation}
	by passing to duals.
	\begin{rmk}\label{rem:SP-case extension non-split}
In contrast to the $\mr{SO}_{4m+2}$-case, the map $\mbb F_2 \to \mf{sp}_{4m+2}$ is non-split. Indeed, coroots of $\mr{Sp}_{4m+2}$ span the coweight lattice $X^*(T)$, which forces the center $\mbb F_2$ to lie in the commutator $[\mf{sp}_{4m+2},\mf{sp}_{4m+2}]$. Also, the quotient $\mf{sp}_{4m+2}/\mbb F_2$ is not irreducible.	
	\end{rmk}

For brevity, we only sketch how to compute the second sheet of the Eilenberg-Moore spectral sequence in this case. Using \Cref{comparison-sp} we can explicitly identify the coaction of $H^*_\Hdg(B\mu_2/\F_2)$ on $H^*_\Hdg(B\mr{Sp}_{4m+2}/\F_2)$ with the coaction of  $H^*_{\mr{sing}}(B\mbb Z/2, \mbb F_2)$ on $H^*_{\mr{sing}}(B\mr{Sp}_{4m+2}(\mbb C), \mbb F_2)$. This identification then also induces an isomorphism of the primitive parts $PH^*_\Hdg(B\mr{Sp}_{4m+2}/\F_2)\simeq PH^*_{\mr{sing}}(B\mr{Sp}_{4m+2}(\mbb C), \mbb F_2)$. Moreover the computations of $\on{Cotor}$ using the twisted cochains (as in \Cref{compute-cotor}) are also compatible, which allows to identify the algebra structures on $$\on{Cotor}^*_{H^*_{\on{H}}(B\mu_2/\F_2)}(\F_2, H^*_{\on{H}}(B\mr{Sp}_{4m+2}/\F_2))\quad \text{and} \quad \on{Cotor}^*_{H^*_{\on{sing}}(B\mbb Z/2,\F_2)}(\F_2, H^*_{\on{sing}}(B\mr{Sp}_{4m+2},\F_2))$$ via \Cref{compatible} (and (\ref{algebra-map})).

This gives an identification 
$$
	E_2^{*,*}\simeq (\F_2[z_2,z_3]\otimes PH^{*}_{\on{H}}(B\mr{Sp}_{4m+2}/\F_2)) + \F_2[z_2,z_3,z_5,b_2,b_3,\dots,b_{2m+1}],
$$
where $b_h\in PH^{16h}_{\on{H}}(B\mr{Sp}_{4m+2}/\F_2)$ are certain elements defined similarly to \Cref{toda36-311} (or \cite[Lemma 3.10]{toda1987cohomology}) using the $*$-product of \Cref{con:star-product} for $a_\sharp\coloneqq q_2$. However, the twisted tensor product construction of (\ref{constr:twisted tensor product}) is naturally bigraded, which allows to compute the Hodge bidegrees (see \Cref{rem:Hodge bigrading}) for $E_2^{*,*}$. Namely, $|z_2|_\Hdg=(1,1)$, $|z_3|_\Hdg=(1,2)$, $|z_5|_\Hdg=(2,3)$ and $|b_h|_\Hdg=(8h,8h)$ (more generally, all elements in $PH^{*}_{\on{H}}(B\mr{Sp}_{4m+2}/\F_2)$ are pure).
\begin{rmk}
    There is another way to understand the bigradings of $z_i$, by computing the Hodge cohomology of $B\on{PSp}_{4m+2}$ in low degrees directly. Namely, 	Let $E$ be the Lagrangian Grassmannian, that is, $E\coloneqq \mr{PSp}_{4m+2}/P$ where $P\subset \on{PSp}_{4m+2}$ is the standard right-end root parabolic subgroup with Levi subgroup isomorphic to $\on{GL}_{2m+1}$; see \cite[Proposition 6.1]{pragacz1991algebro}. By \cite[Proposition 9.5]{totaro2018hodge} we have a spectral sequence
	\begin{equation}\label{symp-ss}E_2^{i,j}\coloneqq H^i_{\on{H}}(B\mr{PSp}_{4m+2}/\F_2)\otimes H^j_{\on{H}}(E/\F_2)\Rightarrow H^*_{\on{H}}(B\mr{GL}_{2m+1}/\F_2).\end{equation}
	The bigraded ring $H^{*,*}_{\on{H}}(E/\F_2)$ is well understood, as we now explain. By \cite[Proposition 7.1]{totaro2018hodge}, the cycle class map
	\[CH^*(E)\otimes_{\Z}\F_2\to H^*_{\on{H}}(E/\F_2)\]
	is an isomorphism. In particular, $H^{i,j}_{\on{H}}(E/\F_2)=0$ unless $i=j$. The Chow group $CH^*(E)$ is torsion-free and can be computed using the cell decomposition; see \cite[Corollary 6.3]{pragacz1991algebro}. 

	
	Using (\ref{symp-ss}) and the above description of $H^*_{\on{H}}(E/\F_2)$, low-degree computations analogous to those in \Cref{pgl-small} show that
	\begin{equation}\label{psp-low}H^i_{\on{H}}(B\textrm{PSp}_{4m+2}/\F_2)=
	\begin{cases}
		\F_2 &\text{if $i=0$}\\
		0 &\text{if $i=1$}\\
		\F_2\ang{x_2} & \text{if $i=2$}\\
		\F_2\ang{x_3} & \text{if $i=3$}\\
		\F_2\ang{x_2^2}& \text{if $i=4$}\\
		\F_2\ang{x_2x_3,x_5}& \text{if $i=5$}.
\end{cases}
\end{equation}
	Here $x_2$ has bidegree $(1,1)$, $x_3$ has bidegree $(1,2)$ and $x_5$ has bidegree $(2,3)$.
	
\end{rmk}	

Returning to the computation of dimensions of Hodge cohomology, we have that the $E_\infty$ page of the Eilenberg-Moore spectral sequence is isomorphic to
	\[(\F_2[z_2,z_3]\otimes PH^{*,*}_{\on{H}}(B\mr{Sp}_{4m+2}/\F_2)) \oplus  \F_2[z_2,z_3,z_5,b_2,b_3,\dots,b_{2m+1}].\]
	
	From the representation theoretic point of view, we get the following: there are two classes $x_3\in H^{1}(\mr{Sp}_{4m+2}, \mf{psp}^\vee_{4m+2})$, $x_5\in H^{1}(\mr{Sp}_{4m+2}, \on{Sym}^2 \mf{psp}^\vee_{4m+2})$, such that the higher cohomology 
	$
	H^{>0}(\mr{Sp}_{4m+2},\on{Sym}^*\mf{psp}_{4m+2}^\vee)
	$ is generated by the ideal 
	$$(x_3,x_5)\subset \F_2[x_3,x_5]\subset H^{>0}(\mr{Sp}_{4m+2},\on{Sym}^*\mf{psp}_{4m+2}^\vee)$$ as a module over the invariants $(\on{Sym}^*\mf{psp}_{4m+2}^\vee)^{\mr{Sp}_{4m+2}}$. Moreover,
	\begin{itemize}
	    \item $(\on{Sym}^*\mf{psp}_{4m+2}^\vee)^{\mr{Sp}_{4m+2}}\otimes x_3\cdot \F_2[x_3]$ embeds into $H^{>0}(\mr{Sp}_{4m+2},\on{Sym}^*\mf{psp}_{4m+2}^\vee)$ via the action map;
	    \item the cokernel of the above map can be described as $x_5\cdot \F_2[x_2,x_3,x_5,b_2,\ldots,b_{2m+1}]$, where $x_2\in (\mf{psp}_{4m+2}^\vee)^{\mr{Sp}_{4m+2}}$ and $b_h\in (\on{Sym}^{8h}\mf{psp}_{4m+2}^\vee)^{\mr{Sp}_{4m+2}}$ are fairly explicit invariant polynomials.\footnote{In particular, $x_2\in (\mf{psp}_{4m+2}^\vee)^{\mr{Sp}_{4m+2}}$ is exactly the image of 1 under the map $\F_2\to \mf{psp}^\vee_{4m+2}$ from \Cref{eq:les for psp}.}
	\end{itemize}
	
	\begin{rmk}
	    The extension class given by $$x_3\in \on{Ext}^1_{\mr{Sp}_{4m+2}}(\F_2,\mf{psp}^\vee_{4m+2})\simeq  H^{1}(\mr{Sp}_{4m+2}, \mf{psp}^\vee_{4m+2})\simeq  H^{1,2}_\Hdg(B\mr{PSp}_{4m+2}/\mbb F_2)$$ can be described explicitly. Indeed, the exact sequence (\ref{eq:les for psp}) gives a class in $\on{Ext}^2_{\mr{Sp}_{4m+2}}(\F_2,\F_2)$, which is necessarily 0, since $\on{Ext}^2_{\mr{Sp}_{4m+2}}(\F_2,\F_2)\simeq H^2(\mr{Sp}_{4m+2},\F_2)\simeq H^{2,0}_\Hdg(B\mr{Sp}_{4m+2}/\F_2)\simeq 0$. Thus, (\ref{eq:les for psp}) comes from some $\mr{Sp}_{4m+2}$-representation $V$ with a two-step filtration $0\subset V_1\subset V_2\subset V$ and such that $V_1\simeq \F_2$, $V_2\simeq \mf{psp}_{4m+2}^\vee$, $V/V_1\simeq \mf{sp}_{4m+2}$ and $V/V_2\simeq \F_2$. In particular, $V$ fits into a short exact sequence 
	    $$
	    0\to \mf{psp}^\vee_{4m+2}\to V\to \F_2\to 0,
	    $$
	    giving a class $[V]\in \on{Ext}^1_{\mr{Sp}_{4m+2}}(\F_2,\mf{psp}^\vee_{4m+2})$. Moreover, one sees from \Cref{rem:SP-case extension non-split} that this extension is non-split and thus $[V]\neq 0$, which forces it to be equal to $x_3$ since $$\on{Ext}^1_{\mr{Sp}_{4m+2}}(\F_2,\mf{psp}^\vee_{4m+2})\simeq H^{1,2}_\Hdg(B\mr{PSp}_{4m+2}/\mbb F_2) \simeq \F_2.$$
	\end{rmk}
	\appendix

	\section{K\"unneth formula for de Rham cohomology}\label{app:de rham cohomology}	
	
	In this section we give a proof of the K\"unneth formula for de Rham cohomology in the context of Artin stacks. The generality we consider is bigger than what is necessary for the applications in the body of the paper: this doesn't really affect the proof and might be useful for a future use. 
	
	Let $R$ be a base ring. We will work in the setting of higher Artin stacks (in the sense of \cite[Section 1.3.3]{TV_HAGII}, see also \cite[Appendix A.1]{kubrak2021p-adic}) these are sheaves of spaces in \'etale topology on the site $\Aff_R$ of affine $R$-schemes, that admit a smooth $(n-1)$-representable atlas for some
	$n \ne 0$ (the latter being an inductively defined notion, see loc. cit. for more details). 
	
	\begin{rmk}
		A classical stack $\mc X\colon \Aff_R \to \on{Grpd}$ can be considered as a higher stack via composing with the nerve functor $N\colon \on{Grpd} \to \on{Spcs}$. The image of this functor can be identified with the subcategory spanned by higher stacks that take values in 1-truncated spaces  $\on{Spcs}_{\le 1} \hookrightarrow \on{Spcs}$ (a space $X\in \on{Spcs}$ is called 1-truncated if $\pi_i(X,x)=0$ for $i>1$ and any base point $x\in X$). See e.g. \cite{hollander2008homotopy}.
	\end{rmk}
	
	\begin{constr}\label{con: Pridham's hypercovers}
		A useful fact (\cite[Theorem 4.7]{Pridham_ArtinHypercovers}) is that for any $n$-Artin stack $\mc X$ there exists an $(n-1)$-coskeletal smooth hypercover $X_\bullet \to X$, such that each $X_i$ is a (possibly infinite) union of affine schemes. If we assume that $\mc X$ is smooth itself and, moreover, is quasi-compact and quasi-separated, the schemes $X_i$ can be chosen to be smooth affine schemes. 
	\end{constr}
	
	\begin{example}\label{ex:hypercover in X/G}
		The classical quotient stack $\mc X=[X/G]$ with $X$ and $G$ being smooth affine schemes over $R$ is a smooth qcqs 1-Artin stack. In this case a hypercover as in \Cref{con: Pridham's hypercovers} can be taken to be the \v Cech nerve of the smooth cover $X\to [X/G]$. We have $X_\bullet\simeq X\times G^{\times \bullet}$ with the standard maps.
	\end{example}
	
	Given a smooth higher Artin stack $\mc X$, one can consider its (relative) de Rham cohomology (\cite[Defintion 1.1.3]{KubrakPrikhodko_HdR}), defined as the homotopy limit
	$$
	\RG_\dR(\mc X/R)\coloneqq \lim_{(S\ra \mc X)\in(\Aff^{\sm}_{/\mc X})^\op }\RG_\dR(S/R).
	$$
	This functor satisfies \'etale descent (which follows from the analogous statement for Hodge cohomology). Since smooth maps are \'etale surjections, for a smooth qcqs Artin stack $\mc X$ one gets a more economical formula in terms of a hypercover $|X_\bullet|\to \mc X$ as in \Cref{con: Pridham's hypercovers}:
	$$
	\RG_\dR(\mc X/R) \ism \Tot \RG_\dR(X_\bullet/R) \in \DMod{R}
	$$
	the totalization\footnote{Or, in other words, $\lim_{[\bullet]\in \Delta}$.} of the cosimplicial complex $\RG_\dR(X_\bullet/R)$. Here, $\RG_\dR(X_n/R)\in \DMod{R}$ is given by the usual de Rham complex $$
	\Omega^0_{X_n/R}\to \Omega^1_{X_n/R}\to\Omega^2_{X_n/R}\to\ldots, $$ and the totalization above can be computed by the means of the corresponding double-complex:
	$$
	\xymatrix{\vdots & \vdots& \vdots& \\
		\Omega^2_{X_0/R}\ar[u]\ar[r]& \Omega^2_{X_1/R}\ar[u]\ar[r]& \Omega^2_{X_2/R} \ar[u]\ar[r]&\ldots \\
		\Omega^1_{X_0/R}\ar[u]\ar[r]& \Omega^1_{X_1/R}\ar[u]\ar[r]& \Omega^1_{X_2/R} \ar[u]\ar[r]&\ldots \\
		\Omega^0_{X_0/R}\ar[u]\ar[r]& \Omega^0_{X_1/R}\ar[u]\ar[r]& \Omega^0_{X_2/R} \ar[u]\ar[r]&\ldots 
	}
	$$
	In particular, in the case of \Cref{ex:hypercover in X/G}, $\RG_\dR(\mc X/R)$ agrees with the definition given by Totaro in \cite{totaro2018hodge}. We also note that $\RG_\dR(\mc X/R)$ lies in\footnote{Here, $\DMod{R}^{\ge 0}\subset \DMod{R}$ is the full subcategory spanned by complexes $M$ with $H^{-i}(M)=0$ for any $i>0$.} $\DMod{R}^{\ge 0}$ for any (smooth) $\mc X$.
	
	We are now ready to prove the K\"unneth formula. We say that a ring $R$ is of finite Tor-dimension if there exists $k\ge 0$ such that for any two (classical) modules $M,N\in \Mod_R$ their derived tensor product $M\otimes^{\mbb L}_RN$ lies in cohomological degrees $\ge -k$.
	
	\begin{prop}[K\"unneth formula for de Rham cohomology]\label{prop: Kunneth in de Rham} Let $\mc X$, $\mc Y$ be smooth qcqs higher Artin stacks over a base ring $R$ that is of finite Tor-dimension. Then multiplication induces a natural equivalence
		$$
		\RG_\dR(\mc X/R)\otimes_{R}^{\mbb L} \RG_\dR(\mc Y/R) \ism \RG_\dR(\mc X\times_R \mc Y/R).
		$$
	\end{prop}
	\begin{proof}
		The proof is analogous to \cite[Proposition 2.2.15]{kubrak2021p-adic}, the main idea being to reduce to the case of affine schemes. Let $X_\bullet \to \mc X$ and $Y_\bullet \to \mc Y$ be hypercovers as in \Cref{con: Pridham's hypercovers}; this provides a hypercover $X_\bullet\times_R Y_\bullet \to X\times_R Y$ as well.  We have 
		$$
		\RG_\dR(\mc X/R)\otimes_{R}^{\mbb L} \RG_\dR(\mc Y/R)\ism \Tot\left(\RG_\dR(X_\bullet/R) \right)\otimes_{R}^{\mbb L} \RG_\dR(\mc Y/R).
		$$
		Under the Tor-finiteness assumption on $R$, the derived tensor product $-\otimes_{R}^{\mbb L} \RG_\dR(\mc Y/R)$ is left $t$-exact up to a shift, and thus by \cite[Corollary 3.1.13]{kubrak2021p-adic} we can move $-\otimes_{R}^{\mbb L} \RG_\dR(\mc Y/R)$ inside the totalization:
		$$
		\Tot\left(\RG_\dR(X_\bullet/R) \right)\otimes_{R}^{\mbb L} \RG_\dR(\mc Y/R) \ism \Tot\left(\RG_\dR(X_\bullet/R) \otimes_{R}^{\mbb L} \RG_\dR(\mc Y/R)\right).
		$$
		Now, using the hypercover $Y_\bullet \to \mc Y$ in the same way, we get
		$$
		\Tot\left(\RG_\dR(X_\bullet/R) \otimes_{R}^{\mbb L} \RG_\dR(\mc Y/R)\right)\ism \lim_{[\bullet_1,\bullet_2]\in \Delta\times \Delta} \RG_\dR(X_{\bullet_1}/R) \otimes_{R}^{\mbb L} \RG_\dR(Y_{\bullet_2}/R).
		$$
		Since $\Delta$ is sifted, the limit over $\Delta\times \Delta$ can be computed after restriction to the diagonal $\Delta\xrightarrow{\mr{diag}} \Delta\times \Delta$, and we get an equivalence $$\lim_{[\bullet_1,\bullet_2]\in \Delta\times \Delta} \RG_\dR(X_{\bullet_1}/R) \otimes_{R}^{\mbb L} \RG_\dR(Y_{\bullet_2}/R) \ism \lim_{[\bullet]\in \Delta} \RG_\dR(X_{\bullet}/R)\otimes_{R}^{\mbb L}\RG_\dR(Y_{\bullet}/R).$$
		
		We then have a commutative diagram 
		$$
		\xymatrix{\RG_\dR(\mc X/R)\otimes_{R}^{\mbb L}\RG_\dR(\mc Y/R) \ar[r]\ar[d]^\sim& \RG_\dR(\mc X\times_R \mc Y/R)\ar[d]^\sim\\
			\lim_{[\bullet]\in \Delta} \RG_\dR(X_{\bullet}/R)\otimes_{R}^{\mbb L}\RG_\dR(Y_{\bullet}/R) \ar[r]& \lim_{[\bullet]\in \Delta} \RG_\dR(X_{\bullet}\times_R Y_{\bullet}/R)
		}
		$$
		where the horizontal arrows are induced by multiplication, while the vertical ones are induced by pull-backs. The left vertical map is an equivalence by the above discussion, while the right one is an equivalence by descent.
		By \cite[Tag 0FMB]{stacks-project} the maps $$\RG_\dR(X_{\bullet}/R) \otimes_{R}^{\mbb L} \RG_\dR(Y_{\bullet}/R)\tto \RG_\dR(X_{\bullet}\times_R Y_{\bullet}/R)$$ are equivalences, thus so is the map between the limits. From the commutative diagram we then deduce the same for the upper horizontal map.
	\end{proof}
	
	\begin{cor}\label{cor: Kunneth in de Rham over a field}
		Let $R=k$ be a field and let $\mc X$, $\mc Y$ be smooth qcqs higher Artin stacks over $k$. Then multiplication induces a natural isomorphism of graded $k$-algebras
		$$
		H^*_\dR(\mc X/k)\otimes_k H^*_\dR(\mc Y/k) \ism H^*_\dR(\mc X\times_k \mc Y/k).
		$$
	\end{cor}
	\begin{proof}
		This follows from \Cref{prop: Kunneth in de Rham} by passing to cohomology.
	\end{proof}


\begin{thebibliography}{KP21b}

\bibitem[CR10]{chaput2010adjoint}
Pierre-Emmanuel Chaput and Matthieu Romagny.
\newblock On the adjoint quotient of {C}hevalley groups over arbitrary base
  schemes.
\newblock {\em J. Inst. Math. Jussieu}, 9(4):673--704, 2010.

\bibitem[Fri82]{friedlander1982etale}
Eric~M. Friedlander.
\newblock {\em \'{E}tale homotopy of simplicial schemes}, volume 104 of {\em
  Annals of Mathematics Studies}.
\newblock Princeton University Press, Princeton, N.J.; University of Tokyo
  Press, Tokyo, 1982.

\bibitem[Hog82]{Hogeweij}
G.M.D. Hogeweij.
\newblock Almost-classical {L}ie algebras. {I}.
\newblock {\em Indagationes Mathematicae (Proceedings)}, 85(4):441--452, 1982.

\bibitem[Hol08]{hollander2008homotopy}
Sharon Hollander.
\newblock A homotopy theory for stacks.
\newblock {\em Israel Journal of Mathematics}, 163(1):93--124, 2008.

\bibitem[Jan03]{jantzen2003representations}
Jens~Carsten Jantzen.
\newblock {\em Representations of algebraic groups}, volume 107 of {\em
  Mathematical Surveys and Monographs}.
\newblock American Mathematical Society, Providence, RI, second edition, 2003.

\bibitem[KP21a]{KubrakPrikhodko_HdR}
Dmitry Kubrak and Artem Prikhodko.
\newblock {Hodge-to-de Rham Degeneration for Stacks}, 05 2021.
\newblock {\em arxiv preprint arXiv:1910.12665}. Accepted by IMRN.

\bibitem[KP21b]{kubrak2021p-adic}
Dmitry Kubrak and Artem Prikhodko.
\newblock $p$-adic {H}odge theory for {A}rtin stacks.
\newblock {\em arXiv preprint arXiv:2105.05319}, 2021.

\bibitem[Pra91]{pragacz1991algebro}
Piotr Pragacz.
\newblock Algebro-geometric applications of {S}chur {$S$}- and
  {$Q$}-polynomials.
\newblock In {\em Topics in invariant theory ({P}aris, 1989/1990)}, volume 1478
  of {\em Lecture Notes in Math.}, pages 130--191. Springer, Berlin, 1991.

\bibitem[Pri15]{Pridham_ArtinHypercovers}
J.P. Pridham.
\newblock Presenting higher stacks as simplicial schemes.
\newblock {\em Advances in Mathematics}, 238:184--245, 2015.

\bibitem[Pri19]{primozic2019computations}
Eric Primozic.
\newblock Computations of de {R}ham cohomology rings of classifying stacks at
  torsion primes.
\newblock {\em arXiv preprint arXiv:1909.13413}, 2019.

\bibitem[Rav86]{ravenel1986complex}
Douglas~C. Ravenel.
\newblock {\em Complex cobordism and stable homotopy groups of spheres}, volume
  121 of {\em Pure and Applied Mathematics}.
\newblock Academic Press, Inc., Orlando, FL, 1986.

\bibitem[{Sta}]{stacks-project}
The {Stacks Project Authors}.
\newblock \textit{Stacks Project}.
\newblock http://stacks.math.columbia.edu.

\bibitem[Tod87]{toda1987cohomology}
Hiroshi Toda.
\newblock Cohomology of classifying spaces.
\newblock In {\em Homotopy theory and related topics ({K}yoto, 1984)}, volume~9
  of {\em Adv. Stud. Pure Math.}, pages 75--108. North-Holland, Amsterdam,
  1987.

\bibitem[Tot18]{totaro2018hodge}
Burt Totaro.
\newblock Hodge theory of classifying stacks.
\newblock {\em Duke Math. J.}, 167(8):1573--1621, 2018.

\bibitem[TV08]{TV_HAGII}
Bertrand To\"en and Gabriele Vezzosi.
\newblock {\em Homotopical Algebraic Geometry II: Geometric Stacks and
  Applications}.
\newblock Memoirs of the American Mathematical Society (v. II). Amer.
  Mathematical Society, 2008.

\end{thebibliography}

\end{document}